\documentclass[a4paper,11pt]{article}
\usepackage[utf8]{inputenc} 
\usepackage[T1]{fontenc} 
\usepackage[english]{babel} 
\usepackage{amsmath,amssymb,amsthm,mathrsfs,dsfont,mathtools,euscript,appendix} 
\usepackage{xcolor} 
\usepackage{authblk}  
\usepackage{ushort} 
\usepackage{url}
\usepackage{enumerate}
\usepackage[shortlabels]{enumitem} 
\usepackage{scalerel} 
\usepackage{relsize}

\usepackage[top=2.5cm, bottom=2.5cm, left=2.5cm, right=2.5cm]{geometry}
\usepackage{graphicx}
\usepackage[numbers]{natbib}

\usepackage[colorlinks=true]{hyperref}
\hypersetup{urlcolor=blue, linkcolor=red, citecolor=blue}
\numberwithin{equation}{section}

\newcommand{\N}{\mathbb{N}}

\newcommand{\iF}{\mathcal{F}}
\newcommand{\R}{\mathbb{R}}
\newcommand{\eP}{\mathbb{P}}
\newcommand{\iP}{\mathcal{P}}
\newcommand{\iW}{\mathcal{W}}
\newcommand{\iC}{\mathcal{C}}
\newcommand{\eC}{\mathscr{C}}
\newcommand{\iQ}{\mathcal{Z}}
\newcommand{\eL}{\mathbb{L}}
\newcommand{\iG}{\mathcal{G}}

\newcommand{\iB}{\mathcal{B}}
\newcommand{\iL}{\mathcal{L}}
\newcommand{\iJ}{\mathcal{J}}
\newcommand{\iE}{\mathcal{E}}
\newcommand{\E}{\mathbb{E}}
\newcommand{\iM}{\mathcal{M}}
\newcommand{\iN}{\mathcal{N}}
\newcommand{\D}{\mathbb{D}}
\newcommand{\etr}{(\iM_F(\R_+)\times\R,w\otimes|.|)}
\newcommand{\vag}{(\iM_F(\R_+)\times\R,v\otimes|.|)}
\newcommand{\ec}{\underline{c}}

\def\approxleq{%
  \def\p{%
    \setbox0=\vbox{\hbox{$<$}}%
    \ht0=0.9ex \box0 }%
  \def\s{%
    \vbox{\hbox{$\sim$}}%
  }%
  \mathrel{\raisebox{0.7ex}{%
      \mbox{$\underset{\s}{\p}$}%
    }}%
}

\newtheorem{theorem}{Theorem}[section]
\newtheorem{prop}{Proposition}[section]
\newtheorem{lemma}{Lemma}[section]

\newtheorem{rmq}{Remark}[section]

\newtheorem{assume}{Assumption}[section]

\title{Well posedness and stochastic derivation of a diffusion-growth-fragmentation equation in a chemostat}

\author{Josu\'e Tchouanti
\thanks{Email: \texttt{josue.tchouanti-fotso@unice.fr}}}

\affil{Neuromod Institute, Universit\'e C\^ote d'Azur\\ 2004 Route des Lucioles, 06902 Valbonne, France}

\begin{document}

    \maketitle
    
    \begin{abstract}
        We study the existence and uniqueness of the solution of a non-linear coupled system constituted of a degenerate diffusion-growth-fragmentation equation and a differential equation, resulting from the modeling of bacterial growth in a chemostat. This system is derived, in a large population approximation, from a stochastic individual-based model where each individual is characterized by a non-negative trait whose dynamics is described by a diffusion process. Two uniqueness results are highlighted. They differ in their hypotheses related to the influence of the resource on individual trait dynamics, the main difficulty being the non-linearity due to this dependence and the degeneracy of the diffusion coefficient. Further we show by probabilistic arguments that the semi-group of the stochastic trait dynamics admits a density. We deduce that the diffusion-growth-fragmentation equation admits a function solution with a certain Besov regularity.
        \medskip
        
        \noindent \textbf{keywords} {Diffusion-growth-fragmentation equation coupled with resource $\cdot$ Stochastic Feller-type diffusion $\cdot$ Mild formulation $\cdot$ Large population approximation $\cdot$ Existence of density.}
        \medskip
        
    
    \end{abstract}
    \section{Introduction}
        Developed by Novick and Szilard \cite{chem1}, and Monod \cite{chem2}, the chemostat is a laboratory device used by biologists to raise the microorganisms and study their interactions while at the same time regulating the population size and the experimental medium. It consists in a culture in a container of constant volume in which a substrate is continuously injected and extracted at the same rate. In the literature, it is extensively used to study the population dynamics, their interactions and the adaptative behavior of microorganisms \cite{jerom,sylCol,chamMelchem}. In this vein, we are interested in the dynamics of a structured asexual population (typically bacteria) coupled with a resource in a chemostat. Each individual is characterized by a quantitative trait $x\in\R_+$ which models a protein density and evolves according to a diffusion of which coefficients depend on the resource. The goal of this work is to show the existence and uniqueness of the solution of the large population dynamics described by the coupled system
        \begin{equation}\label{PDE}
            \left\{\begin{array}{l}
                \displaystyle \partial_tu_t = \partial_x^2\big(D(x,R_t)u_t\big) - \partial_x\big( \zeta(x,R_t)u_t \big) + \iG^{\dag}\big[ b(\cdot,R_t)u_t \big] - d(x)u_t\, , \textrm{ on } (0,T]\times\R^*_+ \vspace{0.20cm}
                \\
                \displaystyle \dot{R}_t = r-R_t - \int_0^{\infty}\chi(x,R_t)u_t(x)dx , \forall t\in(0,T] \vspace{0.20cm}
                \\
                \displaystyle \zeta(0,R_t)u_t(0) - \partial_x\big( D(x,R_t)u_t \big)_{|x=0} = 0 , \forall t\in(0,T]
            \end{array}\right.
        \end{equation}
        with a given initial condition $(u_0,R_0)$. The first two terms on the right side of the parabolic equation in this system are linked to the trait dynamics. The third one models cell division which is accompanied by a fragmentation of the trait according to the operator
        \[ \iG^{\dag}[f](x) = -f(x) + \int_0^1\frac{2}{\alpha} f\left( \frac{x}{\alpha} \right)M(d\alpha) \]
        where $M(d\alpha)$ is a symmetric probability measure on $]0,1[$, and the last term represents death. The second equation describes the resource supply and its dilution at the same rate $1$. The last term corresponds to its consumption by the population. The third equation is a no-flux boundary condition which ensures the quantitative variable to remain non negative.
        
        This model extends the well known growth-fragmentation model that has been extensively studied in the literature (see Bertoin et $al.$ \cite{gfrag1,gfrag2,gfrag3,gfrag4}, or Doumic et $al.$ \cite{gfrag5,gfrag6,gfrag7} who investigated the asymptotic properties of the equation without resource, or Campillo and Fritsch \cite{camfritsch} for its well posedness and stochastic derivation in a chemostat). Here, we also include the resource dynamics, and introduce a diffusion term that can be seen as an intrinsic noise resulting from very fast synthesis and degradations of the quantitative trait (see \cite{vinChi} for a similar framework). The main difficulty to get a uniqueness result lies in the degeneracy and the dependence in resource of this term. In order to deal with it, we will develop an approach based on the regularity of the individual trait dynamics Markovian semi-group. As it will also depend on the solution, it will be necessary to show some regularity according to this dependence and the state variable. Such an approach has been developed in \cite{joaMel} for a model without resource with Lipschitz continuous diffusion coefficient in the stochastic trait dynamics. Although this coefficient is less regular in our case, we show that it is possible to adapt their method when it does not depend on the resource dynamics. The general situation is more complicated because we are unable to obtain the Lipschitz dependence of the semi-group according to the solution, that holds in the previous case. We will then require more regularity and bounds on the coefficients in order to deal with this dependence thanks to the regularity of its generator.
        
        Further, another major point consists in showing that the weak solution of the diffusion-growth-fragmentation equation is a function, even if the initial condition is a measure. Classically, this point is directly linked to the stochastic dynamics of the individual trait. Then we will first show that at any time, the distribution of the trait admits a density with respect to the Lebesgue measure on $\R_+$, and extend the existence of the density to the solution thanks to a mild formulation of the equation. Again, the main difficulty lies in the degeneracy of the diffusion coefficient, but also its weaker regularity (H\"older continuous). Sato \& Ueno \cite[Theorem 2.1 and Appendix]{sato} showed the existence and regularity of the density in the uniformly elliptic case, and Bouleau \& Hirsch \cite[Theorem 2.1.3, p162]{boulHir} showed the existence of the density in the case of Lipschitz continuous coefficients. Here, we adapt a method inspired by Debussche \& Romito \cite{debromito} that has been developed by Romito \cite{romito}, which allows us to show that the density exists outside the degeneracy point $0$ for time dependent and H\"older continuous diffusion coefficients. It remains to justify that the law of the trait at any time does not charge this very point in order to conclude that the density exists on the entire non-negative real half-line.
        
        The paper is then organized as follows: in Section 2, we show that (\ref{PDE}) can be derived from an individual based model that describes each intrinsic dynamics and the resource. For that purpose, we propose a Markovian stochastic process that models the trait distribution within the population coupled to the resource dynamics, given a parameter that scales the initial number of individuals. We show that when this scale parameter goes to infinity, the sequence of laws of the processes admits some limiting values which satisfies a weak formulation of (\ref{PDE}). In Section 3, we show that the solution of this weak formulation is unique in a certain set of continuous measure valued processes and under suitable assumptions. We distinguish an easier case where the diffusion coefficient does not depend on the resource, and the general case where it does. As described here above, the method we use in those cases are based on the regularity of the semi-group of the stochastic process that describes the trait dynamics, according to the state variable and the way it depends on the solution. Finally, in Section 4, we show that this solution admits a density at any time with a Besov regularity. To summarize and including the resource dynamics, we prove that the only weak solution of (\ref{PDE}) belongs to $\eL^{\infty}((0,T], \eL^1(\R_+)\times\R)$.
        
        \paragraph{Notations} For later use, we introduce the following notations :
        \begin{itemize}
            \item $\iM_F(\R_+)$ is the space of finite measures on $\R_+$. It can be endowed with its vague (respectively weak) topology and denoted $(\iM_F(\R_+),v)$ (respectively $(\iM_F(\R_+),w)$), that is the weakest topology making all maps $\iM_F(\R_+)\ni\mu\mapsto \int_{\R_+} f(x)\mu(dx)$ continuous for any $f\in\iC_c(\R_+,\R)$ (respectively $f\in\iC_b(\R_+,\R)$).
            
            \item $\iM$ is the subspace of $\iM_F(\R_+)$ constituted of the finite sums of Dirac masses, thus
            \[ \iM = \left\{ \sum_{i=1}^n\delta_{x_i}: n\in\N, x_1,...,x_n\in\R_+ \right\} \]
            
            \item $\iP(\iE)$ denotes the space of probability measures on a given measurable space $\iE$.
            
            \item $\mathcal{LB}(\R_+,\R)$ denotes the space of functions from $\R_+$ to $\R$ that are bounded and Lipschitz continuous. It is endowed with the norm $\|\phi\|_{\mathcal{LB}} = \|\phi\|_{\textrm{Lip}} + \|\phi\|_{\infty}$
            where
            \[ \|\phi\|_{\textrm{Lip}} = \inf\left\{ c>0 : \frac{|\phi(x) - \phi(y)|}{|x-y|} \leq c, \forall x,y\in\R_+, x\neq y  \right\}. \]
            
            \item $\iJ$ denotes an abstract countable set.
            
            \item The constant $C > 0$ (or $C_p$) can change one line to an other.
        \end{itemize}

    \section{The individual-based model and stochastic derivation}
        In this section, we propose a stochastic individual-based model describing the population dynamics in the chemostat at the individual level with a carrying capacity $K>1$. Our goal is to derive a weak solution of (\ref{PDE}) from this stochastic model when $K$ goes to infinity.
        
        \subsection{Description}
            Let us first introduce a scaling parameter $K>1$ that represents the carrying capacity of the environment. We assume that it is roughtly proportional to the initial number of individuals. The bacterial population is described by a stochastic measure valued process $(\nu^K_t)_{t\geq 0}$ with values in $\frac{1}{K}\iM$ given at any time by the sum of the Dirac masses at the trait of each living cell so that each individual has the same mass $1/K$, and the resource is described by a positive continuous process $(R^K_t)_{t\geq 0}$ that corresponds to its mass concentration. Let us denote by $V^K_t\subset\iJ$ the collection of the living cells at any time $t\geq 0$. Each individual is characterized by a positive trait that evolves during its life-time, and may give birth or die depending on its trait and the availability of the resource as described as follows :
        
            Between jumps, the trait dynamics of an individual $i\in V^K_{t-}$ is described by a non negative diffusion process $X^{i,K}$ that satisfies the following stochastic differential equation 
            \begin{equation}\label{edsi}
                dX^{i,K}_t = \zeta(X^{i,K}_t,R^K_t)dt + \sqrt{2D(X^{i,K}_t,R^K_t)}\,dW^i_t
            \end{equation}
            during its life-time. The set $\iW = \{W^i, i\in\iJ\}$ is assumed to be a family of independent Brownian motions.
            
            \noindent Furthermore, an individual with trait $x\geq 0$ divides at rate $b(x,R^K_{t-})$ and its trait is split at the same time so that the mother cell stays with a random proportion $\alpha\in\, ]0,1[$ chosen according to a symmetric probability distribution $M(d\alpha)$, and the daughter starts its life with the remaining proportion $1-\alpha$. Such an individual dies at rate $d(x)$ either by extraction at rate $1$ or by natural death. It is then natural to consider that
            \begin{equation}\label{assd}
                d(x)\geq 1 , \forall x\geq 0
            \end{equation}
            but we do not need this hypothesis in the sequel.
            
            \noindent Finally, the mechanisms of the resource dynamics are well known for the chemostat and consist in a continuous supply of concentration $r_{in}>0$ and dilution of the resource at rate $1$, and its consumption by the micro-organisms during their life-time. Let us consider that an individual with trait $x\geq 0$ consumes a small amount $\chi(x,R^K_t)/K$ of the resource at time $t\geq 0$, then
            \begin{equation}\label{ress}
                R_t^K = R_0 + \int_0^t\bigg\{ r_{in}-R^K_s - \frac{1}{K}\sum_{i\in V^K_s}\chi(X^{i,K}_s,R^K_s)\bigg\}ds \, ,\, \forall t\geq 0.
            \end{equation}
            
            In order to describe the coupled dynamics of the population and the resource, we introduce a probability space $(\Omega,\iF,\eP)$ so that the initial condition $(\nu^K_0,R_0)$ is almost surely $\frac{1}{K}\iM\times\R_+$-valued for each $K>1$, and independent of the family $\iW$. Let us also introduce the independent Poisson point measures $\iN_1(d\alpha,dz,di,ds)$, $\iN_2(dz,di,ds)$ on $[0,1]\times\R_+\times\iJ\times\R_+$ and $\R_+\times\iJ\times\R_+$ with intensities $M(d\alpha)\eta(dz,di,ds)$ and $\eta(dz,di,ds)$ respectively, where
            \begin{equation}
                \eta(dz,di,ds) = dz\bigg( \sum_{j\in\iJ}\delta_j(di) \bigg)ds.
            \end{equation}
            Those Poisson point measures are assumed to be independent of the families $\iW$ and $\{\nu^K_0,K>1 \}$. We finally introduce the canonical filtration $(\iF_t)_{t\geq 0}$ generated by the sequence of initial conditions $\{\nu^K_0,K>1 \}$, the family of Brownian motions $\iW$ and the Poisson point measures $\iN_1, \iN_2$. This filtration is assumed to be complete and right continuous.

            We are interested in the dynamics of a coupled stochastic process composed of the $\frac{1}{K}\iM$-valued process defined by
            \begin{equation}
                \nu^K_t = \frac{1}{K}\sum_{i\in V^K_t}\delta_{X^{i,K}_t},\forall t\geq 0
            \end{equation}
            and the continuous real valued process $(R^K_t)_{t\geq 0}$. The trajectories of this coupled stochastic process can be represented by the following system 
            \begin{equation}\label{trajK}\left\{\begin{array}{l}
                \displaystyle \textrm{for all } f\in\iC^{2}_b(\R_+,\R), \vspace{0.20cm}
                \\
                \displaystyle \left\langle\nu^K_t,f\right\rangle = \left\langle\nu^K_0,f\right\rangle + \frac{1}{K}\int_0^t\sum_{i\in V^K_s}\sqrt{2 D(X^{i,K}_s,R^K_s)}\,f'(X^{i,K}_s) dW^i_s  \vspace{0.20cm}
                \\
                \displaystyle \qquad\quad +~ \int_0^t\int_0^{\infty} \left\{ \zeta(x,R^K_s) f'(x) + D(x,R^K_s)f''(x) \right\}\nu^K_s(dx)ds  \vspace{0.20cm}
                \\
                \displaystyle \qquad\quad +~\frac{1}{K}\iiiint_{[0,t]\times\iJ\times\R_+\times [0,1]} 1_{\left\{ i\in V_{s-}^K \right\}\cap\left\{ z\leq b(X^{i,K}_{s-},R^K_{s-}) \right\}} \vspace{0.25cm}
                \\
                \displaystyle \qquad\qquad\qquad\qquad \left[ -f(X^{i,K}_{s-}) + f(\alpha X^{i,K}_{s-}) + f((1-\alpha)X^{i,K}_{s-}) \right]\iN_{1}(d\alpha,dz,di,ds) \vspace{0.20cm}
                \\
                \displaystyle \qquad\quad -~\frac{1}{K}\iiint_{[0,t]\times\iJ\times\R_+} 1_{\left\{ i\in V^K_{s-} \right\}\cap\left\{ z\leq d(X^{i,K}_{s-}) \right\}}f(X^{i,K}_{s-})\iN_{2}(dz,di,ds), \vspace{0.25cm}
                \\
                \displaystyle R^K_t = R_0 + \int_0^t \bigg\{ r_{in}-R^K_s - \big\langle\nu^K_s,\chi(\cdot,R^K_s)\big\rangle \bigg\}ds.
            \end{array}\right.
            \end{equation}
            The first and the second integrals in the first equation correspond to the trait dynamics and follows from the It\^o formula applied on (\ref{edsi}), and the jumps correspond to births and deaths respectively. We refer to \cite{fourMel,chamMel} for the construction of such a process. In the sequel, we denote $\bar{R} = r_{in}\vee R_0$ and make the following assumptions.
            \begin{assume}\label{assf}
                    \item[\textsc{(A.1)}] The coefficients $\zeta(x,r),D(x,r),b(x,r),d(x),\chi(x,r)$ are Lipschitz continuous for $x\geq 0$ and $r\in[0,\bar{R}]$.
                        
                    \item[\textsc{(A.2)}] The coefficients $b(x,r), d(x), \chi(x,r)$ are non negative and bounded for $x\geq 0$ and $r\in [0,\bar{R}]$.
                    
                    \item[\textsc{(A.3)}] The coefficient $D(x,r)$ is non negative and
                    \[ \zeta(0,r)>0 \textrm{ and }D(0,r) = 0 , \forall r\in [0,\bar{R}]. \]
                    
                    \item[\textsc{(A.3)}] For any $x\geq 0$, $\chi(x,0) = 0$.
            \end{assume}
            \begin{rmq}
                The conditions $\zeta(0,r) > 0$ and $D(0,r) = 0$ ensure that the trait remains positive, since its diffusion coefficient vanishes and its drift is non-negative when it reaches $0$. Similarly, the condition $\chi(x,0) = 0$ is natural and ensures that the process $(R^K_t)_{t\geq 0}$ remains positive, since its drift is positive when it reaches $0$. Moreover, it follows from the comparison principle that
                \begin{equation}
                    R^K_t \leq r_{in} + (R_0-r_{in})e^{-t} \leq \bar{R} \textrm{ for all }t\geq 0.
                \end{equation}
            \end{rmq}
            \noindent For the sake of simplicity, we introduce the operators
            \begin{equation}
                L_rf(x) = \zeta(x,r)f'(x) + D(x,r)f''(x) , \forall x\geq 0
            \end{equation}
            given $r\in[0,\bar{R}]$, and 
            \begin{equation}\label{fragop}
                \iG[f](x) = -f(x) + 2\int_0^1f(\alpha x)M(d\alpha) , \forall x\geq 0,
            \end{equation}
            then we have the following result.
            \begin{lemma}\label{tailK}
                Under Assumption \ref{assf},
                \begin{enumerate}
                    \item if there exists $q\geq 1$ such that $\sup_K\E\left( \left\langle\nu^K_0,1\right\rangle^q \right) < \infty$, then
                    \begin{equation}\label{momK}
                        \sup_K\E\left\{ \sup_{0\leq t\leq T}\left[ \left\langle\nu^K_t,1\right\rangle^q \right] \right\} < \infty\textrm{ for all }T>0.
                    \end{equation}
                    
                    \item if there exists $q\geq 1$ such that $\sup_K\E\left( \left\langle\nu^K_0,1+x^q\right\rangle \right) < \infty$, then
                    \begin{equation}\label{momxK}
                        \sup_K\E\left\{ \sup_{0\leq t\leq T}\left[ \left\langle\nu^K_t,1+x^q\right\rangle \right] \right\} < \infty\textrm{ for all }T>0.
                    \end{equation}
                    
                    \item if $\E\left( \left\langle\nu^K_0,1\right\rangle \right)<\infty$, then the process $(\nu^K_t)_{t\geq 0}$ admits the Doob-Meyer decomposition
                    \begin{equation}\label{doobM}
                        \left\langle\nu^K_t,f\right\rangle = \left\langle\nu^K_0,f\right\rangle + V^{K,f}_t + M^{K,f}_t, \forall t\geq 0
                    \end{equation}
                    for any $f\in\iC^2_b(\R_+,\R)$ such that $xf'$ and $xf''$ are bounded, where $V^{K,f}$ is the process with finite variation defined by
                    \begin{equation}
                        V^{K,f}_t = \int_0^t\big\langle\nu^K_s, L_{R^K_s}f + b(\cdot,R^K_s)\iG [f] - d f \big\rangle ds
                    \end{equation}
                    and $M^{K,f}$ a square integrable martingale of quadratic variation
                    \begin{equation}\label{qvar}\begin{array}{l}
                        \displaystyle \left\langle M^{K,f}\right\rangle_t = \frac{2}{K}\int_0^t\int_0^{\infty}D(x,R_s^K)|f'(x)|^2\nu^K_s(dx)ds + \frac{1}{K}\int_0^t\int_0^{\infty}d(x)f^2(x)\nu^K_s(dx) ds \vspace{0.15cm}
                        \\
                        \displaystyle \qquad\qquad~ +~ \frac{1}{K}\int_0^t\int_0^{\infty}b(x,R^K_s)\int_0^1\big[ -f(x) + f(\alpha x) + f((1-\alpha)x)\big]^2M(d\alpha)\nu^K_s(dx)ds.
                    \end{array}
                    \end{equation} 
                \end{enumerate} 
            \end{lemma}
            \noindent The proof is relatively classical and similar to the one obtained in \cite{bjour,chamMel,joaMel,chiTran}. It is left to the reader.
            \begin{rmq}
                It follows from (\ref{qvar}) that the quadratic variation of the martingale part of the above Doob-Meyer decomposition satisfies the property
                \begin{equation}\label{diff}
                    \left\langle M^{K,f}\right\rangle_t \leq C_T\,\frac{\|f\|_{\infty}^2+\|xf'\|_{\infty}^2}{K}\sup_{0\leq u\leq T}\left\langle\nu^K_u,1\right\rangle, \forall t\in [0,T].
                \end{equation}
            \end{rmq}
            
        \subsection{Large population approximation and existence theorem}
            In this section, we are interested in a large population approximation of the above stochastic model. The main result consists in showing that under suitable assumptions, the sequence of laws of the coupled stochastic dynamics of the population and resource is tight in the set of probability distributions $\iP\left(\D([0,T],\etr\right)$, and the processes in its limiting values are the solutions of (\ref{PDE}) in the weak sense. Let us make the following assumptions.
            \begin{assume}\label{assK}
                \item[\textsc{(A.1)}] There exists $\varrho>0$ such that the initial conditions $\nu^K_0\in\frac{1}{K}\iM$ satisfy
                \[ \sup_K\E\left[ \left\langle\nu^K_0,1\right\rangle^{1+\varrho} \right] < +\infty. \]
                
                \item[\textsc{(A.2)}] The sequence $\left\{ \nu^K_0,K>1 \right\}$ converges in law in $(\iM_F(\R_+),w)$ towards a mesure valued initial condition $\nu_0\in\iM_F(\R_+)$.
            \end{assume} 
            \noindent The main result of this section follows.
            \begin{theorem}\label{exist}
                Under Assumptions \ref{assf} and \ref{assK}, the sequence of laws $\left\{\iQ^K = \iL(\nu^K,R^K),K>1 \right\}$ is tight in the space of probability measures $\iP(\D([0,T],\etr))$ for all $T>0$. In addition, each process $(\nu,R)$ in the support of its limiting values is continuous, satisfies the bound $\sup_{0\leq t\leq T}\left\langle\nu_t,1\right\rangle < \infty$ and the system
                \begin{equation}\label{determ}
                    \left\{\begin{array}{l}
                        \displaystyle \forall t\in [0,T], \forall f\in\iC^2_c(\R_+,\R), \vspace{0.20cm}
                        \\
                        \displaystyle \left\langle\nu_t,f\right\rangle = \left\langle\nu_0,f\right\rangle + \int_0^t \big\langle\nu_s, L_{R_s}f + b(\cdot,R_s)\iG[f] - df\big\rangle ds  \vspace{0.20cm}
                        \\
                        \displaystyle R_t = R_0 + \int_0^t\bigg\{ r_{in}-R_s - \big\langle\nu_s,\chi(\cdot,R_s)\big\rangle \bigg\}ds
                    \end{array}\right.
                \end{equation}
                that is a weak formulation of (\ref{PDE}) {with the initial condition $(\nu_0,R_0)$}.
            \end{theorem}
            \begin{rmq}
                Equation (\ref{PDE}) has therefore at least one weak solution. 
            \end{rmq}
            \noindent The proof of Theorem \ref{exist} is relatively classical and uses the techniques developed in \cite{ald,roel,chamMel,bjour,fourMel,joaMel} or \cite{sylVin}, Theorem 7.4. It consists in three main steps: a uniform control of trait distribution tails in the population, the tightness of the sequence of laws $\{\iQ^K,K>1\}$ and identification of the limit. It is detailed in Appendix \ref{appA}.
            
    \section{Uniqueness of the solution}
        In this section, we show that there exists a unique solution of the system (\ref{determ}). The main difficulties are the degeneracy of the diffusion coefficient and the non linearity of the drift and diffusion coefficients. Then we will first of all proceed by a decoupling of this system in order to obtain only one non-linear equation on the measure-valued process and establish some important properties of that decoupling. Secondly, we will write a mild formulation of this equation with continuous test functions thanks to the semigroup of the individual trait. Finally, we will use the properties of this semigroup to show in two different cases that the system (\ref{determ}) admits a unique solution.
            
        \subsection{Decoupling and mild formulation of the limit system}
            Let us start with the following results that ensure us that the limit system (\ref{determ}) can be decoupled so that it is equivalent to only one equation that describes the process $(\nu_t)_{t\in[0,T]}$ for the given initial condition $(\nu_0,R_0)$.
            \begin{prop}
                If $(\nu,R)\in\iC([0,T],\etr)$ is a solution of (\ref{determ}), then the process $(R_t)_{t\in[0,T]}$ is completely defined by the measure valued process $(\nu_t)_{t\in[0,T]}$ in the sense that if $(\nu_t)_{t\in[0,T]}$ is given, then $(R_t)_{t\in[0,T]}$ exists and is unique. More precisely, at any time $t\geq 0$ the quantity $R_t$ is characterized by the history of the measure valued process $\nu$ until the time $t$.
            \end{prop}
            \begin{proof}
                Given the process $(\nu_t)_{t\in[0,T]}$, the resource dynamics is described by the equation
                \begin{equation}\label{eqR} 
                    R_t = R_0 + \int_0^tF^{\nu}(s,R_s)ds\, ,\, \forall t\in [0,T]
                \end{equation}
                with
                \begin{equation}\label{fnu}
                    F^{\nu}(t,r) = r_{in}-r - \big\langle\nu_t,\chi(\cdot,r)\big\rangle , (t,r)\in [0,T]\times \R.
                \end{equation}
                Since $F^{\nu}(t,0) = r_{in}>0$, each solution of (\ref{eqR}) is non-negative. In addition, 
                \[ \chi(x,r) = |\chi(x,r) - \chi(x,0)| \leq r \|\chi\|_{\textrm{Lip}}. \]
                Then if we set
                \begin{equation}
                    \rho^{\nu}_T = \| \chi\|_{\textrm{Lip}}\sup_{0\leq t\leq T}\langle\nu_t,1\rangle,
                \end{equation}
                we obtain the inequality 
                \[ (1+\rho_T^{\nu})\left[ \frac{r_{in}}{1+\rho_T^{\nu}} - r \right] = r_{in}-r - \rho^{\nu}_Tr \leq F^{\nu}(t,r) \leq r_{in}-r. \]
                By using the comparison theorem for ordinary differential equations, that implies that each solution of (\ref{eqR}) satisfies the bound
                \[ \frac{r_{in}}{1+\rho_T^{\nu}} + \left(R_0 - \frac{r_{in}}{1+\rho_T^{\nu}} \right)e^{-(1+\rho_T^{\nu})t} \leq R_t \leq r_{in} + (R_0 - r_{in})e^{-t}. \]
                We deduce the bound property
                \begin{equation}\label{boundR}
                    R_0e^{-(1+\rho_T^{\nu})T} \leq R_t \leq r_{in}\vee R_0 \leq \bar{R}, \forall t\in [0,T].
                \end{equation}
                Moreover, for all $t\in [0,T]$ and $r_1,r_2\in [0,\bar{R}]$,
                \begin{displaymath}\begin{array}{l}
                    \displaystyle |F^{\nu}(t,r_1) - F^{\nu}(t,r_2)| = |r_2-r_1 + \int_0^{\infty}\left[ \chi(x,r_2) - \chi(x,r_1)  \right]\nu_t(dx)| \vspace{0.20cm}
                    \\
                    \displaystyle \qquad\qquad\qquad\qquad\quad~\, \leq |r_1-r_2| + \int_0^{\infty}|\chi(x,r_1)-\chi(x,r_2)|\nu_t(dx) \vspace{0.20cm}
                    \\
                    \displaystyle \qquad\qquad\qquad\qquad\quad~\, \leq \left(1+\rho^{\nu}_T\right) |r_1-r_2|
                \end{array}
                \end{displaymath}
                It then follows from the Cauchy-Lipschitz theorem that the equation (\ref{eqR}) admits a unique solution.
            \end{proof}
            The limit system (\ref{determ}) is then decoupled by introducing the new notation
            \begin{equation}\label{decoup}
                g[\nu](x,t) = g(x,R_t), \forall (x,t)\in \R_+\times [0,T]
            \end{equation}
            for each continuous function $g(x,r)$. This notation is consistent with the regularity according to the first variable. Indeed, if $g(x,r)$ is of class $\iC^k$ in its first variable with continuous derivatives, then the function $g[\nu](x,t)$ will be as well. Moreover, the function $F^{\nu}$ given in (\ref{fnu}) being continuous, the trajectory $t\rightarrow R_t$ is differentiable and then for each $g(x,r)$ of class $\iC^1$ in its second variable, the function $g[\nu](x,t)$ will be as well. We then introduce the following lemma whose proof is left to the reader.
            \begin{lemma}\label{regdec}
                Let $(x,r)\mapsto g(x,r)$ be a function of class $\iC^1$ on $\R_+\times [0,\bar{R}]$, then for each $M>0$, the local Lipschitz property holds
                \begin{equation}\label{ineqxs}
                    \big| g[\nu](x,s)-g[\nu](y,t) \big| \leq C_{M,T}\left( |x-y| + |t-s| \right) \, ,\, \forall s,t\leq T\, ,\, \forall x,y\in [0,M].
                \end{equation}
            \end{lemma}
            \noindent We can now rewrite equation (\ref{determ}) in the equivalent form
            \begin{equation}\label{maind}
                \left\{\begin{array}{l}
                    \displaystyle \forall t\in[0,T]\, ,\, \forall f\in\iC^{2}_c(\R_+,\R), \vspace{0.20cm}
                    \\
                    \displaystyle \left\langle\nu_t,f\right\rangle = \left\langle\nu_0,f\right\rangle +  \int_0^t\big\langle\nu_s, L^{\nu}_{s}f + b[\nu](\cdot,s)\iG[f] - d f\big\rangle ds
                \end{array}\right.
            \end{equation}
            with $L^{\nu}_s = L_{R_s}$ for any $s\in[0,T]$. In order to get the mild formulation of this system, let us consider for a given solution $(\nu_t)_{t\in[0,T]}$ of (\ref{maind}) the stochastic differential equation
            \begin{equation}\label{edsf}
                dX_t = \zeta[\nu](X_t,t)dt + \sqrt{2D[\nu](X_t,t)}\,dW_t
            \end{equation}
            that admits for every $x\geq 0$ and $s\in[0,T]$ a unique weak solution that starts at $x$ at time $s\geq 0$ (see Lemma \ref{appenBf} in Appendix \ref{appB}). This solution is almost surely non negative at any time and satisfies the control of moments
            \begin{equation}\label{moment}
                \forall p>0,\, \E_{x,s}\left\{ \sup_{s\leq u\leq T}(X_{u})^p \right\} \leq C_{p,T}(1+x^p)
            \end{equation}
            (see Lemma \ref{appenBf}). Let us denote by $(X^{\nu,x}_{s,t})_{t\in[s,T]}$ its trajectory and $(P^{\nu}_{s,t})_{s,t\geq 0}$ its semigroup defined by
            \begin{equation}
                P^{\nu}_{s,t}\phi(x) = \E\left[ \phi(X^{\nu,x}_{s,s+t}) \right],\,\forall \phi\in\iC_b(\R_+,\R).
            \end{equation}
            Its infinitesimal operator is $(L^{\nu}_s)_{s\in[0,T]}$ and we have the following result.
            \begin{lemma}\label{lemMild}
                Equation (\ref{maind}) admits the mild formulation
                \begin{equation}\label{mild}
                \left\{\begin{array}{l}
                    \displaystyle \forall t\geq 0,\forall \phi\in\iC_b(\R_+,\R) \vspace{0.20cm}
                    \\
                    \displaystyle \langle\nu_t,\phi\rangle = \langle\nu_0,P^{\nu}_{0,t}\phi\rangle + \int_0^t\big\langle\nu_s, b[\nu](\cdot,s)\iG[P^{\nu}_{s,t-s}\phi] - d P^{\nu}_{s,t-s}\phi \big\rangle ds. 
                \end{array}\right.
                \end{equation}
            \end{lemma}
            \begin{proof} 
                It classically follows from equation (\ref{maind}), that for each test function $f_s(x) = f(s,x)\in\iC^{1,2}_b([0,T]\times\R_+,\R)$ such that $\partial_sf_s + L^{\nu}_sf_s$ is bounded and $t\in [0,T]$,
                \begin{equation}\label{eqsN}
                    \left\langle\nu_t,f_t\right\rangle = \left\langle\nu_0,f_0\right\rangle +  \int_0^t\big\langle\nu_s, \partial_sf_s + L^{\nu}_sf_s + b[\nu](\cdot,s)\iG[f_s] - d f_s \big\rangle ds.
                \end{equation}
                In addition, for each test function $\phi\in\iC^2_b(\R_+,\R)$, the function
                \[ f_s(x) = P^{\nu}_{s,t-s}\phi(x), \forall (s,x)\in [0,t]\times\R_+ \]
                is a solution of the problem
                \begin{displaymath}\left\{\begin{array}{l}
                    \displaystyle \partial_sf_s(x) + L^{\nu}_sf_s(x) = 0, (s,x)\in[0,t]\times\R_+ \vspace{0.20cm}
                    \\
                    \displaystyle f_{\big| s=t} = \phi \,\textrm{ sur }\R_+.
                \end{array}\right.
                \end{displaymath}
                We then use this function in (\ref{eqsN}) and obtain (\ref{mild}) for $\phi\in\iC^{2}_b(\R_+,\R)$. It is easily extended by a density argument to continuous and bounded test functions.
            \end{proof}
            \noindent We also have the following property on the semi-group, which means that it is consistent with the linear space $\mathcal{LB}(\R_+,\R)$ :
            \begin{lemma}\label{lemContrz}
                The operator $P^{\nu}_{s,t}$ satisfies the uniform property
                \begin{equation}\label{ineqP}
                    \sup_{t+s\leq T}\left| P^{\nu}_{s,t}\phi(x)-P^{\nu}_{s,t}\phi(y) \right|\leq C\|\phi\|_{\textrm{Lip}}|x-y|, \forall \phi\in\mathcal{LB}(\R_+,\R).
                \end{equation}
            \end{lemma}
            \begin{proof}
                For all $s, t\geq 0$ such that $s+t\leq T$ and $x,y\geq 0$, we have
                \begin{displaymath}
                    \begin{array}{l}
                        \displaystyle \left| P^{\nu}_{s,t}\phi(x) - P^{\nu}_{s,t}\phi(y) \right| =  \left| \E\left[ \phi(X^{\nu,x}_{s,t+s}) - \phi(X^{\nu,y}_{s,t+s}) \right] \right| \vspace{0.20cm}
                        \\
                        \displaystyle \hspace{3.5cm} \leq \E\left( \left| \phi(X^{\nu,x}_{s,t+s}) - \phi(X^{\nu,y}_{s,t+s}) \right| \right) \vspace{0.20cm}
                        \\
                        \displaystyle \hspace{3.5cm} \leq \|\phi\|_{\textrm{Lip}}\E\left( \left| X^{\nu,x}_{s,t+s} - X^{\nu,y}_{s,t+s} \right| \right) \vspace{0.20cm}
                    \end{array}
                \end{displaymath}
                where $(X^{\nu,x}_{s,u})_{u\in[s,T]}$ and $(X^{\nu,y}_{s,u})_{u\in[s,T]}$ satisfy the stochastic differential equation (\ref{edsf}) with the same Brownian motion. It then follows from Lemma \ref{appenBf} in Appendix \ref{appB} that
                \[ \left| P^{\nu}_{s,t}\phi(x) - P^{\nu}_{s,t}\phi(y) \right| \leq C \|\phi\|_{\textrm{Lip}}|x-y|, \forall s+t\leq T.  \]
            \end{proof}
            
        \subsection{A simpler case of uniqueness}
            In this section, we show a uniqueness result when the diffusion coefficient does not depend on the resource dynamics. The method is adapted from that of \cite{joaMel} with a local H\"older continuous diffusion coefficient for the stochastic differential equation that describes the individual trait dynamics. Let us introduce the distance on $\iM_F(\R_+)$ defined by
            \begin{equation}
                |||\mu_1-\mu_2|||_{\mathcal{LB}} = \sup_{\phi\in\mathcal{LB}(\R_+,\R) ; \|\phi\|_{\mathcal{LB}}\leq 1}\langle\mu_1-\mu_2,\phi\rangle, \forall \mu_1,\mu_2\in\iM_F(\R_+).
            \end{equation}
            We start with the following results that highlight the consistence of the decoupling and the semigroup according to the solution considered.
            \begin{lemma}\label{lemLip}
                Let $(x,r)\mapsto g(x,r)$ be a Lipschitz continuous fonction from $\R_+\times [0,\bar{R}]$ to $\R$ and $\nu^1, \nu^2\in\iC([0,T],(\iM_F(\R_+),w))$ be two solutions of (\ref{maind}), then 
                \begin{equation}\label{ineqgf}
                    \sup_{(x,s)\in\R_+\times[0,t]}\left| g[\nu^1](x,s) - g[\nu^2](x,s) \right| \leq C\|g\|_{\textrm{Lip}}\int_0^t|||\nu^1_u-\nu^2_u|||_{\mathcal{LB}}\, du, \forall t\in [0,T].
                \end{equation}
            \end{lemma}
            \begin{proof}
                Let us denote by $R^{\nu^1}$ and $R^{\nu^2}$ the processes that are solutions of (\ref{eqR}) for $\nu^1$ and $\nu^2$ respectively, for the same initial condition $R_0$. Then thanks to the decoupling formula (\ref{decoup}) we have for all $t\in [0,T]$ and $(x,s)\in\R_+\times[0,t]$,
                \begin{equation}\label{ineqgs}
                    \left| g[\nu^1](x,s) - g[\nu^2](x,s) \right| = \big| g(x,R^{\nu^1}_s) - g(x,R^{\nu^2}_s) \big| \leq \|g\|_{\textrm{Lip}}\big| R^{\nu^1}_s - R^{\nu^2}_s \big|.
                \end{equation}
                In addition, it follows from (\ref{eqR}) that
                \begin{displaymath}
                    \begin{array}{l}
                        \displaystyle \big| R_s^{\nu^1} - R^{\nu^2}_s \big| = \left| \int_0^s\left\{ R^{\nu^2}_u-R^{\nu^1}_u - \int_0^{\infty}\chi(x,R^{\nu^1}_u)\nu^1_u(dx) + \int_0^{\infty}\chi(x,R^{\nu^2}_u)\nu^2_u(dx) \right\}du \right| \vspace{0.20cm}
                        \\
                        \displaystyle \hspace{2.1cm} \leq \int_0^s\bigg\{ \left|R^{\nu^2}_u-R^{\nu^1}_u\right| + \int_0^{\infty}\left|\chi(x,R^{\nu^2}_u)-\chi(x,R^{\nu^1}_u)\right|\nu^1_u(dx) \vspace{0.20cm} 
                        \\
                        \displaystyle \hspace{7cm} +~ \left|\int_0^{\infty}\chi(x,R^{\nu^2}_u)(\nu^2_u-\nu^1_u)(dx)\right| \bigg\}du \vspace{0.20cm}
                        \\
                        \displaystyle \hspace{2.1cm} \leq \left( 1 + \|\chi\|_{\textrm{Lip}}\sup_{0\leq z\leq T}\left\langle\nu^1_z,1\right\rangle \right)\int_0^s \left|R^1_u-R^2_u\right|du + (\|\chi\|_{\infty}+\|\chi\|_{\textrm{Lip}})\int_0^t|||\nu^1_u-\nu^2_u|||_{\mathcal{LB}}\,du
                    \end{array}
                \end{displaymath}
                and by the Gronwall lemma, for all $s\leq t$,
                \[ \big| R_s^{\nu^1} - R^{\nu^2}_s \big| \leq (\|\chi\|_{\infty}+\|\chi\|_{\textrm{Lip}}) \exp\left\{\left( 1 + \|\chi\|_{\textrm{Lip}}\sup_{0\leq z\leq T}\left\langle\nu^1_z,1\right\rangle \right)T\right\} \int_0^t|||\nu^1_u-\nu^2_u|||_{\mathcal{LB}}\,du . \]
                We introduce this inequality in (\ref{ineqgs}) to obtain (\ref{ineqgf}).
            \end{proof}
            \begin{lemma}\label{lemContr}
                For each test function $\phi\in\mathcal{LB}(\R_+,\R)$ and $\nu^1, \nu^2\in\iC([0,T],(\iM_F(\R_+),w))$ two solutions of (\ref{maind}), one has 
                \begin{equation}\label{ineqPn}
                    \sup_{s\leq t}\left\| \left(P^{\nu^1}_{s,t-s}-P^{\nu^2}_{s,t-s}\right)\phi \right\|_{\infty} \leq C\|\phi\|_{\mathrm{Lip}}\int_0^t|||\nu^1_u-\nu^2_u|||_{\mathcal{LB}}\,du.
                \end{equation}
            \end{lemma}
            \begin{proof}
                For all $0\leq s\leq t\geq T$ and $x\geq 0$, we have
                \begin{equation}\label{diffP}
                    \begin{array}{l}
                        \displaystyle \left| P^{\nu^1}_{s,t-s}\phi(x) - P^{\nu^2}_{s,t-s}\phi(x) \right| =  \left| \E\left[ \phi(X^{\nu^1,x}_{s,t}) - \phi(X^{\nu^2,x}_{s,t}) \right] \right| \vspace{0.20cm}
                        \\
                        \displaystyle \hspace{4.25cm} \leq \E\left( \left| \phi(X^{\nu^1,x}_{s,t}) - \phi(X^{\nu^2,x}_{s,t}) \right| \right) \vspace{0.20cm}
                        \\
                        \displaystyle \hspace{4.25cm} \leq \|\phi\|_{\textrm{Lip}}\E\left( \left| X^{\nu^1,x}_{s,t} - X^{\nu^2,x}_{s,t} \right| \right)
                    \end{array}
                \end{equation}
                where $(X^{\nu^1,x}_{s,u})_{u\in[s,T]}$ and $(X^{\nu^2,x}_{s,u})_{u\in[s,T]}$ satisfy (\ref{edsf}) with the same Brownian motion. In addition, using the inequality $|\sqrt{a} - \sqrt{b}|\leq \sqrt{|a-b|}$ for all $a,b\geq 0$, we have for all $s\leq t'\leq t$ and $\varepsilon>0$,
                \begin{displaymath}\begin{array}{l}
                    \displaystyle \int_{s}^{t'} 1_{\{0< X^{\nu^1,x}_{s,u}-X^{\nu^2,x}_{s,u}\leq \varepsilon\}}\frac{d\left\langle X^{\nu^1,x}_{s,\cdot}-X^{\nu^2,x}_{s,\cdot}\right\rangle_u}{X^{\nu^1,x}_{s,u}-X^{\nu^2,x}_{s,u}} = \int_{s}^{t'} 1_{\{0< X^{\nu^1,x}_{s,u}-X^{\nu^2,x}_{s,u}\leq \varepsilon\}}\frac{2\left|\sqrt{D(X^{\nu^1,x}_{s,u})} - \sqrt{D(X^{\nu^2,x}_{s,u})}\right|^2}{X^{\nu^1,x}_{s,u}-X^{\nu^2,x}_{s,u}}du \vspace{0.20cm}
                    \\
                    \displaystyle \hspace{7.1cm} \leq 2\int_{s}^{t'} 1_{\{0< X^{\nu^1,x}_{s,u}-X^{\nu^2,x}_{s,u}\leq \varepsilon\}}\frac{\left|D(X^{\nu^1,x}_{s,u}) - D(X^{\nu^2,x}_{s,u})\right|}{X^{\nu^1,x}_{s,u}-X^{\nu^2,x}_{s,u}}du \vspace{0.20cm}
                    \\
                    \displaystyle \hspace{7.1cm} \leq 2\|D\|_{\textrm{Lip}}\int_{s}^{t'} 1_{\{0< X^{\nu^1,x}_{s,u}-X^{\nu^2,x}_{s,u}\leq \varepsilon\}}du < +\infty.
                \end{array}
                \end{displaymath}
                We deduce with \cite{revYor}{, Ch. IX, lemma 3.3} that the process $X^{\nu^1,x}_{s,\cdot}-X^{\nu^2,x}_{s,\cdot}$ has a zero local time at $0$ and the Tanaka formula gives for all $s\leq t'\leq t$
                \[ \left| X^{\nu^1,x}_{s,t'} - X^{\nu^2,x}_{s,t'} \right| = \int_s^{t'}\textrm{sign}(X^{\nu^1,x}_{s,u} - X^{\nu^2,x}_{s,u})d(X^{\nu^1,x}_{s,u} - X^{\nu^2,x}_{s,u}). \]
                Thanks to the control of moments (\ref{momest}) in Lemma \ref{appenBf}, that implies
                \begin{displaymath}
                    \begin{array}{l}
                        \displaystyle \E\left( \left| X^{\nu^1,x}_{s,t'} - X^{\nu^2,x}_{s,t'} \right| \right) =  \E\left(\int_s^{t'}\textrm{sign}( X^{\nu^1,x}_{s,u} - X^{\nu^2,x}_{s,u})\left[ \zeta[\nu^1](X^{\nu^1,x}_{s,u},u) - \zeta[\nu^2](X^{\nu^2,x}_{s,u},u) \right]du \right) \vspace{0.20cm}
                        \\
                        \displaystyle \hspace{3.45cm} \leq \E\left(\int_s^{t'}\left| \zeta[\nu^1](X^{\nu^1,x}_{s,u},u) - \zeta[\nu^2](X^{\nu^2,x}_{s,u},u) \right|du \right) \vspace{0.20cm}
                        \\
                        \displaystyle \hspace{3.45cm} \leq \int_s^{t'}\E\left(\big| \zeta[\nu^1] - \zeta[\nu^2] \big|(X^{\nu^1,x}_{s,u},u) \right)du + \|\zeta[\nu^2]\|_{\textrm{Lip}} \int_s^{t'}\E\left(\left| X^{\nu^1,x}_{s,u} - X^{\nu^2,x}_{s,u} \right| \right)du.
                    \end{array}
                \end{displaymath}
                Using Lemma \ref{lemLip} and Gronwall's lemma, it follows that
                \[ \E\left( \left| X^{\nu^1,x}_{s,t'} - X^{\nu^2,x}_{s,t'} \right| \right) \leq C(t-s)e^{\|\zeta\|_{\textrm{Lip}}(t'-s)}\int_s^t |||\nu^1_u-\nu^2_u|||_{\mathcal{LB}}\,du. \]
                We inject this inequality for $t'=t$ in (\ref{diffP}) to obtain (\ref{ineqPn}).
            \end{proof}
            \noindent The main result for this part follows:
            \begin{theorem}\label{uniqone}
                Under Assumptions \ref{assf} and \ref{assK}, the system (\ref{determ}) admits a unique solution in $\iC([0,T],\etr)$ that satisfies $\sup_{0\leq t\leq T}\langle\nu_t,1\rangle < \infty$.
            \end{theorem}
            \begin{proof}
                The existence of the solution and the bound comes from Theorem \ref{exist} here above. It then suffices to show that the solution is unique. For that purpose, let us assume that there exist two solutions denoted $(\nu^1,R^1)$ and $(\nu^2,R^2)$ in $\iC([0,T],\etr)$.
                Then it follows from (\ref{mild}) that for all test function $\phi\in\mathcal{LB}(\R_+,\R)$,
                \begin{equation}\label{diffs}\begin{array}{l}
                    \displaystyle \left\langle\nu^1_t-\nu^2_t,\phi\right\rangle = \left\langle\nu_0,\left(P^{\nu^1}_{0,t}-P^{\nu^2}_{0,t}\right)\phi\right\rangle \vspace{0.20cm}
                    \\
                    \displaystyle \qquad +~ \int_0^t\int_0^{\infty}\left\{b[\nu^1](x,s)\iG[P^{\nu^1}_{s,t-s}\phi](x) - d(x) P^{\nu^1}_{s,t-s}\phi(x) \right\}(\nu^1_s-\nu^2_s)(dx)ds \vspace{0.20cm}
                    \\
                    \displaystyle \qquad +~ \int_0^t\int_0^{\infty}\left\{ b[\nu^1](x,s)\iG[\left(P^{\nu^1}_{s,t-s}-P^{\nu^2}_{s,t-s}\right)\phi](x) - d(x)\left(P^{\nu^1}_{s,t-s}-P^{\nu^2}_{s,t-s}\right)\phi(x) \right\}\nu^2_s(dx)ds \vspace{0.20cm}
                    \\
                    \displaystyle \qquad +~ \int_0^t\int_0^{\infty}\iG[P^{\nu^2}_{s,t-s}\phi](x)\bigg\{ b[\nu^1](x,s) - b[\nu^2](x,s) \bigg\}\nu^2_s(dx)ds.
                \end{array}
                \end{equation}
                Furthermore, the operator $\iG$ is consistent with the linear space $\mathcal{LB}(\R_+,\R)$. Indeed, for any test function $f\in\mathcal{LB}(\R_+,\R)$ we have
                \[ |\iG[f](x)| \leq |f(x)| + 2\int_0^1|f(\alpha x)|M(d\alpha) \leq 3\|f\|_{\infty}, \forall x\geq 0 \]
                and for all $x,y\geq 0$,
                \begin{displaymath}
                    \begin{array}{l}
                        \displaystyle \big|\iG[f](x) - \iG[f](y)\big| = \left| f(x) - f(y) + 2\int_0^1\left[ f(\alpha x)-f(\alpha y) \right]M(d\alpha) \right| \vspace{0.20cm}
                        \\
                        \displaystyle \hspace{2.95cm} \leq \left| f(x) - f(y)\right| + 2\int_0^1\left| f(\alpha x)-f(\alpha y)\right| M(d\alpha) \vspace{0.20cm}
                        \\
                        \displaystyle \hspace{2.95cm} \leq 2\|f\|_{\textrm{Lip}}|x-y|.
                    \end{array}
                \end{displaymath}
                Hence, we can deduce from Lemma \ref{lemContrz} that for each $s\leq t$ the function
                \[ \R_+\ni x\mapsto b[\nu^1](x,s)\iG[P^{\nu^1}_{s,t-s}\phi](x) - d(x)P^{\nu^1}_{s,t-s}\phi(x) \]
                is in $\mathcal{LB}(\R_+,\R)$ so that
                \begin{displaymath}
                    \begin{array}{l}
                        \displaystyle \left\| b[\nu^1](\cdot,s)\iG[P^{\nu^1}_{s,t-s}\phi] - d(\cdot)P^{\nu^1}_{s,t-s}\phi \right\|_{\mathcal{LB}} = \left\| b[\nu^1](\cdot,s)\iG[P^{\nu^1}_{s,t-s}\phi] - d(\cdot)P^{\nu^1}_{s,t-s}\phi \right\|_{\infty} \vspace{0.20cm}
                        \\
                        \displaystyle \hspace{7cm} +~ \left\| b[\nu^1](\cdot,s)\iG[P^{\nu^1}_{s,t-s}\phi] - d(\cdot)P^{\nu^1}_{s,t-s}\phi \right\|_{\textrm{Lip}} \vspace{0.20cm}
                        \\
                        \displaystyle \hspace{6.4cm} \leq (3\|b\|_{\infty} + \|d\|_{\infty})\|\phi\|_{\infty} + 2\|b\|_{\infty}\|P^{\nu^1}_{s,t-s}\phi\|_{\textrm{Lip}} + 3\|\phi\|_{\infty}\|b\|_{\textrm{Lip}} \vspace{0.20cm}
                        \\
                        \displaystyle \hspace{6.4cm} \leq (3\|b\|_{\infty} + \|d\|_{\infty} + 3\|b\|_{\textrm{Lip}})\|\phi\|_{\infty} + 2\|b\|_{\infty}C\|\phi\|_{\textrm{Lip}}
                    \end{array}
                \end{displaymath}
                and then
                \[ \left\| b[\nu^1](\cdot,s)\iG[P^{\nu^1}_{s,t-s}\phi] - d(\cdot)P^{\nu^1}_{s,t-s}\phi \right\|_{\mathcal{LB}} \leq C \|\phi\|_{\mathcal{LB}}, \forall s\leq t. \]
                With equation (\ref{diffs}) and Lemmas \ref{lemContr}, \ref{lemLip}, that implies that
                \begin{displaymath}
                    \left\langle\nu^1_t-\nu^2_t,\phi\right\rangle \leq C \|\phi\|_{\mathcal{LB}}\left(1 + \sup_{0\leq s\leq T}\langle\nu^2_u,1\rangle\right) \int_0^t |||\nu^1_s-\nu^2_s|||_{\mathcal{LB}}\,ds.
                \end{displaymath}
                It follows that for all $t\in [0,T]$
                \begin{displaymath}
                    |||\nu^1_t-\nu^2_t|||_{\mathcal{LB}} = \sup_{\phi\in\mathcal{LB}(\R_+,\R); \|\phi\|_{\mathcal{LB}}\leq 1}\left\langle\nu^1_t-\nu^2_t,\phi\right\rangle \leq C \left(1 + \sup_{0\leq s\leq T}\langle\nu^2_u,1\rangle\right) \int_0^t |||\nu^1_s-\nu^2_s|||_{\mathcal{LB}}\,ds
                \end{displaymath}
                and by Gronwall's lemma we deduce that $\nu^1_t = \nu^2_t$ for all $t\in[0,T]$. We conclude that $(\nu^1,R^1) = (\nu^2,R^2)$.
            \end{proof}
            We saw that for this method, the control on the difference $P^{\nu^1}_{s,t-s} - P^{\nu^2}_{s,t-s}$ given in Lemma \ref{lemContr} is very important. However, in the general situation where the diffusion coefficient of the individual trait dynamics depends on the resource we are not able to get such an inequality. In the following section, we suggest another method for this general case that requires more regularity on the coefficients.
     
        \subsection{The general case of uniqueness}
            In this section, we show a more complicated uniqueness result for (\ref{determ}) in the general case where the diffusion coefficient depends on the resource dynamics. As in the previous section, the method consists in building a solution of an underlying parabolic partial differential equation in order to remove the derivatives of the test function. We need this function to be regular enough. Let us then introduce the new distance on $\iM_F(\R_+)$ defined by
            \begin{equation}
                |||\mu^1 - \mu^2||| = \sup_{\phi\in\iC^2_b(\R_+,\R);\|\phi\|\leq 1}\left\langle\mu^1-\mu^2,\phi\right\rangle, \forall \mu^1,\mu^2\in\iM_F(\R_+)
            \end{equation}
            where
            \begin{equation}
                \|\phi\| = \|\phi\|_{\infty} + \|\phi'\|_{\infty} + \|\phi''\|_{\infty}, \forall \phi\in\iC^2_b(\R_+,\R).
            \end{equation}
            We consider the additional assumptions.
            \begin{assume}\label{assder}
                    \item[\quad\textsc{(A.1)}] $\zeta\in\iC^2(\R_+\times[0,\bar{R}],\R)$ is such that $\|\partial_x\zeta\|_{\infty} + \|\partial_r\zeta\|_{\infty} + \|\partial^2_x\zeta\|_{\infty} < +\infty$,
                    
                    \item[\quad\textsc{(A.2)}] $D\in\iC^2(\R_+\times[0,\bar{R}],\R)$ is such that $\|\partial_xD\|_{\infty} + \|\partial_rD\|_{\infty} + \|\partial^2_xD\|_{\infty} < +\infty$,
                    
                    \item[\quad\textsc{(A.3)}] $b\in\iC^{2,1}(\R_+\times[0,\bar{R}],\R)$ is such that $\|\partial_xb\|_{\infty} + \|\partial_rb\|_{\infty} + \|\partial^2_xb\|_{\infty} < +\infty$,
                    
                    \item[\quad\textsc{(A.4)}] $\chi\in\iC^{2,1}(\R_+\times[0,\bar{R}],\R)$ is such that $\|\partial_x\chi\|_{\infty} + \|\partial_r\chi\|_{\infty} + \|\partial^2_x\chi\|_{\infty} < +\infty$,
                    
                    \item[\quad\textsc{(A.5)}] $d\in\iC^2(\R_+,\R)$ is such that $\|d'\|_{\infty} + \|d''\|_{\infty}  < +\infty$,
                    
                    \item[\quad\textsc{(A.6)}] $D(x,r) >0$ for all $(x,r)\in(0,\infty)\times(0,\bar{R}]$.
            \end{assume}
            \noindent Let us start with an important result that states the regularity and some bounds of the semigroup of the stochastic individual trait dynamics.
            \begin{theorem}\label{theoPDE}
                Under \textsl{(A.1), (A.2), (A.6)} in Assumption \ref{assder}, for given $t\in(0,T]$ and a test function $\phi\in\iC^2_b(\R_+,\R)$, the parabolic problem
                \begin{equation}\label{EDP}
                    \left\{\begin{array}{l}
                        \displaystyle \partial_sf_s(x) + \zeta[\nu](x,s)\partial_xf_s(x) + D[\nu](x,s)\partial_x^2f_s(x) = 0, x>0, s\in [0,t) \vspace{0.20cm}
                        \\
                        \displaystyle f_{\big|s=t} = \phi \textrm{ on }\R_+
                    \end{array}\right.
                \end{equation}
                admits a unique solution in the subspace of $\iC([0,t]\times\R_+,\R)\cap\iC^{1,2}([0,t)\times(0,\infty),\R)$ constituted of functions $f$ that satisfy
                \begin{equation}\label{contrfg}
                    \exists C,p >0, |\partial_xf_s(x)| + |\partial_x^2f_s(x)| \leq C\big(1+x^{p}\big), \forall s< t, \forall x>0.
                \end{equation}
                In addition, this solution belongs to $\iC^{1,1}([0,t)\times\R_+,\R)$, is given by
                \begin{equation}\label{exprf}
                    f^{\phi}_s(x) = \E\left[ \phi(X^{\nu,x}_{s,t}) \right], \forall (s,x)\in [0,t]\times\R_+
                \end{equation}
                and satisfies the property
                \begin{equation}\label{contrf}
                    |f^{\phi}_s(x)| + |\partial_xf^{\phi}_s(x)| + |\partial_x^2f^{\phi}_s(x)| \leq C_T\|\phi\|, \forall s<t, \forall x>0.
                \end{equation}
            \end{theorem}
            
            \begin{proof}
                We split this proof into the following steps :
                \paragraph{Step 1: Uniqueness and characterization of the solution.} Let us consider that the system (\ref{EDP}) admits a solution $(s,x)\mapsto f_s(x)$ in $\iC([0,t]\times\R_+,\R)\cap\iC^{1,2}([0,t)\times(0,\infty),\R)$ that satisfies the polynomial bound (\ref{contrfg}). Since this solution is not necessarily differentiable on the straight line $x=0$, we consider the regularising sequence defined for each integer $n\geq 1$ by
                \[ \rho_n(x) = n\rho(nx),\forall x\geq 0 \]
                with $\rho$ such that
                \[ \rho\in\iC^{\infty}_c(\R_+,\R), \rho\geq 0\, ,\, \textrm{Supp}(\rho)\subseteq [ 1,2 ]\, \textrm{ and }\, \int_0^{\infty}\rho(y)dy = 1 , \]
                and we introduce the sequence of functions defined by
                \begin{equation}\label{regf}
                    f^{n}_s(x) = \int_0^{\infty}f_s(x+y)\rho_n(y)dy\, ,\, \forall s\in[0,t],\forall x\in\R_+.
                \end{equation}
                Then for each $n\geq 1$ we have $f^{n}\in\iC([0,t]\times\R_+,\R)\cap\iC^{1,\infty}([0,t)\times\R_+,\R)$ that converges pointwise towards $f$ on $[0,t]\times\R_+$ as $n\rightarrow\infty$. In addition, it satisfies on $[0,t)\times(0,\infty)$
                \begin{displaymath}
                    \begin{array}{l}
                        \displaystyle \left| \partial_sf^n_s(x) + L^{\nu}_sf^n_s(x) \right| = \left| \int_0^{\infty}\bigg[ \partial_sf_s(x + y) + L^{\nu}_s\big(f_s(\cdot+y)\big)(x) \bigg]\rho_n(y)dy \right| \vspace{0.20cm}
                        \\
                        \displaystyle \hspace{3.45cm} =\, \bigg| \int_0^{\infty}\bigg[ \big( \zeta[\nu](x,s) - \zeta[\nu](x+y,s) \big)\partial_xf(x+y) \vspace{0.20cm} 
                        \\
                        \displaystyle \hspace{5.5cm} +\, \big( D[\nu](x,s) - D[\nu](x+y,s) \big)\partial_x^2f(x+y) \bigg]\rho_n(y)dy \bigg| \vspace{0.20cm}
                        \\
                        \displaystyle \hspace{3.45cm} \leq \, C\int_0^{\infty} y\bigg( \big|\partial_xf_s(x+y)\big| + \big| \partial_x^2f_s(x+y) \big| \bigg)\rho_n(y)dy
                    \end{array}
                \end{displaymath}
                and then,
                \begin{equation}\label{contrfnL}
                    \left| \partial_sf^n_s(x) + L^{\nu}_sf^n_s(x) \right| \leq \frac{C}{n}\int_0^{\infty}y\bigg( \big|\partial_xf_s(x+\frac{y}{n})\big| + \big| \partial_x^2f_s(x+\frac{y}{n}) \big| \bigg)\rho(y)dy.
                \end{equation}
                In particular, it follows from the bound (\ref{contrfg}) that
                \begin{equation}\label{contrfgn}
                    \left|\partial_{s}f^{n}_s(x) + L^{\nu}_sf^{n}_s(x)\right| \leq \frac{C}{n} (1+x^{p})\, , \, \forall s< t\, ,\, \forall x\geq 0.
                \end{equation}
                Further, the It\^o formula gives with (\ref{moment})
                \[ \E\left[ f^{n}_t(X^x_{s,t}) \right] = f^{n}_s(x) + \E_{x,s}\int_s^t\bigg[ \partial_{u}f^{n}_u(X^x_{s,u}) + L^{\nu}_sf^{n}_u(X^x_{s,u}) \bigg]du \]
                that implies with (\ref{contrfnL}) that
                \[ \left| \E\left[ f^{n}_t(X^x_{s,t}) \right] - f^{n}_s(x) \right| \leq \frac{C}{n}\int_s^t\left( 1 + \E\left[ (X^x_{s,u})^p \right] \right)du. \]
                The Feynman-Kac characterization (\ref{exprf}) follows as $n\rightarrow\infty$. Such a solution $f$ is then unique and it only suffices now to show that it exists.
                
                \paragraph{Step 2: Construction of a regular solution.} A bound property for the second derivative in $x$ of a solution is difficult to obtain directly from the characterization in equation (\ref{exprf}) because the diffusion coefficient of the stochastic differential equation (\ref{edsf}) is not differentiable. Here we suggest another method. We start with the equation that would be satisfied by the first derivative in $x$ of the solution we are aiming to construct. Let us then introduce the new system
                \begin{equation}\label{edpd}\left\{\begin{array}{l}
                    \displaystyle \partial_sg_s(x) + h(x,s)\partial_xg_s(x) + D[\nu](x,s)\partial_x^2g_s(x) = -\partial_x\zeta[\nu](x,s)g_s(x),\,x> 0,s\in [0,t ) \vspace{0.20cm}
                    \\
                    \displaystyle g_{\big| s=t} = \phi' \,\textrm{ on }\R_+
                \end{array}\right.
                \end{equation}
                where
                \begin{equation}
                    h(x,s) = \zeta[\nu](x,s) + \partial_xD[\nu](x,s)\, ,\, \forall (s,x)\in[0,t]\times\R_+
                \end{equation}
                is differentiable, Lipschitz continuous in $x$ uniformly in $s$ thanks to \textsl{(A.1), (A.2)} in Assumption \ref{assder}, and satisfies $h(0,s) > 0$ for all $s\in[0,t]$. We also introduce the function defined by
                \begin{equation}\label{exprg}
                    g^{\phi}_s(x) = \E_{x,s}\left[ \phi'(Y_{t})e^{\int_s^t\partial_x\zeta[\nu](Y_{u},u)du} \right], \forall x\geq 0,\forall s\leq t
                \end{equation}
                where 
                \begin{equation}\label{edsd}
                    dY_u = h(Y_{u},u)du + \sqrt{2D[\nu]\left(Y_{u},u\right)}\, dW_u\, ,\, \forall u\leq t.
                \end{equation}
                Equation (\ref{edsd}) admits a unique weak solution for any deterministic non-negative initial condition and initial time (see Lemma \ref{appenBf} in Appendix \ref{appB}). It is shown in the same appendix that there is strong existence, and then we can use the right continuous and completed canonical filtration of the given Brownian motion $W$ that we will denote $(\iF^{W}_t)_{t\in [0,T]}$. The notation $\E_{x,s}$ refers to the expectation conditionned by the initial time $s<t$ and the initial condition $Y_s=x\geq 0$. The function $(x,s)\mapsto g^{\phi}_s(x)$ is well defined because $\phi'$ and $\partial_x\zeta$ are bounded thanks to (A.1) in Assumption \ref{assder}. In addition, it is continuous on $[0,t]\times\R_+$ and satisfies for all $x,y\geq 0$ and $s_1,s_2\leq t$,
                \begin{equation}\label{contrgp}
                    \left| g^{\phi}_{s_1}(x)-g^{\phi}_{s_2}(y) \right| \leq  C_T \left(\|\phi'\|_{\infty} + \|\phi''\|_{\infty}\right)\left\{ |x-y| + (1+x)^{1/2}|s_1-s_2|^{1/2} + (1+x)|s_1-s_2| \right\}
                \end{equation}
                (see Lemma \ref{appenBs} in Appendix \ref{appB}). We construct a classical solution for (\ref{edpd}) by restricting the problem on increasingly greater intervals over which it is uniformly parabolic. Let us then introduce two monotonous sequences $(\varepsilon_n)_n,(m_n)_n$ of positive real numbers that converge towards $0$ and $\infty$ respectively, and satisfy $\varepsilon_0 < m_0$. It follows from the property (\ref{boundR}) and \textsl{(A.6)} in Assumption \ref{assder} that
                \[ \forall n\in\N\, ,\,\exists a_n>0\, ,\, D[\nu](x,s)\geq a_n\, ,\, \forall (x,s)\in [\varepsilon_n,m_n]\times [0,t]. \]
                For each $n$, we consider the restricted system
                \begin{equation}\label{edpr}\left\{\begin{array}{l}
                    \displaystyle \partial_sg_s(x) + h(x,s)\partial_xg_s(x) + D[\nu](x,s)\partial_x^2g_s(x) = -\partial_x\zeta[\nu](x,s)g_s(x),\,(s,x)\in[0,t)\times(\varepsilon_n,m_n) \vspace{0.20cm}
                    \\
                    \displaystyle g_s(x) = g^{\phi}_s(x)\, ,\, (s,x)\in [0,t)\times\{\varepsilon_n,m_n\} \vspace{0.20cm}
                    \\
                    \displaystyle g_{\big| s=t} = \phi' \,\textrm{ on }[\varepsilon_n,m_n],
                \end{array}\right.
                \end{equation}
                that has Lipschitz continuous coefficients (see Lemma \ref{regdec}), and compatible continuous boundary and final values. It then admits a unique classical solution in $\iC([0,t]\times[\varepsilon_n,m_n],\R)\cap\iC^{1,2}([0,t)\times (\varepsilon_n,m_n),\R)$ thanks to \cite{fried75}{, Theorem 3.6 p138,} that is characterized by the Feynman-Kac formula
                \begin{equation}\label{solG}
                    g^n_s(x) = \E_{x,s}\left\{ \phi'(Y_t)e^{\int_s^t\partial_x\zeta[\nu](Y_u,u)du}1_{\{\tau_n=t\}} \right\} + \E_{x,s}\left\{ g^{\phi}_{\tau_n}(Y_{\tau_n})e^{\int_s^{\tau_n}\partial_x\zeta[\nu](Y_u,u)du}1_{\{\tau_n<t\}} \right\} 
                \end{equation}
                (see \cite{fried75}{, Theorem 5.2, p147}). The stopping time $\tau_n$ represents the minimum between the final time $t$ and the hitting time of the boundary $\{\varepsilon_n,m_n\}$ by the stochastic process $(Y_u)$. Further, it follows from (\ref{exprg}) that
                \begin{displaymath}
                    \begin{array}{l}
                        \displaystyle \E_{x,s}\left\{ g^{\phi}_{\tau_n}(Y_{\tau_n})e^{\int_s^{\tau_n}\partial_x\zeta[\nu](Y_u,u)du}1_{\{\tau_n<t\}} \right\} \vspace{0.20cm}
                        \\
                        \displaystyle \qquad\qquad\qquad\quad = \E_{x,s}\left\{ \E_{Y_{\tau_n},\tau_n}\left[ \phi'(Y_t)e^{\int_{\tau_n}^t\partial_x\zeta[\nu](Y_u,u)du}  \right]e^{\int_s^{\tau_n}\partial_x\zeta[\nu](Y_u,u)du}1_{\{\tau_n<t\}} \right\}
                    \end{array}
                \end{displaymath}
                and thanks to the strong Markov property,
                \begin{displaymath}
                    \begin{array}{l}
                        \displaystyle \E_{x,s}\left\{ g^{\phi}_{\tau_n}(Y_{\tau_n})e^{\int_s^{\tau_n}\partial_x\zeta[\nu](Y_u,u)du}1_{\{\tau_n<t\}} \right\} \vspace{0.20cm}
                        \\
                        \displaystyle \qquad\qquad\qquad = \E_{x,s}\left\{ \E_{x,s}\left[ \phi'(Y_t)e^{\int_{\tau_n}^t\partial_x\zeta[\nu](Y_u,u)du}\big|\iF^W_{\tau_n}  \right]e^{\int_s^{\tau_n}\partial_x\zeta[\nu](Y_u,u)du}1_{\{\tau_n<t\}} \right\} \vspace{0.20cm}
                        \\
                        \displaystyle \qquad\qquad\qquad = \E_{x,s}\left\{ \phi'(Y_t)e^{\int_{s}^t\partial_x\zeta[\nu](Y_u,u)du} 1_{\{\tau_n<t\}} \right\}.
                    \end{array}
                \end{displaymath}
                As $\tau_n\leq t$ a.s, that implies with (\ref{solG}) that $g^n = g^{\phi}$ on $[0,t]\times[\varepsilon_n,m_n]$. As $n$ becomes larger and larger, we deduce that $g^{\phi}\in\iC([0,t]\times\R_+,\R)\cap\iC^{1,2}([0,t)\times (0,\infty),\R)$ and is a classical solution of (\ref{edpd}). Furthermore, it follows from its definition that
                \begin{equation}\label{norx}
                    |g^{\phi}_s(x)|\leq \|\phi'\|_{\infty}e^{\|\partial_x\zeta\|_{\infty}(t-s)}\, ,\, \forall (s,x)\in [0,t]\times\R_+,
                \end{equation}
                and the inequality (\ref{contrgp}) implies that for all $x>0$ and $s\in [0,t)$,
                \begin{equation}\label{difx}
                    \left|\partial_xg^{\phi}_s(x)\right| = \lim_{y\rightarrow x}\left|\frac{g^{\phi}_s(y)-g^{\phi}_s(x)}{y-x} \right| \leq C_T \left(\|\phi'\|_{\infty} + \|\phi''\|_{\infty} \right).
                \end{equation}
                Then let us introduce for a given $\theta>0$ the function defined for all $(s,x)\in[0,t)\times\R_+$  by
                \begin{equation}\label{deff}
                    f^{\theta,\phi}_s(x) = \phi(\theta) + \int_s^t\big\{ \zeta[\nu](\theta,u)g^{\phi}_u(\theta) + D[\nu](\theta,u)\partial_xg_u(\theta)\big\}du + \int_{\theta}^xg^{\phi}_s(y)dy , 
                \end{equation}
                that is in $\iC^{0,1}([0,t]\times\R_+,\R)\cap\iC^{1,3}([0,t)\times(0,\infty),\R)$ and satisfies 
                \[ \partial_xf^{\theta,\phi} = g^{\phi} \,\textrm{ and }\, f^{\theta,\phi}_{| s=t} = \phi.\]
                Then thanks to (\ref{norx}), (\ref{difx}) we have the bound
                \begin{equation}\label{norxg}
                    |\partial_xf^{\theta,\phi}_s(x)| + |\partial_x^2f^{\theta,\phi}_s(x)| \leq C_T\left( \|\phi'\|_{\infty} + \|\phi''\|_{\infty} \right), \forall x>0, \forall s<t
                \end{equation}
                that corresponds to the condition (\ref{contrfg}) with some $p > 0$. The parameter $\theta$ is chosen positive to ensure that it is possible to swap the integral and time derivative symbols in the last term of (\ref{deff}), given that we know nothing about $\partial_sg^{\phi}_s$ in the neighbourhood of $0$. Hence for all $(s,x)\in[0,t)\times(0,\infty)$,
                \begin{displaymath}
                    \begin{array}{l}
                        \displaystyle \partial_sf^{\theta,\phi}_s(x) = -\big\{\zeta[\nu](\theta,s)g^{\phi}_s(\theta) + D[\nu](\theta,s)\partial_xg_s(\theta)\big\} + \int_{\theta}^x\partial_sg^{\phi}_s(y)dy \vspace{0.20cm}
                        \\
                        \displaystyle \hspace{1.6cm} = -\big\{\zeta[\nu](\theta,s)g^{\phi}_s(\theta) + D[\nu](\theta,s)\partial_xg_s(\theta)\big\} \, - \vspace{0.20cm}
                        \\
                        \displaystyle \hspace{3cm} \int_{\theta}^x\left\{ h(y,s)\partial_yg_s^{\phi}(y) + D[\nu](y,s)\partial_y^2g^{\phi}_s(y) + \partial_y\zeta[\nu](y,s)g^{\phi}_s(y) \right\}dy \vspace{0.20cm}
                        \\
                        \displaystyle \hspace{1.6cm} = -\big\{\zeta[\nu](\theta,s)g^{\phi}_s(\theta) + D[\nu](\theta,s)\partial_xg_s(\theta)\big\} - \int_{\theta}^x\partial_y\left\{ \zeta [\nu]g^{\phi} + D[\nu]\partial_yg^{\phi}\right\}(y,s)dy \vspace{0.20cm}
                        \\
                        \displaystyle \hspace{1.6cm} = -\big\{ \zeta[\nu](x,s) g^{\phi}_s(x) + D[\nu](x,s)\partial_xg^{\phi}_s(x) \big\}.
                    \end{array}
                \end{displaymath}
                That implies that the function $(s,x)\mapsto \partial_sf^{\theta,\phi}_s(x)$ is a solution of (\ref{EDP}) and can be continuously extended on the straight line $x=0$ by setting
                \begin{equation}\label{dersb}
                    \partial_sf^{\theta,\phi}_s(0) = -\zeta[\nu](0,s)g^{\phi}_s(0)\, , \, \forall s\leq t.
                \end{equation}
                In addition, $(s,x)\mapsto f_s^{\theta,\phi}(x)$ belongs to $\iC([0,t]\times\R_+,\R)\cap\iC^{1,1}([0,t)\times\R_+,\R)\cap \iC^{1,2}([0,t)\times(0,\infty),\R)$ and satisfies the bound (\ref{contrfg}). Hence the \textbf{Step 1} here above holds and the Feynman-Kac characterization (\ref{exprf}) follows so that the solution is independent of the choice of $\theta$. Let us now denote it by $f^{\phi}_s(x)$, then 
                \begin{equation}
                    |f^{\phi}_s(x)|\leq \|\phi\|_{\infty}\, ,\, \forall (s,x)\in[0,t]\times\R_+
                \end{equation}
                and the bound (\ref{contrf}) follows thanks to the inequality (\ref{norxg}).
            \end{proof}
            \noindent We now state the following result whose proof is similar to the one of Lemma \ref{lemLip}.
            \begin{lemma}\label{lemLipg}
                Let $(x,r)\mapsto g(x,r)$ be a function from $\R_+\times [0,\bar{R}]$ to $\R$ that is Lipschitz continuous, and $\nu^1, \nu^2\in\iC([0,T],(\iM_F(\R_+),w))$ be two solutions of (\ref{maind}), then 
                \begin{equation}\label{ineqg}
                    \sup_{\substack{x\geq 0 \\ 0\leq s\leq t}}\left| g[\nu^1](x,s) - g[\nu^2](x,s) \right| \leq C\|g\|_{\mathrm{Lip}}\int_0^t|||\nu^1_u-\nu^2_u||| du, \forall t\in [0,T].
                \end{equation}
            \end{lemma}
            \noindent The uniqueness result can then be stated.
            \begin{theorem}\label{uniqtwo}
                Under Assumptions \ref{assf}, \ref{assK} and \ref{assder}, the system (\ref{determ}) admits a unique solution in $\iC([0,T],\etr)$ that satisfies $\sup_{0\leq t\leq T}\langle\nu_t,1\rangle < +\infty$.
            \end{theorem}
            \begin{proof}
                The existence and the bound follow from Theorem \ref{exist}. It then suffices to show that the solution is unique. Let $\phi\in\iC^2_b(\R_+,\R)$ and $(\nu^1,R^1),(\nu^2,R^2)\in \iC([0,T],\etr)$ be two solutions of (\ref{determ}) such that $\sup_{0\leq t\leq T}\langle\nu^1_t+\nu^2_t,1\rangle < +\infty$. We denote by $f^{\phi}_s(x)$ the solution of (\ref{EDP}) in the previous theorem for $\nu = \nu^1$ and the final condition $\phi$. As in (\ref{regf}), we regularize this solution by setting
                \[ f^{n,\phi}_s(x) = \int_0^{\infty}f^{\phi}_s(x+y)\rho_n(y)dy\, ,\, \forall (s,x)\in[0,t)\times\R_+. \]
                That allows us to approximate $f^{\phi}_s(x)$ and its derivatives by a sequence of functions that are smooth enough. Let us recall that $f^{n,\phi}\in\iC([0,t]\times\R_+,\R)\cap\iC^{1,\infty}([0,t)\times\R_+,\R)$ for all $n\geq 1$. We deduce from (\ref{contrf}) and (\ref{contrfnL}) that
                \begin{equation}\label{cregx}
                    \big|\partial_sf^{n,\phi}_s(x) + L^{\nu^1}_sf^{n,\phi}_s(x)\big| \leq \frac{C_{T}}{n}\|\phi\| , \forall x\geq 0,\forall s\leq t.
                \end{equation}
                It follows with Lemma \ref{lemLipg} that for any $x\geq 0$ and $s\leq t$,
                \begin{displaymath}
                    \begin{array}{l}
                        \displaystyle \left| \partial_sf^{n,\phi}_s(x) + L^{\nu^2}_sf^{n,\phi}_s(x) \right| \leq \left| \left(L^{\nu^2}_s- L^{\nu^1}_s\right)f^{n,\phi}_s(x) \right| + \frac{C_T}{n}\|\phi\| \vspace{0.20cm}
                        \\
                        \displaystyle \hspace{4.2cm} \leq C_T \left( \|f^{n,\phi}_s\|_{\textrm{Lip}}\int_0^t|||\nu^1_u-\nu^2_u|||du + \frac{\|\phi\|}{n} \right)
                    \end{array}
                \end{displaymath}
                and thanks to (\ref{contrf}),
                \begin{equation}\label{cregX}
                    \left| \partial_sf^{n,\phi}_s(x) + L^{\nu^2}_sf^{n,\phi}_s(x) \right| \leq C_T \|\phi\| \left( \int_0^t|||\nu^1_u-\nu^2_u|||du + \frac{1}{n} \right).
                \end{equation}
                We deduce from (\ref{cregx}) and (\ref{cregX}) here above that for each $j=1,2$, we have
                \[ \langle\nu^j_t,f^{n,\phi}_t\rangle = \langle\nu_0,f_0^{n,\phi}\rangle + \int_0^t\big\langle \nu_s^j, \partial_sf^{n,\phi}_s + L^{\nu_j}_sf^{n,\phi}_s + b[\nu^j](\cdot,s)\iG[f^{n,\phi}_s] - d f^{n,\phi}_s \big\rangle ds.  \]
                That implies that
                \begin{displaymath}\begin{array}{l}
                    \displaystyle \langle\nu^1_t-\nu^2_t,f^{n,\phi}_t\rangle = \int_0^{t}\big\langle\nu^1_s, \partial_sf^{n,\phi}_s + L^{\nu^1}_sf^{n,\phi}_s \big\rangle ds - \int_0^{t}\big\langle\nu^2_s, \partial_sf^{n,\phi}_s + L^{\nu^2}_sf^{n,\phi}_s \big\rangle ds \vspace{0.20cm}
                    \\
                    \displaystyle \hspace{2.5cm} +\, \int_0^{t}\big\langle\nu^1_s-\nu^2_s, b[\nu^1](\cdot,s)\iG[f^{n,\phi}_s] - d f^{n,\phi}_s\big\rangle ds \vspace{0.20cm}
                    \\
                    \displaystyle \hspace{2.5cm} +\, \int_0^{t}\big\langle\nu^2_s, \left\{b[\nu^1] - b[\nu^2]\right\}(\cdot,s)\iG[f^{n,\phi}_s]\big\rangle ds
                \end{array}
                \end{displaymath}
                and by considering the bound $\|\iG[f^{n,\phi}_s]\|\leq\|\phi\|_{\infty}$, Lemma \ref{lemLipg}, (\ref{cregx}) and (\ref{cregX}),
                \begin{equation}\label{ineqdiff}\begin{array}{l}
                    \displaystyle \langle\nu^1_t-\nu^2_t,f^{n,\phi}_t\rangle \leq C_T(\nu^1,\nu^2)\|\phi\|\bigg( \frac{1}{n}\int_0^t\big\langle\nu^1_s+\nu^2_s,1\big\rangle ds + \int_0^t|||\nu^1_s-\nu^2_s||| ds \bigg)  \vspace{0.20cm}
                    \\
                    \displaystyle \hspace{2.5cm} +\, \int_0^{t}\big\langle\nu^1_s-\nu^2_s, b[\nu^1](\cdot,s)\iG[f^{n,\phi}_s] - d f^{n,\phi}_s\big\rangle ds .
                \end{array}
                \end{equation}
                Since $\iG$ is an endomorphism of $\big( \iC^2_b(\R_+,\R),\|.\| \big)$, we obviously verify that given $s< t$ and $n\geq 1$, the function $b[\nu^1](\cdot,s)\iG[f^{n,\phi}_s] - d f^{n,\phi}_s$ belongs to $\iC^2_b(\R_+,\R)$ and satisfies
                \[ \left\| b[\nu^1](\cdot,s)\iG[f^{n,\phi}_s] - d f^{n,\phi}_s\right\| \leq C_T\|f^{n,\phi}_s\|. \]
                Since $\|f^{n,\phi}_s\|\leq\|\phi\|$, that implies with (\ref{ineqdiff}) that
                \[ \langle\nu^1_t-\nu^2_t,f^{n,\phi}_t\rangle \leq C_T(\nu^1,\nu^2)\|\phi\|\bigg( \frac{1}{n}\int_0^t\big\langle\nu^1_s+\nu^2_s,1\big\rangle ds + \int_0^t|||\nu^1_s-\nu^2_s||| ds \bigg). \]
                As $n\rightarrow\infty$, this inequality becomes thanks to the dominated convergence theorem for the left term,
                \[ \langle\nu^1_t-\nu^2_t,\phi\rangle \leq C_T(\nu^1,\nu^2)\|\phi\| \int_0^t|||\nu^1_s-\nu^2_s||| ds. \]
                We conclude by taking at first the supremum for $\|\phi\|\leq 1$ and applying the Gronwall lemma, that $\nu^1_t = \nu^2_t$.  
            \end{proof}

    \section{Existence of a function solution}
        In this section, we use the mild formulation (\ref{mild}) to show that under suitable assumptions the solution of the decoupled equation (\ref{maind}) admits a density with respect to the Lebesgue measure. This density will then be the solution of (\ref{PDE}) in the weak sense given by (\ref{determ}). We first make the following general assumptions.
        \begin{assume}\label{asseps}
            \item[\textsc{(A.1)}] There exists $\ec>0$ such that for all $(x,r)\in\R_+\times(0,\bar{R}]$,
            \[ \zeta(x,r)\geq \ec \textrm{ and } D(x,r)\geq \ec x. \]
            
            \item[\textsc{(A.2)}] $D(x,r)$ is null for $x=0$ and positive for all $(x,r)\in (0,\infty)\times(0,\bar{R}]$. In addition, there exist $C,k>0$ and $\beta\in (0,1]$ such that
            \[ |D(x,r) - D(x,r')| \leq C(1 + x^k)|r-r'|^{\beta}, \forall x\geq 0, \forall r,r'\in (0,\bar{R}]. \]
            
            \item[\textsc{(A.3)}] $\zeta(x,r), D(x,r)\in\iC(\R_+\times(0,\bar{R}],\R)$ are Lipschitz continuous in $x\in\R_+$, uniformly in $r\in(0,\bar{R}]$.
        \end{assume}
        \noindent We easily verify that the above assumption (A.2) implies that the decoupled function $D[\nu]$ satisfies the same H\"older property in its time variable. Indeed, it follows from (\ref{eqR}) that for all $x\geq 0$ and $s,t\in [0,T]$,
        \[ |D[\nu](x,t) - D[\nu](x,s)| = |D(x,R_t) - D(x,R_s)| \leq C(1+x^k)|R_t-R_s|^{\beta} \leq C(1+x^k)|t-s|^{\beta}. \]
        We have the preliminary result
        \begin{prop}\label{sgdensity}
            Under Assumption \ref{asseps}, there exists a unique non negative function $p^{\nu}_{s,t}(x,y)$ defined for all $(t,x,y)\in[0,T]\times\R_+\times\R_+$ and $s< t$, such that for all test functions $\phi\in\iC_b(\R_+,\R)$,
            \begin{equation}\label{semg}
                P^{\nu}_{s,t-s}\phi(x) = \int_0^{\infty}p^{\nu}_{s,t}(x,y)\phi(y)dy.
            \end{equation}
            In addition, if we set
            \begin{equation}\label{constabm}
                c_{\alpha,k} = \alpha\vee\frac{(k - 1)_+}{2} ~\textrm{ and }~\eta_{\alpha,\beta,m} = \beta\wedge\frac{1}{2}\wedge\frac{m-3\alpha}{2m/\alpha} 
            \end{equation}
            for any $0<\alpha<1$, $m> 3\alpha$ and $m\geq 1$, then the function
            \[ y\mapsto D^{m/2}[\nu](y,t)p^{\nu}_{s,t}(x,y) \]
            is, for any $(x,t)\in\R_+\times(0,T]$ and $s<t$, in the Besov space $\iB^{\frac{2m}{2m + 3\alpha}\eta_{\alpha,\beta,m}}_{1,\infty}(\R_+^*)$ (defined in Appendix \ref{appC}) and satisfies the bound
            \begin{equation}\label{besovdens}
                \sup_{s< t\leq T}(t-s)^{m/2}\left\|D^{m/2}[\nu](\cdot,t)p^{\nu}_{s,t}(x,\cdot)\right\|_{\iB^{\frac{2m}{2m+3\alpha}\eta_{\alpha,\beta,m}}_{1,\infty}(\R^*_+)} \leq C \left( 1+x^{c_{\alpha,k}+m/2} \right) ,\,\forall x\geq 0.
            \end{equation}
        \end{prop}
        \begin{proof}
            The left term in (\ref{semg}) is defined by
            \[ P^{\nu}_{s,t-s}\phi(x) = \E\left[ \phi(X^{x}_{s,t}) \right] , \forall x\geq 0, \forall 0\leq s< t\leq T. \]
            The proof then consists in showing that the random variable $X^x_{s,t}$ for fixed $x\geq 0$ and $0\leq s<t\leq T$ admits a density with respect to the Lebesgue measure. As the diffusion coefficient in equation (\ref{edsf}) is degenerate at $x=0$, we will first verify that there is no atom at this point and secondly show that there is a density on the open set $\R_+^{*}$.
            \paragraph{Step 1: No atom on the boundary $\{0\}$.} Let us first notice that for all $x\geq 0$ and $0\leq s<T$, it follows from the comparison theorem that $X^0_{s,t} \leq X^x_{s,t}$ a.s for all $t\in(s,T]$. It then suffices to show that $X^0_{s,t}$ does not have an atom at $\{0\}$. 
            
            \noindent Let us introduce for $h>0$ the convex continuous and non increasing function
            \[ \chi_h(x) = \left( 1 - \frac{x}{h} \right)_+. \]
            Then as $h\rightarrow 0$, the sequence $(\chi_h)_{h>0}$ converges toward $1_{\{x=0\}}$. In addition, we have
            \begin{equation}\label{ineqh}
                \frac{1}{2}1_{[0,h/2]} \leq \chi_h\leq 1_{[0,h]} \textrm{ on }\R_+
            \end{equation}
            and the function defined for fixed $t\in (0,T]$ by
            \begin{equation}
                u_h(x,s) = \E_{x,s}\left[ \chi_h(X^x_{s,t}) \right] , \forall (x,s)\in\R^*_+\times[0,t]
            \end{equation}
            is a classical solution of the partial differential equation
            \begin{equation}
                \partial_su_h + \zeta[\nu](x,s)\partial_xu_h + D[\nu](x,s)\partial_x^2u_h = 0 \textrm{ on }\R^*_+\times[0,t).
            \end{equation}
            Our aim is to show that $\eP(X^0_{s,t}=0) = \lim_{h\rightarrow 0}u_h(0,s) = 0$ for all $s\in[0,t]$. For that purpose, let us first notice that the function $u_h$ is convex since its initial condition $\chi_h$ is also convex (see \cite{lions}). It follows thanks to (A.1) in Assumption \ref{asseps} that
            \begin{equation}
                \partial_su_h + \zeta[\nu](x,s)\partial_xu_h + \ec \,x\partial_x^2u_h \leq 0 \textrm{ on }\R^*_+\times[0,t).
            \end{equation}
            It follows from the Maximum principle that $u_h \leq \bar{u}_h$ where
            \[ \bar{u}_h(x,s) = \E\left[ \chi_h(Y^x_{s,t}) \right] \]
            with
            \[ dY_t = \zeta[\nu](Y_t,t)dt + \sqrt{2\ec\, Y_t}\, dW_t. \]
            The function $\bar{u}_h$ is, indeed, a classical solution of the partial differential equation
            \[ \partial_s\bar{u}_h + \zeta(x,s)\partial_x\bar{u}_h + \ec\, x\partial_x^2\bar{u}_h = 0 , \textrm{ on }\R_+^*\times[0,t) \]
            with the same terminal condition $\chi_h$. Further, thanks to (A.1) in Assumption \ref{asseps}, it follows from the comparison theorem for stochastic differential equations that $Y^x_{s,t}\geq Z^x_{s,t}$ a.s, where
            \begin{equation}\label{edsZ}
                dZ_t = \ec\, dt + \sqrt{2\ec\, Z_t}\, dW_t.
            \end{equation}
            Since $\chi_h$ is decreasing we obtain 
            \[ \bar{u}_h(x,s) \leq \tilde{u}_h(x,s) := \E\left[ \chi_h(Z^x_{s,t}) \right], \]
            and thanks to (\ref{ineqh}), 
            \begin{equation}\label{ineqF}
                \frac{1}{2}F^x_{s,t}(h/2) \leq u_h(x,s) \leq \tilde{u}_h(x,s) \leq \tilde{F}^x_{s,t}(h)
            \end{equation}
            where $F^x_{s,t}$ and $\tilde{F}^x_{s,t} = \tilde{F}^x_{t-s}$ are the cumulative distribution functions of the random variables $X^x_{s,t}$ and $Z^x_{s,t}$ respectively. Further it directly follows from the autosimilarity property of the Brownian motion and the uniqueness in law of the solution of (\ref{edsZ}) given the initial condition, that
            \[ cZ^{x/c}_{t/c} = Z^x_{t} \textrm{ in law, }\forall c>0, \forall t,x\in\R_+. \]
            We deduce that specifically for $x=0$,
            \[ \tilde{F}^0_t(y) = \tilde{F}^0_{t/c}(y/c) , \forall t,y\geq 0, \forall c>0 \]
            and then there exists a function $G:\R_+\rightarrow\R_+$ such that
            \begin{equation}\label{eqG}
                \tilde{F}^0_t(y) = G(y/t) , \forall y\geq 0, \forall t>0.
            \end{equation}
            Since $\tilde{F}^0$ is a weak solution of
            \begin{equation}\label{systF}
                \left\{ \begin{array}{l}
                            \displaystyle \partial_t\tilde{F}^0_{t} - \ec\partial_y(y\partial_y\tilde{F}^0_{t}) + \ec\partial_y\tilde{F}^0_{t} = 0 \vspace{0.20cm}
                            \\
                            \displaystyle (\tilde{F}^0_{t}(y))_{|t=0} = 1, \forall y\geq 0,
                        \end{array}\right.
            \end{equation}
            we consider the time change $\bar{F}_t(y) = \tilde{F}_{t/\ec}(y)$ and $\bar{G}(y) = G(\ec\, y)$, then we have
            \[ \partial_t\bar{F}_{t} - \partial_y(y\partial_y\bar{F}_{t}) + \partial_y\bar{F}_{t} = 0 \]
            with the same initial condition as (\ref{systF}). One deduces with (\ref{eqG}), that
            \[ -y\bar{G}’(y)+\bar{G}’(y)-(y\bar{G}’(y))’ = 0 \]
            whose general solution is given by
            \[ \bar{G}(y) = a_1 + a_2e^{-y} , \forall y>0 \]
            with $a_1,a_2\in\R$. Coming back to the cumulative distribution function of $Z^0_t$, we get
            \begin{equation}\label{limF}
                \tilde{F}^0_t(y) = \bar{G}\left( \frac{y}{\ec\,t} \right) \xrightarrow[y\rightarrow 0]{} a_1+a_2.
            \end{equation}
            Since $(Z^0_{t})_{t\in[s,T]}$ is a non-negative semi-martingale with a zero local time at $0$ (see Lemma \ref{appenBf}, Appendix \ref{appB}), it follows from the Tanaka formula that
            \[ \int_0^t1_{\{Z^0_{u} = 0\}}dZ^0_{u} = \int_0^t1_{\{Z^0_{u} = 0\}}d(Z^0_{u})^+ = \int_0^t1_{\{Z^0_{u} = 0\}}\left( 1_{\{ Z^0_{u}>0 \}}dZ^0_{u} + \frac{1}{2}dL^0_{u}(Z^0) \right) = 0, \]
            and then
            \[ 0 = \int_0^t1_{\{Z^0_{u} = 0\}}\left[ \ec\, du + \sqrt{2\ec\, Z^0_{u}}\, dW_u \right] = \ec\int_0^t1_{\{Z^0_{u} = 0\}}du. \]
            We deduce by taking the expectation, that
            \[ \int_0^t\eP\big(Z^0_{u} = 0\big)du = 0 \, ,\, \forall t\in(s,T]. \]
            Noticing that $\eP(Z^0_{u} = 0)= \tilde{F}^0_u(0)$, we conclude that {$a_1+a_2=0$} in (\ref{limF}). It follows thanks to (\ref{ineqF}) that
            \[ \eP(X^0_{s,t} = 0) = \lim_{h\rightarrow 0}F^0_{s,t}(h/2) \leq 2 \lim_{h\rightarrow 0}\tilde{F}^0_{t-s}(h) = 0\, ,\, \forall t\in(s,T]. \]
            Then for any test function $\phi\in\iC_b(\R_+,\R)$,
            \[ \E\left[ \phi(X^x_{s,t}) \right] = \phi(0)\underbrace{\eP\left[ X^x_{s,t} = 0 \right]}_{= \, 0} +\, \E\left[ \phi(X^x_{s,t}) 1_{\{X^x_{s,t}>0\}} \right], \]
            and then it suffices to show that the random variable $X^x_{s,t}$ has a density with respect to the Lebesgue measure on $\R^*_+$, to conclude that it has a density on $\R_+$.
            \paragraph{Step 2: Existence of a density on $\R_+^*$.} We adapt the proof developed by Romito \cite{romito}, Proposition 3.1 in the case of non homogeneous and non bounded coefficients. It consists in showing the sufficient condition given in Lemma \ref{lemcond} (see Appendix \ref{appC}) with a suitable function. For that purpose, we introduce for a given initial time $s\in[0,T)$, a fixed time $t\in(s,T]$ and $0< \varepsilon<t-s$, the non negative function
            \begin{equation}
                \sigma_{t-\varepsilon,t}(y) = \sqrt{\frac{1}{\varepsilon}\int_{t-\varepsilon}^tD[\nu](y,u)\,du} \, ,\, \forall y\geq 0
            \end{equation}
            that satisfy with $D^{1/2}[\nu](\cdot,t)$ the two following properties : 
            \begin{equation}\label{propSig}
                \begin{array}{l}
                    \displaystyle \textrm{(P1) : }D^{1/2}[\nu](y,t) + \sigma_{t-\varepsilon,t}(y) \leq Cy^{1/2} , \forall y\geq 0, \vspace{0.20cm}
                    \\
                    \displaystyle \textrm{(P2) : }|D^{1/2}[\nu](x,t) - D^{1/2}[\nu](y,t)| + |\sigma_{t-\varepsilon,t}(x) - \sigma_{t-\varepsilon,t}(y)| \leq C|x-y|^{1/2} , \forall x,y\geq 0.
                \end{array}
            \end{equation}
            Let $m\geq 1$, $0<\alpha<1$ and $h\in[-1,1]$, our aim is to show that the inequality (\ref{hypmes}) holds for the random variable $X^x_{s,t}$ with the function $y\mapsto D^{m/2}[\nu](y,t)$. Indeed, for any $\phi\in\eC^{\alpha}_b(\R)$,
            \begin{equation}\label{firstdec}
                \begin{array}{l}
                    \displaystyle \E\left[ D^{m/2}[\nu](X^x_{s,t},t)\Delta^m_h\phi(X^x_{s,t}) \right] = \E\left[ \left\{ D^{m/2}[\nu](X^x_{s,t},t) - \sigma_{t-\varepsilon,t}^m(X^x_{s,t}) \right\}\Delta^m_h\phi(X^x_{s,t}) \right] \vspace{0.20cm} 
                    \\
                    \displaystyle \hspace{6cm} +\, \E\left[ \sigma_{t-\varepsilon,t}^m(X^x_{s,t})\Delta^m_h\phi(X^x_{s,t}) \right].
                \end{array}
            \end{equation}
            Thanks to (A.2) in Assumption \ref{asseps}, we obtain
            \begin{displaymath}
                \begin{array}{ll}
                    \displaystyle \left| D^{m/2}[\nu](X^x_{s,t},t) - \sigma^m_{t-\varepsilon,t}(X^x_{s,t}) \right| \leq C\left| D^{1/2}[\nu](X^x_{s,t},t) - \sigma_{t-\varepsilon,t}(X^x_{s,t}) \right| \vspace{0.20cm}
                    \\
                    \displaystyle \hspace{7cm}\times\bigg(1 + |\sigma_{t-\varepsilon,t}(X^x_{s,t})|^{m-1} + |D^{1/2}[\nu](X^x_{s,t},t)|^{m-1} \bigg) \vspace{0.20cm}
                    \\
                    \displaystyle \hspace{5.3cm} \leq C\left(\sqrt{\frac{1}{\varepsilon}\int_{t-\varepsilon}^t\big|D[\nu](X^x_{s,t},t)-D[\nu](X^x_{s,t},u)\big|du}\,\right)\bigg(1+ (X^x_{s,t})^{\frac{m-1}{2}}\bigg) \vspace{0.20cm}
                    \\
                    \displaystyle \hspace{5.3cm} \leq C\left(\sqrt{\frac{1}{\varepsilon}\int_{t-\varepsilon}^t(t-u)^{\beta}du}\,\right) \left(1 + (X^x_{s,t})^k\right)^{1/2} \bigg(1+(X^x_{s,t})^{\frac{m-1}{2}}\bigg) \vspace{0.20cm}
                    \\
                    \displaystyle \hspace{5.3cm} \leq C \left(1 + (X^x_{s,t})^{\frac{m+k-1}{2}}\right)\varepsilon^{\beta/2}.
                \end{array}
            \end{displaymath}
            It follows from the moment property (\ref{moment}) and the property 1. of Lemma \ref{propD} (see Appendix \ref{appC}) that the first term in the right side of (\ref{firstdec}) satisfies
            \begin{equation}\label{ffterm}
                \E\left[ \left\{ D^{m/2}[\nu](X^x_{s,t},t) - \sigma_{t-\varepsilon,t}^m(X^x_{s,t}) \right\}\Delta^m_h\phi(X^x_{s,t}) \right] \leq C {\big(1+x^{\frac{m+k-1}{2}}\big)} |h|^{\alpha}\|\phi\|_{\eC^{\alpha}_b(\R)}\varepsilon^{\beta/2}.
            \end{equation}
            In order to get a good control on the second term of the same equation, we introduce the following new process 
            \begin{displaymath}
                \forall u\in [s,t], \, Y^{\varepsilon,x}_{s,u} = \left\{\begin{array}{ll}
                    \displaystyle X^x_{s,u} & \textrm{ if }u\leq t-\varepsilon, \vspace{0.20cm}
                    \\
                    \displaystyle X^x_{s,t-\varepsilon} + \int_{t-\varepsilon}^u\sqrt{2D[\nu](X^x_{s,t-\varepsilon},u)}\, dW_u & \textrm{ if } u > t-\varepsilon.
                \end{array}\right.
            \end{displaymath}
            This process is continuous and Gaussian on the interval $(t-\varepsilon, t]$ when it is conditioned by $X^x_{t-\varepsilon}$. Then,
            \begin{equation}\label{decompo}
                \begin{array}{l}
                    \displaystyle \E\left[ \sigma_{t-\varepsilon,t}^m(X^x_{s,t})\Delta^m_h\phi(X^x_{s,t}) \right] = \E\left[ \left\{ \sigma_{t-\varepsilon,t}^m(X^x_{s,t}) - \sigma_{t-\varepsilon,t}^m(X^x_{s,t-\varepsilon}) \right\}\Delta^m_h\phi(X^x_{s,t}) \right] \vspace{0.20cm}
                    \\
                    \displaystyle \qquad\qquad\qquad\qquad\qquad\qquad +~ \E\left[ \sigma_{t-\varepsilon,t}^m(X^x_{s,t-\varepsilon})\left\{ \Delta^m_h\phi(X^x_{s,t}) - \Delta^m_h\phi(Y^{\varepsilon,x}_{s,t})\right\} \right] \vspace{0.20cm}
                    \\
                    \displaystyle \qquad\qquad\qquad\qquad\qquad\qquad +~ \E\left[ \sigma^m_{t-\varepsilon,t}(X^x_{s,t-\varepsilon})\Delta^m_h\phi(Y^{\varepsilon,x}_{s,t}) \right].
                \end{array}
            \end{equation}
            It now suffices to show that each term of this decomposition can be well controlled. For the first one, it follows from (\ref{propSig}) and the property 1. of Lemma \ref{propD} (see Appendix \ref{appC}) that
            \begin{displaymath}
                \begin{array}{l}
                    \displaystyle \E\left[ \left\{ \sigma_{t-\varepsilon,t}^m(X^x_{s,t}) - \sigma_{t-\varepsilon,t}^m(X^x_{s,t-\varepsilon}) \right\}\Delta^m_h\phi(X^x_{s,t}) \right] \vspace{0.30cm} 
                    \\
                    \displaystyle \hspace{1cm} \leq C |h|^{\alpha}\|\phi\|_{\eC^{\alpha}_b(\R)}\E\left[ \left| \sigma_{t-\varepsilon,t}^m(X^x_{s,t}) - \sigma_{t-\varepsilon,t}^m(X^x_{s,t-\varepsilon}) \right| \right] \vspace{0.20cm}
                    \\
                    \displaystyle \hspace{1cm} \leq C |h|^{\alpha}\|\phi\|_{\eC^{\alpha}_b(\R)}\E\bigg[ \left| \sigma_{t-\varepsilon,t}(X^x_{s,t}) - \sigma_{t-\varepsilon,t}(X^x_{s,t-\varepsilon}) \right|\vspace{0.20cm}
                    \\
                    \displaystyle \hspace{5cm}\times\,\bigg(1+ |\sigma_{t-\varepsilon,t}(X^x_{s,t})|^{m-1}+|\sigma_{t-\varepsilon,t}(X^x_{s,t-\varepsilon})|^{m-1}\bigg) \bigg] \vspace{0.20cm}
                    \\
                    \displaystyle \hspace{1cm} \leq C |h|^{\alpha}\|\phi\|_{\eC^{\alpha}_b(\R)}\E\left[ \left| \sigma_{t-\varepsilon,t}(X^x_{s,t}) - \sigma_{t-\varepsilon,t}(X^x_{s,t-\varepsilon}) \right|\sup_{0\leq u\leq t}\bigg(1+(X^x_{s,u})^{\frac{m-1}{2}}\bigg) \right] \vspace{0.20cm}
                    \\
                    \displaystyle \hspace{1cm} \leq C |h|^{\alpha}\|\phi\|_{\eC^{\alpha}_b(\R)}\E\left[ \left| X^x_{s,t} - X^x_{s,t-\varepsilon} \right|^{1/2}\sup_{0\leq u\leq t}\left(1+(X^x_{s,u})^{\frac{m-1}{2}} \right) \right].
                \end{array}
            \end{displaymath}
            The Cauchy-Schwarz inequality and the moment property (\ref{moment}) imply that
            \begin{displaymath}\begin{array}{l}
                \displaystyle \E\left[ \left\{ \sigma_{t-\varepsilon,t}^m(X^x_{s,t}) - \sigma_{t-\varepsilon,t}^m(X^x_{s,t-\varepsilon}) \right\}\Delta^m_h\phi(X^x_{s,t}) \right] \vspace{0.20cm} 
                \\
                \displaystyle \hspace{2cm}\leq C |h|^{\alpha}\|\phi\|_{\eC^{\alpha}_b(\R)}\E\left[ \left| X^x_{s,t} - X^x_{s,t-\varepsilon} \right|\right]^{1/2}\E\left[\sup_{0\leq u\leq t}\left(1+(X^x_{s,u})^{m-1} \right) \right]^{1/2}  \vspace{0.20cm}
                \\
                \displaystyle \hspace{2cm} \leq C {\big( 1 + x^{\frac{m-1}{2}} \big)} |h|^{\alpha}\|\phi\|_{\eC^{\alpha}_b(\R)}\E\left[ \left| X^x_{s,t} - X^x_{s,t-\varepsilon} \right|\right]^{1/2}.
            \end{array}
            \end{displaymath}
            In addition, for all $s\leq t-\varepsilon\leq r\leq t$, 
            \[ \left| X^x_{s,r} - X^x_{s,t-\varepsilon} \right| \leq \int_{t-\varepsilon}^r|\zeta[\nu](X^x_{s,u},u)|du + \sup_{t-\varepsilon\leq r'\leq r}\left| \int_{t-\varepsilon}^{r'}\sqrt{2D[\nu](X^x_{s,u},u)}\, dW_u \right|. \]
            Then, using the Burkholder-Davis-Gundy inequality and the moment property (\ref{moment}),
            \begin{equation}\label{ineqX}
                \begin{array}{l}
                    \displaystyle \E\left[ \left| X^x_{s,r} - X^x_{s,t-\varepsilon} \right| \right] \leq C\left\{\varepsilon \E\left[ \sup_{s\leq u\leq t}\left(1 + X^x_{s,u} \right) \right] + \E\left[ \left( \int_{t-\varepsilon}^r 2D[\nu](X^x_{s,u},u)du \right)^{1/2} \right]\right\} \vspace{0.20cm}
                    \\
                    \displaystyle \qquad\qquad\qquad\qquad~ \leq C\left\{\varepsilon\E\left[ \sup_{s\leq u\leq t}\left(1 + X^x_{s,u} \right) \right] + \varepsilon^{1/2}\E\left[ \sup_{s\leq u\leq t}\left(X^x_{s,u}\right)^{1/2} \right] \right\} \vspace{0.20cm}
                    \\
                    \displaystyle \qquad\qquad\qquad\qquad~ \leq C{\big( 1+x \big)}\varepsilon^{1/2}.
                \end{array}
            \end{equation}
            That implies for the first term in the right side of (\ref{decompo}) that
            \begin{equation}\label{firsterm}
                \E\left[ \left\{ \sigma^m_{t-\varepsilon,t}(X^x_{s,t}) - \sigma^m_{t-\varepsilon,t}(X^x_{s,t-\varepsilon}) \right\}\Delta^m_h\phi(X^x_{s,t}) \right] \leq C {\big( 1+x^{\frac{m}{2}} \big)} |h|^{\alpha}\|\phi\|_{\eC^{\alpha}_b(\R)}\varepsilon^{1/4}.
            \end{equation}
            The second one satisfies with (P1) in (\ref{propSig}), the property 2. of Lemma \ref{propD} (see Appendix \ref{appC}), the H\"older inequality and the moment property (\ref{moment}),
            \begin{displaymath}
                \begin{array}{l}
                    \displaystyle \E\left[ \sigma_{t-\varepsilon,t}^m(X^x_{s,t-\varepsilon})\left\{ \Delta^m_h\phi(X^x_{s,t}) - \Delta^m_h\phi(Y^{\varepsilon,x}_{s,t})\right\} \right] \vspace{0.20cm} 
                    \\
                    \displaystyle \hspace{2cm} \leq C \|\phi\|_{\eC^{\alpha}_b(\R)} \E\left[\left| X^x_{s,t} - Y^{\varepsilon,x}_{s,t} \right|^{\alpha} \sup_{s\leq u\leq t}(X^x_{s,u})^{m/2} \right] \vspace{0.20cm}
                    \\
                    \displaystyle \hspace{2cm} \leq C \|\phi\|_{\eC^{\alpha}_b(\R)} \E\left[\left| X^x_{s,t} - Y^{\varepsilon,x}_{s,t} \right|\right]^{\alpha}\E\left[ \sup_{s\leq u\leq t}(X^x_{s,u})^{m/2(1-\alpha)} \right]^{1-\alpha} \vspace{0.20cm}
                    \\
                    \displaystyle \hspace{2cm} \leq C {\big( 1+x^{m/2} \big)} \|\phi\|_{\eC^{\alpha}_b(\R)} \E\left[\left| X^x_{s,t} - Y^{\varepsilon,x}_{s,t} \right|\right]^{\alpha}.
                \end{array}
            \end{displaymath}
            On the other hand,
            \begin{displaymath}
                \begin{array}{l}
                    \displaystyle \left| X^x_{s,t} - Y^{\varepsilon,x}_{s,t} \right| \leq \int_{t-\varepsilon}^t|\zeta[\nu](X^x_{s,u},u)|du \vspace{0.20cm}
                    \\
                    \displaystyle \hspace{2.7cm} +\, \sup_{t-\varepsilon\leq t'\leq t}\left| \int_{t-\varepsilon}^{t'}\left[ \sqrt{2D[\nu](X^x_{s,u},u)} - \sqrt{2D[\nu](X^x_{s,t-\varepsilon},u)}\, \right]\, dW_u \right|,
                \end{array}
            \end{displaymath}
            and thanks to the Burkholder-Davis-Gundy, the H\"older inequalities and (\ref{moment}),
            \begin{displaymath}
                \begin{array}{l}
                    \displaystyle \E\left[ \left| X^x_{s,t} - Y^{\varepsilon,x}_{s,t} \right| \right] \leq C\varepsilon \E\left[ \sup_{s\leq u\leq t}\left(1 + X^x_{s,u} \right) \right] \vspace{0.20cm} 
                    \\
                    \displaystyle \hspace{3cm} +~ \E\left[ \left(\int_{t-\varepsilon}^t\left| \sqrt{2D[\nu](X^x_{s,u},u)} - \sqrt{2D[\nu](X^x_{s,t-\varepsilon},u)}\, \right|^2du \right)^{1/2}\right] \vspace{0.20cm}
                    \\
                    \displaystyle \hspace{2.8cm} \leq C\varepsilon \E\left[ \sup_{s\leq u\leq t}\left(1 + X^x_{s,u} \right) \right] + C\E\left[ \left(\int_{t-\varepsilon}^t\left| X^x_{s,u} - X^x_{s,t-\varepsilon} \right|du \right)^{1/2}\right] \vspace{0.20cm}
                    \\
                    \displaystyle \hspace{2.8cm} \leq C\varepsilon \E\left[ \sup_{s\leq u\leq t}\left(1 + X^x_{s,u} \right) \right] + C \E\left[\int_{t-\varepsilon}^t\left| X^x_{s,u} - X^x_{s,t-\varepsilon} \right|du \right]^{1/2} \vspace{0.20cm}
                    \\
                    \displaystyle \hspace{2.8cm} \leq C\left\{ {\big(1+x\big)}\varepsilon + \E\left[\int_{t-\varepsilon}^t\left| X^x_{s,u} - X^x_{s,t-\varepsilon} \right|du \right]^{1/2} \right\}.
                \end{array}
            \end{displaymath}
            The last term in the right side is well controlled with (\ref{ineqX}), so that
            \[ \E\left[\int_{t-\varepsilon}^t\left| X^x_{s,u} - X^x_{s,t-\varepsilon} \right|du \right] = \int_{t-\varepsilon}^t\E\left[\left| X^x_{s,u} - X^x_{s,t-\varepsilon} \right| \right]du \leq C{\big( 1+x \big)}\varepsilon^{3/2}, \]
            that implies for the second term in the decomposition (\ref{decompo}) that
            \begin{equation}\label{secondterm}
                \E\left[ \sigma_{t-\varepsilon,t}^m(X^x_{s,t-\varepsilon})\left\{ \Delta^m_h\phi(X^x_{s,t}) - \Delta^m_h\phi(Y^{\varepsilon,x}_{s,t})\right\} \right] \leq C{\big( 1+x^{\alpha + m/2} \big)}\|\phi\|_{\eC^{\alpha}_b(\R)}\varepsilon^{3\alpha/4}.
            \end{equation}
            Finally, by conditioning by $X^x_{t-\varepsilon}$, the third and last term of that decomposition satisfies
            \begin{displaymath}
                \begin{array}{l}
                    \displaystyle \E\left[ \sigma_{t-\varepsilon,t}^m(X^x_{s,t-\varepsilon})\Delta^m_h\phi(Y^{\varepsilon,x}_{s,t}) \right] \vspace{0.20cm} 
                    \\
                    \displaystyle \hspace{1.5cm} = \E\left[ \sigma_{t-\varepsilon,t}^m(X^x_{s,t-\varepsilon})\Delta^m_h\phi\left(X^x_{s,t-\varepsilon} + \int_{t-\varepsilon}^t\sqrt{2D[\nu](X^x_{s,t-\varepsilon},u)}dW_u \right) \right] \vspace{0.20cm}
                    \\
                    \displaystyle \hspace{1.5cm} = \E\left\{\E\left[ \sigma_{t-\varepsilon,t}^m(y)\Delta^m_h\phi\left(y + \int_{t-\varepsilon}^t\sqrt{2D[\nu](y,u)}dW_u \right) \right]_{\big|y = X^x_{s,t-\varepsilon}} \right\}
                \end{array}
            \end{displaymath}
            The expression inside the second expectation is zero if $y=0$, and for $y>0$ it satisfies
            \[ \int_{t-\varepsilon}^t\sqrt{2D[\nu](y,u)}dW_u \thicksim \iN(0,2\varepsilon\sigma^2_{t-\varepsilon,t}(y)). \]
            Using the function 
            \[ g^{\varepsilon,y}_{t}(z) = \frac{1}{2\sigma_{t-\varepsilon,t}(y)\sqrt{\pi\varepsilon}}\exp\left[ -\frac{x^2}{4\varepsilon\sigma^2_{t-\varepsilon,t}(y)} \right] \]
            and the properties 3. and 4. of Lemma \ref{propD} (see Appendix \ref{appC}), that implies that
            \begin{displaymath}
                \begin{array}{l}
                    \displaystyle \E\left[ \sigma_{t-\varepsilon,t}^m(y)\Delta^m_h\phi\left(y + \int_s^t\sqrt{2D[\nu](y,u)}dW_u \right) \right] = \sigma^m_{t-\varepsilon,t}(y) \int_{-\infty}^{+\infty} \Delta^m_h\phi\left(y + z \right) g^{\varepsilon,y}_{t}(z)dz  \vspace{0.20cm}
                    \\
                    \displaystyle \hspace{7.9cm} = \sigma^m_{t-\varepsilon,t}(y) \int_{-\infty}^{+\infty} \phi\left(y + z \right) \Delta^m_{-h}g^{\varepsilon,y}_{t}(z)dz \vspace{0.20cm}
                    \\
                    \displaystyle \hspace{7.9cm} \leq \sigma^m_{t-\varepsilon,t}(y) \|\phi\|_{\infty}\|\Delta^m_{-h}g^{\varepsilon,y}_{t}\|_{L^1} \vspace{0.30cm}
                    \\
                    \displaystyle \hspace{7.9cm} \leq \sigma_{t-\varepsilon,t}^m(y) |h|^m \|\phi\|_{\eC^{\alpha}_b(\R)}\|\partial_z^m g^{\varepsilon,y}_{t}\|_{L^1}.
                \end{array}
            \end{displaymath}
            Further, 
            \[ \|\partial_z^mg^{\varepsilon,y}_t\|_{L^1} \leq \frac{C_m}{\varepsilon^{m/2}\sigma^m_{t-\varepsilon,t}(y)}  \]
            (see Lemma \ref{dernorm} in Appendix \ref{appD}), and then
            \[ \E\left[ \sigma_{t-\varepsilon,t}^m(y)\Delta^m_h\phi\left(y + \int_s^t\sqrt{2D[\nu](y,u)}dW_u \right) \right] \leq C_m |h|^m \|\phi\|_{\iC^{\alpha}_b(\R)} \varepsilon^{-m/2}. \]
            The following control on the third term of the decomposition (\ref{decompo}) follows :
            \begin{equation}\label{thirdterm}
                \E\left[ \sigma_{t-\varepsilon,t}^m(X^x_{s,t-\varepsilon})\Delta^m_h\phi(Y^{\varepsilon,x}_{s,t}) \right] \leq C |h|^m\|\phi\|_{\iC^{\alpha}_b(\R)}\varepsilon^{-m/2}.
            \end{equation}
            Mixing (\ref{firsterm}), (\ref{secondterm}) and (\ref{thirdterm}), it follows from the decomposition (\ref{decompo}) that
            \[ \E\left[ \sigma^m_{t-\varepsilon,t}(X^x_{s,t})\Delta^m_h\phi(X^x_{s,t}) \right] \leq C\bigg(1 + x^{\alpha+m/2}\bigg)\|\phi\|_{\iC^{\alpha}_b(\R)}\left[ |h|^{\alpha}\varepsilon^{1/4} + \varepsilon^{3\alpha/4} + |h|^m\varepsilon^{-m/2} \right], \]
            and then with (\ref{ffterm}), we obtain
            \begin{displaymath}
                \begin{array}{l}
                    \displaystyle \E\left[ D^{m/2}[\nu](X^x_{s,t},t)\Delta^m_h\phi(X^x_{s,t}) \right] \leq C{\big( 1 + x^{\alpha+m/2} + x^{\frac{m+k-1}{2}} \big)}\|\phi\|_{\eC^{\alpha}_b(\R)}\left[ |h|^{\alpha}\varepsilon^{\frac{1}{2}(\beta\wedge\frac{1}{2})} + \varepsilon^{3\alpha/4} + |h|^m\varepsilon^{-m/2} \right]. 
                \end{array}
            \end{displaymath}
            We choose for $h\neq 0$
            \[ \varepsilon = (t-s) \frac{|h|^{\frac{2m}{m + 3\alpha/2}}}{ 1+|h|^{\frac{2m}{m + 3\alpha/2}} } \in\, ]0,t-s[. \]
            Then since $|h|\leq 1$, we get with (\ref{constabm})
            \begin{displaymath}
                \begin{array}{l}
                    \displaystyle \E\left[ D^{m/2}[\nu](X^x_{s,t},t)\Delta^m_h\phi(X^x_{s,t}) \right] \vspace{0.20cm} 
                    \\
                    \displaystyle \hspace{1cm} \leq C{\big(1+x^{c_{\alpha,k}+m/2}\big)}\|\phi\|_{\eC^{\alpha}_b(\R)}\left[ |h|^{\alpha}\left( (t-s) \frac{|h|^{\frac{2m}{m + 3\alpha/2}}}{ 1+|h|^{\frac{2m}{m + 3\alpha/2}} } \right)^{\frac{1}{2}(\beta\wedge\frac{1}{2})} \right. \vspace{0.20cm} 
                    \\
                    \displaystyle \hspace{2cm} \left. +\, \left( (t-s) \frac{|h|^{\frac{2m}{m + 3\alpha/2}}}{ 1+|h|^{\frac{2m}{m + 3\alpha/2}} } \right)^{3\alpha/4} + |h|^m\left( (t-s) \frac{|h|^{\frac{2m}{m + 3\alpha/2}}}{ 1+|h|^{\frac{2m}{m + 3\alpha/2}} } \right)^{-m/2}\, \right] \vspace{0.20cm}
                    \\
                    \displaystyle \hspace{1cm} \leq C{\big(1+x^{c_{\alpha,k}+m/2}\big)}\|\phi\|_{\eC^{\alpha}_b(\R)}\bigg[ (t-s)^{\frac{1}{2}(\beta\wedge\frac{1}{2})}|h|^{\alpha + \frac{1}{2}(\beta\wedge\frac{1}{2})\frac{2 m}{m + 3\alpha/2}}  \vspace{0.20cm} 
                    \\
                    \displaystyle \hspace{2cm} +\, (t-s)^{3\alpha/4}|h|^{\frac{3\alpha}{4}\frac{2 m}{m + 3\alpha/2}} + \left(\frac{2}{t-s}\right)^{m/2}|h|^{m-\frac{m}{2}\frac{2 m}{m + 3\alpha/2}} \bigg]  \vspace{0.20cm}
                    \\
                    \displaystyle \hspace{1cm} \leq C\left( 1+(t-s)^{-m/2}\right){\big(1+x^{c_{\alpha,k}+m/2}\big)}\|\phi\|_{\eC^{\alpha}_b(\R)}\left[ |h|^{\alpha + \frac{1}{2}(\beta\wedge\frac{1}{2})\frac{2 m}{m + 3\alpha/2}} + |h|^{\frac{3\alpha}{4}\frac{2 m}{m + 3\alpha/2}} + |h|^{m-\frac{m}{2}\frac{2 m}{m + 3\alpha/2}} \right]  \vspace{0.20cm}
                    \\
                    \displaystyle \hspace{1cm} \leq C\left( 1+(t-s)^{-m/2}\right)\big(1+x^{c_{\alpha,k}+m/2}\big)\|\phi\|_{\eC^{\alpha}_b(\R)}\left[ |h|^{\alpha + (\beta\wedge\frac{1}{2})\frac{2 m}{2m + 3\alpha}} + |h|^{\frac{3\alpha m}{2m + 3\alpha}} \right] \vspace{0.20cm}
                    \\
                    \displaystyle \hspace{1cm} \leq C\left( 1+(t-s)^{-m/2}\right)\big(1+x^{c_{\alpha,k}+m/2}\big)\|\phi\|_{\eC^{\alpha}_b(\R)}\left[ |h|^{\alpha + (\beta\wedge\frac{1}{2})\frac{2 m}{2m + 3\alpha}} + |h|^{\alpha + \alpha\,\frac{m-3\alpha}{2m + 3\alpha}} \right].
                \end{array}
            \end{displaymath}
            It follows with (\ref{constabm}) that for $m>3\alpha$,
            \begin{equation}\label{bounddens}
                \E\big[ D^{m/2}[\nu](X^x_{s,t},t)\Delta^m_h\phi(X^x_{s,t}) \big] \leq C\left( 1+(t-s)^{-m/2}\right)\big(1+x^{c_{\alpha,k}+m/2}\big)\|\phi\|_{\eC^{\alpha}_b(\R)} |h|^{\alpha + \frac{2 m}{2m + 3\alpha}\eta_{\alpha,\beta,m}}
            \end{equation}
            and for $m> 3\alpha$, we have
            \[ \alpha < \alpha + \frac{2 m}{2m + 3\alpha}\eta_{\alpha,\beta,m} < m. \]
            We deduce from Lemma \ref{lemcond} (see Appendix \ref{appC}) and hypothesis (A.2) in Assumption \ref{asseps} that the random variable $X^x_{s,t}$ admits a density with respect to the Lebesgue measure on $\R^*_+$. In addition, if we denote by $y\mapsto p^{\nu}_{s,t}(x,y)$ its density, then the function $y\mapsto D^{m/2}[\nu](y,t)p^{\nu}_{s,t}(x,y)$ is in the Besov space $\iB^{\frac{2 m}{2m + 3\alpha}\eta_{\alpha,\beta,m} }_{1,\infty}(\R_+^*)$ and satisfies
            \[ \left\| D^{m/2}[\nu](\cdot,t)p^{\nu}_{s,t}(x,\cdot) \right\|_{\iB^{\frac{2 m}{2m + 3\alpha}\eta_{\alpha,\beta,m} }_{1,\infty}(\R_+^*)} \leq C\left( 1+(t-s)^{-m/2}\right)\big(1+x^{c_{\alpha,k}+m/2}\big) \]
            that implies the uniform bound (\ref{besovdens}).
            
        \end{proof}
        \noindent This result shows that despite the degeneracy of the diffusion coefficient in (\ref{edsf}), the law of the solution of this equation is regularized. This property is very important because it induces a regularization of the measure $\nu_t$ as we will show in the following result.
        \begin{theorem}\label{theodensity}
            Under the assumptions of Theorems \ref{uniqone} or \ref{uniqtwo}, and Assumption \ref{asseps}, the measure $\nu_t$ admits a density $x\mapsto u_t(x)$ with respect to the Lebesgue measure for each $t>0$. In addition, if we use the notation $p^u$ to denote the density $p^{\nu}$ defined in Proposition \ref{sgdensity}, then the couple $(u,R)$ is characterized by the mild system
            \begin{equation}\label{mildfunc}\left\{
                \begin{array}{l}
                    \displaystyle \forall t\in (0,T], \forall x>0, \vspace{0.20cm}
                    \\
                    \displaystyle u_t(x) = \left\langle\nu_0,p_{0,t}^u(\cdot,x) \right\rangle + \int_0^t\int_0^{\infty}\left[b(x',R_s)\iG[p^u_{s,t}(\cdot,x)](x') - d(x')p^{u}_{s,t}(x',x) \right] u_s(x')dx'ds \vspace{0.20cm}
                    \\
                    \displaystyle R_t = R_0 + \int_0^t\left\{ r_{in}-R_s-\int_0^{\infty}\chi(x',R_s)u_s(x')dx'\right\}ds
                \end{array}\right.
            \end{equation}
            and is the unique solution of (\ref{PDE}) in $\eL^{\infty}((0,T],\eL^1(\R_+)\times\R)$, in the sense given by (\ref{determ}) for the initial condition $(\nu_0,R_0)$. 
            
            Furthermore, if the initial condition satisfies
            \begin{equation}\label{condbesov}
                \int_0^{\infty}\big(1+x\big)\nu_0(dx) < \infty,
            \end{equation}
            then we have $u\in\eL^{\infty}\big((0,T],\eL^1(\R_+,(1+x)dx)\big)$ and for any $\lambda\in (0,\frac{5-2\sqrt{6}}{3}]$
            \begin{equation}\label{ineqbesov}
                \sup_{0<t\leq T}t^{1/2}\left\| D^{1/2}(\cdot,R_t)u_t \right\|_{\iB_{1,\infty}^{\lambda}(\R^*_+)} < +\infty.
            \end{equation}
        \end{theorem}
        \medskip
        \begin{proof} 
            For a given non negative function $\phi\in\iC_b(\R_+,\R_+)$ it follows from the mild formulation that
            \begin{displaymath}
                \begin{array}{l}
                    \displaystyle \langle\nu_t,\phi\rangle = \int_0^{\infty}\int_0^{\infty}\phi(y)p^{\nu}_{0,t}(x,y)dy\nu_0(dx) \vspace{0.20cm} 
                    \\
                    \displaystyle \hspace{1.5cm} + \int_0^t\int_0^{\infty}\bigg\{ b[\nu](x,s)\left[ - \int_0^{\infty}\phi(y)p^{\nu}_{s,t}(x,y)dy + 2\int_0^{1}\int_0^{\infty}\phi(y)p^{\nu}_{s,t}(\alpha x,y)dy M(d\alpha) \right] \vspace{0.20cm} 
                    \\
                    \displaystyle \hspace{7cm} -~ d(x)\int_0^{\infty}\phi(y)p^{\nu}_{s,t}(x,y)dy \bigg\}\nu_s(dx)ds \vspace{0.20cm}
                    \\
                    \displaystyle \hspace{1.2cm} \leq \int_0^{\infty}\underbrace{\left( \int_0^{\infty}p^{\nu}_{0,t}(x,y)\nu_0(dx) + 2\|b\|_{\infty}\int_0^t\int_0^{\infty}\int_0^1p^{\nu}_{s,t}(\alpha x,y)M(d\alpha)\nu_s(dx)ds \right)}_{H^{\nu}_t(y)}\phi(y)dy
                \end{array}
            \end{displaymath}
            In addition, the function $y\mapsto H^{\nu}_t(y)$ is non negative and belongs to $\eL^1(\R_+)$. Indeed,
            \begin{displaymath}\begin{array}{l}
                \displaystyle \int_0^{\infty}H^{\nu}_t(y)dy = \int_0^{\infty}\int_0^{\infty}p^{\nu}_{0,t}(x,y)dy\,\nu_0(dx) \vspace{0.20cm} 
                \\
                \displaystyle \hspace{3cm} +\, 2\|b\|_{\infty}\int_0^t\int_0^{\infty}\int_0^1\int_0^{\infty}p^{\nu}_{s,t}(\alpha x,y)dy\,M(d\alpha)\nu_s(dx)ds \vspace{0.20cm}
                \\
                \displaystyle \qquad\qquad\qquad = \langle\nu_0,1\rangle + 2\|b\|_{\infty}\int_0^t\langle\nu_s,1\rangle ds < +\infty.
            \end{array}
            \end{displaymath}
            We deduce that for all $t>0$, the measure $\nu_t$ is absolutely continuous with respect to the Lebesgue measure. We denote by $u_t$ its density, then Theorems \ref{uniqone} and \ref{uniqtwo} imply that $u\in \eL^{\infty}((0,T],\eL^1(\R_+))$. The mild system (\ref{mildfunc}) directly follows from the mild formulation (\ref{mild}) and the representation (\ref{semg}). Indeed, for all non negative functions $\phi\in\iC_b(\R_+,\R)$ and $t>0$,
            \begin{displaymath}
                \begin{array}{l}
                    \displaystyle \int_0^{\infty}u_t(x)\phi(x)dx = \int_0^{\infty}\int_0^{\infty}\phi(y)p^{u}_{0,t}(x,y)dy\nu_0(dx) \vspace{0.20cm} 
                    \\
                    \displaystyle \hspace{0.5cm} + \int_0^t\int_0^{\infty}\bigg\{ b(x,R_s)\left[ - \int_0^{\infty}\phi(y)p^{u}_{s,t}(x,y)dy + 2\int_0^{1}\int_0^{\infty}\phi(y)p^{u}_{s,t}(\alpha x,y)dyM(d\alpha) \right] \vspace{0.20cm} 
                    \\
                    \displaystyle \hspace{6.5cm} -~ d(x)\int_0^{\infty}\phi(y)p^{u}_{s,t}(x,y)dy \bigg\}u_s(x)dxds \vspace{0.20cm}
                    \\
                    \displaystyle \hspace{0.5cm} = \int_0^{\infty}\left( \int_0^{\infty}p^{u}_{0,t}(x,y)\nu_0(dx)\right)\phi(y)dy  \vspace{0.20cm}
                    \\
                    \displaystyle \qquad\quad + \int_0^{\infty}\bigg(\int_0^t\int_0^{\infty}\bigg\{ b(x,R_s)\left[ - p^{u}_{s,t}(x,y) + 2\int_0^{1}p^{u}_{s,t}(\alpha x,y)M(d\alpha) \right] \vspace{0.20cm} 
                    \\
                    \displaystyle \hspace{6.5cm} -\, d(x)p^{u}_{s,t}(x,y) \bigg\}u_s(x)dxds\bigg)\phi(y)dy.
                \end{array}
            \end{displaymath}
            We obtain the mild system (\ref{mildfunc}), that implies for any $(x,t)\in\R^*_+\times(0,T]$ and any integer $m\geq 1$, that
            \begin{displaymath}
                \begin{array}{l}
                    \displaystyle D^{m/2}[\nu](x,t)u_t(x) = \left\langle\nu_0,\tilde{\sigma}_t^{m}(x)p_{0,t}^u(\cdot,x) \right\rangle \vspace{0.20cm}
                    \\
                    \displaystyle \hspace{1cm} +~\int_0^t\int_0^{\infty}\left[b(x',R_s)\iG[D^{m/2}[\nu](x,t)p^u_{s,t}(\cdot,x)](x') - d(x')D^{m/2}[\nu](x,t)p^{u}_{s,t}(x',x) \right] u_s(x')dx'ds.
                \end{array}
            \end{displaymath}
            Then, it follows from Lemma \ref{appendBes} that for any $0<\alpha<1$ such that $m>3\alpha$,
            \begin{displaymath}
                \begin{array}{l}
                    \displaystyle \left\| D^{m/2}[\nu](\cdot,t)u_t\right\|_{\iB_{1,\infty}^{\frac{2m}{2m+3\alpha}\eta_{\alpha,\beta,m}}(\R^*_+)} \leq \int_0^{\infty}\left\|D^{m/2}[\nu](\cdot,t)p_{0,t}^u(x',\cdot)\right\|_{\iB_{1,\infty}^{\frac{2m}{2m+3\alpha}\eta_{\alpha,\beta,m}}(\R^*_+)} \nu_0(dx') \vspace{0.20cm}
                    \\
                    \displaystyle \hspace{2cm} +~\int_0^t\int_0^{\infty}\bigg\{ b(x',R_s)\bigg[ \left\|D^{m/2}[\nu](\cdot,t)p^u_{s,t}(x',\cdot)\right\|_{\iB_{1,\infty}^{\frac{2m}{2m+3\alpha}\eta_{\alpha,\beta,m}}(\R^*_+)} \vspace{0.20cm}
                    \\
                    \displaystyle \hspace{4cm} +~ 2\int_0^1\left\| D^{m/2}[\nu](\cdot,t)p^u_{s,t}(\alpha' x',\cdot)\right\|_{\iB_{1,\infty}^{\frac{2m}{2m+3\alpha}\eta_{\alpha,\beta,m}}(\R^*_+)}M(d\alpha') \bigg] \vspace{0.20cm} 
                    \\
                    \displaystyle \hspace{4.5cm} +~ d(x')\left\|D^{m/2}[\nu](\cdot,t)p^{u}_{s,t}(x',\cdot)\right\|_{\iB_{1,\infty}^{\frac{2m}{2m+3\alpha}\eta_{\alpha,\beta,m}}(\R^*_+)} \bigg\} u_s(x')dx'ds.
                \end{array}
            \end{displaymath}
            Using (\ref{besovdens}), this implies that 
            \begin{displaymath}
            \begin{array}{l}
                \displaystyle \left\| D^{m/2}[\nu](\cdot,t)u_t\right\|_{\iB_{1,\infty}^{\frac{2m}{2m+3\alpha}\eta_{\alpha,\beta,m}}(\R^*_+)} \leq C\bigg\{t^{-m/2}\int_0^{\infty}\big( 1+x^{c_{\alpha,k}+m/2} \big) \nu_0(dx) 
                \\
                \displaystyle \hspace{7cm} +\, \int_0^t(t-s)^{-m/2}\int_0^{\infty} \big( 1+x^{c_{\alpha,k}+m/2} \big) u_s(x)dxds \bigg\}.
            \end{array}
            \end{displaymath}
            We take $m=1$ in order to ensure that $\int_0^t(t-s)^{-m/2}ds<\infty$. Then we have the condition $0<\alpha<1/3$. Let us also notice that under the assumptions of Theorems \ref{uniqone} and \ref{uniqtwo}, we have $k=0$, $\beta=1$ and then $c_{\alpha,0} = \alpha$, $\eta_{\alpha,1,1} = \frac{1-3\alpha}{2/\alpha}$. It follows that
            \begin{displaymath}
            \begin{array}{l}
                \displaystyle \left\| D^{1/2}[\nu](\cdot,t)u_t\right\|_{\iB_{1,\infty}^{\frac{\alpha(1-3\alpha)}{2+3\alpha}}(\R^*_+)} \leq C\bigg\{t^{-1/2}\int_0^{\infty}\big( 1+x^{\alpha+1/2} \big) \nu_0(dx) 
                \\
                \displaystyle \hspace{7cm} +\, \int_0^t(t-s)^{-1/2}\int_0^{\infty} \big( 1+x^{\alpha+1/2} \big) u_s(x)dxds \bigg\} \vspace{0.20cm}
                \\
                \displaystyle \hspace{4.75cm} \leq C\bigg\{ t^{-1/2}\int_0^{\infty}\big( 1+x^{\alpha+1/2} \big) \nu_0(dx) \vspace{0.20cm} 
                \\
                \displaystyle \hspace{7cm} +\, t^{1/2}\sup_{0<s\leq T}\int_0^{\infty} \big( 1+x^{\alpha+1/2} \big) u_s(x)dx \bigg\} \vspace{0.20cm}
                \\
                \displaystyle \hspace{4.75cm} \leq Ct^{-1/2}\bigg\{ \int_0^{\infty}\big( 1+x \big) \nu_0(dx) + t\sup_{0<s\leq T}\int_0^{\infty} \big( 1+x \big) u_s(x)dx \bigg\}.
            \end{array}
            \end{displaymath}
            Since $0<\alpha<1/3$, one can replace $\frac{\alpha(1-3\alpha)}{2+3\alpha}$ in the left term of this inequality by $\lambda$ chosen arbitrarily in the image set of $]0,1/3[$ by the function $\alpha\mapsto \frac{\alpha(1-3\alpha)}{2+3\alpha}$ that is $(0,\frac{5-2\sqrt{6}}{3}]$. It suffices now to show that the supremum in the right term of the last inequality here above is finite in order to conclude that the bound (\ref{ineqbesov}) holds. For that purpose, one can assume that $\sup_{K}\E\big[ \langle\nu^K_0,1+x\rangle \big] < \infty$, it suffices to take the initial condition $\nu^K_0 = \frac{1}{K}\sum_{i\in V^K_0}\delta_{x^i}$ with $\# V^K_0=\lfloor\langle\nu_0,1\rangle K\rfloor$ and $x^1, ..., x^{\# V^K_0}$ i.i.d with law $\frac{1}{\langle\nu_0,1\rangle}\nu_0(dx)$. From Lemma \ref{tailK}, we deduce that
            \[ \sup_{K}\E\bigg[ \sup_{0\leq t\leq T}\big\langle\nu^K_t,1+x\big\rangle \bigg] < \infty. \]
            That implies thanks to the Fatou Lemma that if $\psi(x)$ is an increasing function with a compact support such that $\psi\equiv 1$ on $[0,1/2]$ and $\psi\equiv 0$ on $[1,\infty)$, then for any $n\geq 1$
            \begin{displaymath}
                \begin{array}{l}
                    \displaystyle \int_0^{\infty}(1+x)\psi\bigg(\frac{x}{n}\bigg)u_s(x)dx \leq \liminf_{K\rightarrow\infty}\E\bigg[ \int_0^{\infty}(1+x)\psi\bigg(\frac{x}{n}\bigg)\nu^K_s(dx) \bigg] \vspace{0.20cm}
                    \\
                    \displaystyle \hspace{4.45cm} \leq\, \sup_K\E\bigg[ \sup_{0\leq t\leq T}\big\langle\nu^K_t,1+x\big\rangle \bigg] < \infty.
                \end{array}
            \end{displaymath}
            We conclude by taking the supremum on $s\in(0,T]$ after applying the monotone convergence theorem as $n\rightarrow\infty$, that
            \[ \sup_{0< s\leq T}\int_0^{\infty}(1+x)u_s(x)dx < \infty \]
            which ends the proof.
        \end{proof}
    
    \bigskip
    
    \section*{Acknowledgement}
        I would like to thank Sylvie M\'el\'eard and Carl Graham for their continual guidance during this work. I also thank Sylvie M\'el\'eard for fruitful exchanges and discussions. This work has been supported by the Chair “Modélisation Mathématique et Biodiversité” of Veolia-École Polytechnique-Muséum national d’Histoire naturelle-Fondation X.

%
%
%
\bibliographystyle{plain}
\bibliography{biblio}

\begin{thebibliography}{10}

\bibitem{ald}
David Aldous.
\newblock Stopping times and tightness.
\newblock {\em Ann. Probab.}, 6:335--340, 1978.

\bibitem{sylVin}
Vincent Bansaye and Sylvie M\'el\'eard.
\newblock {\em Stochastic Models for Structured Populations: Scaling Limits and
  Long Time Behavior}.
\newblock Springer, 2015.

\bibitem{vinChi}
Vincent Bansaye and Viet~Chi Tran.
\newblock Branching feller diffusion for cell division with parasite infection.
\newblock {\em ALEA, Latin American Journal of Probability and Mathematical
  Statistics}, 8:95--127, 2011.

\bibitem{gfrag1}
Jean Bertoin.
\newblock The asymptotic behavior of fragmentation processes.
\newblock {\em J. Eur. Math. Soc. (JEMS)}, 5(4):395--416, 2003.

\bibitem{gfrag2}
Jean Bertoin.
\newblock On a {F}eynman-{K}ac approach to growth-fragmentation semigroups and
  their asymptotic behaviors.
\newblock {\em J. Funct. Anal.}, 277(11), 2019.

\bibitem{gfrag3}
Jean Bertoin and Alexander~R. Watson.
\newblock Probabilistic aspects of critical growth-fragmentation equations.
\newblock {\em Adv. Appl. Prob.}, 48:37--61, 2016.

\bibitem{gfrag4}
Jean Bertoin and Alexander~R. Watson.
\newblock A probabilistic approach to spectral analysis of growth-fragmentation
  equations.
\newblock {\em J. Func. Anal.}, 274(8):2163--2204, 2018.

\bibitem{boulHir}
Nicolas Bouleau and Francis Hirsch.
\newblock {\em Dirichlet Forms and Analysis on Wiener Space}.
\newblock de Gruyter Studies in Mathematics 14, 1991.

\bibitem{camfritsch}
Fabien Campillo and Coralie Fritsch.
\newblock Weak convergence of a mass-structured individual-based model.
\newblock {\em Appl Math Optim}, 72:37--73, 2015.

\bibitem{chamMelchem}
Nicolas Champagnat, Pierre-Emmanuel Jabin, and Sylvie M\'el\'eard.
\newblock Adaptation in a stochastic multi--ressources chemostat model.
\newblock {\em Journal de Math\'ematiques Pures et Appliqu\'ees},
  101(6):755--788, 2014.

\bibitem{chamMel}
Nicolas Champagnat and Sylvie M\'el\'eard.
\newblock Invasion and adaptative evolution for individual-based spatially
  structured populations.
\newblock {\em Journal of Mathematical Biology}, 55:147--188, 2007.

\bibitem{sylCol}
Pierre Collet, Servet Mart\'inez, and Sylvie M\'el\'eard.
\newblock Stochastic models for a chemostat and long-time behavior.
\newblock {\em Adv. Appl. Prob.}, 45:822--836, 2013.

\bibitem{gfrag5}
Tomasz D\c{e}biec, Marie Doumic, Piotr Gwiazda, and Emil Wiedemann.
\newblock Relative entropy method for measure solutions of the
  growth-fragmentation equation.
\newblock {\em SIAM J. Math. Anal.}, 50(6):5811--5824, 2018.

\bibitem{debromito}
Arnaud Debussche and Marco Romito.
\newblock Existence of densities for the 3{D} {N}avier–{S}tokes equations
  driven by {G}aussian noise.
\newblock {\em Probability Theory and Related Fields}, 158:575--596, 2014.

\bibitem{gfrag6}
Marie Doumic and Miguel Escobedo.
\newblock Time asymptotics for a critical case in fragmentation and
  growth-fragmentation equations.
\newblock {\em Kinet. Relat. Models}, 9(2):251--297, 2016.

\bibitem{joaMel}
Joaquin Fontbona and Sylvie M\'el\'eard.
\newblock Non local {L}otka-{V}olterra system with cross-diffusion in an
  heterogeneous medium.
\newblock {\em Journal of Mathematical Biology}, 70(4):829--854, 2015.

\bibitem{fourMel}
Nicolas Fournier and Sylvie M\'el\'eard.
\newblock A microscopic probabilistic description of a locally regulated
  population and macroscopic approximation.
\newblock {\em The Annals of Applied Probability}, 14(4):1880--1919, 2004.

\bibitem{fried75}
Avner Friedman.
\newblock {\em Differential Equations and Applications}, volume~1.
\newblock Academic Press, Inc. (LONDON) LTD., 1975.

\bibitem{jerom}
J\'er\^ome Harmand, Claude Lobry, Alain Rapaport, and Tewfik Sari.
\newblock {\em Le Ch\'emostat : Th\'eorie math\'ematique de la culture de
  micro--organismes}, volume~1.
\newblock 2017.

\bibitem{ikWa}
Nobuyuki Ikeda and Shinzo Watanabe.
\newblock {\em Stochastic Differential Equations and Diffusion Processes}.
\newblock North-Holland Mathematical Library, 1988.

\bibitem{gfrag7}
Marie~Doumic Jauffret and Pierre Gabriel.
\newblock Eigen elements of a general aggregation-fragmentation model.
\newblock {\em Math. Models Methods Appl. Sci.}, 20(5):757--783, 2010.

\bibitem{bjour}
Benjamin Jourdain, Sylvie M\'el\'eard, and Wojbor~A. Woyczynski.
\newblock L\'evy flights in evolutionary ecology.
\newblock {\em Journal of Mathematical Biology}, 65:677--707, 2012.

\bibitem{lions}
Pierre-Louis Lions and Marek Musiela.
\newblock Convexity of solutions of parabolic equations.
\newblock {\em C. R. Acad. Sci. Paris}, 342:915--921, 2006.

\bibitem{melroel}
Sylvie M\'el\'eard and Sylvie Roelly.
\newblock Sur les convergences \'etroite ou vague de processus \`a valeurs
  mesures.
\newblock {\em C. R. Acad. Sci. Paris S\'er. I Math.}, 317:785--788, 1993.

\bibitem{chem2}
Jacques Monod.
\newblock La technique de culture continue, th\'eorie et applications.
\newblock {\em Annales de l'Institut Pasteur}, 79(4):390--410, 1950.

\bibitem{chem1}
Aaron Novick and Leo Szilard.
\newblock Description of the chemostat.
\newblock {\em Science}, 112(2920):715--716, 1950.

\bibitem{revYor}
Daniel Revuz and Marc Yor.
\newblock {\em Continuous Martingales and Brownian Motion}.
\newblock Springer, 1998.

\bibitem{roel}
Sylvie Roelly-Coppoletta.
\newblock A criterion of convergence of measure--valued processes : application
  to measure branching processes.
\newblock {\em Stoch.Stoch. Rep.}, 17:43--65, 1986.

\bibitem{romito}
Marco Romito.
\newblock A simple method for the existence of a density for stochastic
  evolutions with rough coefficients.
\newblock {\em Electronic Journal of Probability}, 23, 2018.

\bibitem{sato}
K.~Sato and T.~Ueno.
\newblock Multi-dimensional diffusion and the markov process on the boundary.
\newblock 4(3):529--605, 1965.

\bibitem{chiTran}
Viet~Chi Tran.
\newblock Large population limit and time behaviour of a stochastic particle
  model describing an age-structured population.
\newblock {\em ESAIM: Probability and Statistics, EDP Sciences}, 12:345--386,
  2008.

\bibitem{trieb83}
Hans Triebel.
\newblock {\em Theory of function spaces}.
\newblock Monographs in Mathematics, vol. 78, Birkh\"auser Verlag, Basel, 1983.

\bibitem{trieb92}
Hans Triebel.
\newblock {\em Theory of function spaces II}.
\newblock Monographs in Mathematics, vol. 84, Birkh\"auser Verlag, Basel, 1992.

\bibitem{yaWaII}
Toshio Yamada and Shinzo Watanabe.
\newblock On the uniqueness of solutions of stochastic differential equations.
\newblock {\em J. Math. Kyoto Univ. (JMKYAZ) 11-3}, pages 553--563, 1971.

\end{thebibliography}

    \bigskip

    \section*{Appendix}
    \appendix
    \section{Proof of Theorem \ref{exist}}\label{appA}
        \subsection{Step 1 : Control of tails}
        Since the individual trait is unbounded, we need to get a control on the measure valued process tails. We then have the following lemma that is adapted from \cite{bjour}.
            \begin{lemma}\label{fn}
                Let $T>0$, if a subsequence of $\left\{ \left(\nu^K_t\right)_{t\geq 0},K>1 \right\}$ converges in law in the Skorohod space $\D([0,T],(\iM_F(\R_+),v))$ then its limit $(\nu_t)_{t\in [0,T]}$ satisfies
                \begin{equation}\label{contr}
                    \E\left\{ \sup_{0\leq t\leq T}\left[\left\langle\nu_t,1\right\rangle^{1+\varrho} \right] \right\} < \infty.
                \end{equation}
                In addition, there exists a decreasing sequence of inscreasing functions $(f_n)_{n\geq 1}\subset\iC_b^2(\R_+,\R)$ such that $f_n=0$ on $[0,n/2]$ and $f_n = 1$ on $[n,\infty)$, which satisfies
                \begin{equation}\label{limfn}
                    \lim_{n\rightarrow\infty}\limsup_{K\rightarrow\infty}\E\left\{ \sup_{0\leq t\leq T}\left\langle\nu^K_t,f_n\right\rangle \right\} = 0.
                \end{equation}
            \end{lemma}
            \begin{proof}[Proof of Lemma \ref{fn}]
                In order to show (\ref{contr}), we approximate the functions $x\mapsto 1$ by an increasing sequence of continuous and compact supported functions in order to use the convergence assumption for the vague topology. We then introduce the uniformly bounded sequence of continuous functions defined for $n\in\N^{\star}$ by
                \begin{displaymath}\phi_n(x) = \left\{\begin{array}{ll}
                    \displaystyle 1 &\textrm{ if }0\leq x < n-1, \vspace{0.20cm}
                    \\
                    \displaystyle n-x &\textrm{ if }n-1\leq x < n, \vspace{0.20cm}
                    \\
                    \displaystyle 0 &\textrm{ if }x\geq n.
                \end{array}\right.
                \end{displaymath}
                that converges increasingly pointwise towards the constant function $x\mapsto 1$. It follows from the Fatou lemma that holds thanks to the convergence of the sequence of measure valued process for the vague topology, that we have for all $n\in\N$
                \begin{displaymath}\begin{array}{l}
                    \displaystyle \E\left\{ \sup_{0\leq t\leq T}\left\langle\nu_t,\phi_n\right\rangle^{1+\varrho} \right\} \leq \liminf_{K\rightarrow\infty}\E\left\{ \sup_{0\leq t\leq T}\left\langle\nu^K_t,\phi_n\right\rangle^{1+\varrho} \right\} \leq \liminf_{K\rightarrow\infty}\E\left\{ \sup_{0\leq t\leq T}\left\langle\nu^K_t,1\right\rangle^{1+\varrho} \right\}.
                \end{array}
                \end{displaymath}
                As $n\rightarrow\infty$ in this inequality, we obtain by monotony
                \begin{displaymath}
                    \E\left\{ \sup_{0\leq t\leq T}\left\langle\nu_t,1\right\rangle^{1+\varrho} \right\} \leq \sup_K\E\left\{ \sup_{0\leq t\leq T}\left\langle\nu^K_t,1\right\rangle^{1+\varrho} \right\} <\infty.
                \end{displaymath}
                In order to show (\ref{limfn}), we approximate the indicator of the outside of the compacts by a non-increasing and uniformly bounded sequence of regular fonctions. This allows us to control the tails of the distributions $\{\iQ^K,K>1\}$. Let us then introduce an increasing function $f\in\iC^2(\R_+,\R)$ such that $f\equiv 0$ on $[0,1/2]$ and $f\equiv 1$ on $[1,\infty)$. We set
                \begin{equation}\label{funcfn}
                    f_n(x) = f\big( \frac{x}{n} \big)~,~\forall n\geq 1
                \end{equation}
                which are indeed in $\iC^2_b(\R_+,\R)$ with uniformly bounded derivatives. It then follows that for all $s\leq t\leq T$
                \begin{displaymath}\begin{array}{l}
                    \displaystyle \left\langle\nu^K_s,f_n\right\rangle \leq \left\langle\nu^K_0,f_n\right\rangle + M^{K,f_n}_s + \int_0^s\int_0^{\infty}\left\{ \zeta(x,R^K_u)f'_n(x) + D(x,R^K_s)f''_n(x) \right\}\nu^K_u(dx)du \vspace{0.15cm}
                    \\
                    \displaystyle \qquad\qquad~ + 2\int_0^s\int_0^{\infty}\left\{b(x,R^K_u)\int_0^1f_n(\alpha x)M(d\alpha) \right\}\nu^K_u(dx)du \vspace{0.20cm}
                    \\
                    \displaystyle \qquad\qquad \leq \left\langle\nu^K_0,f_n\right\rangle + A^{K,f_n}_T + 2\|b\|_{\infty}\int_0^s\left\langle\nu^K_u,f_n\right\rangle du + \sup_{0\leq u\leq T}\left|M^{K,f_n}_u\right|
                \end{array}
                \end{displaymath}
                where
                \begin{displaymath}
                    A^{K,f_n}_T := C\int_0^T \left\langle\nu^K_u,(1+x)(f'_n + \left|f''_n\right|)\right\rangle du.
                \end{displaymath}
                The sequence $(A^{K,f_n}_T)_{K>1}$ is uniformly integrable according to $K$ and converges in law towards the random variable $C\int_0^T \left\langle\nu_u,(1+x)(f'_n + \left|f''_n\right|)\right\rangle du$ as $K\rightarrow\infty$ because the function $x\mapsto (1+x)(f'_n + \left|f''_n\right|)(x)$ is continuous with compact support. We have
                \begin{displaymath}
                    \E\left\{\sup_{0\leq s\leq t}\left\langle\nu^K_s,f_n\right\rangle\right\} \leq \E\left\{\left\langle\nu^K_0,f_n\right\rangle + A^{K,f_n}_T + \sup_{0\leq u\leq T}\left|M^{K,f_n}_u\right| \right\} + 2\|b\|_{\infty}\int_0^t\E\left\{\sup_{0\leq u\leq s}\left\langle\nu^K_u,f_n\right\rangle\right\} ds
                \end{displaymath}
                that induces thanks to Gronwall's lemma that
                \begin{displaymath}
                    \E\left\{\sup_{0\leq s\leq t}\left\langle\nu^K_s,f_n\right\rangle\right\} \leq \E\left\{\left\langle\nu^K_0,f_n\right\rangle + A^{K,f_n}_T + \sup_{0\leq u\leq T}\left|M^{K,f_n}_u\right| \right\}e^{2\|b\|_{\infty}t}.
                \end{displaymath}
                As $K\rightarrow\infty$, it follows from (\ref{diff}) and the Doob inequality that the right term with the martingale converges towards $0$. In addition the remaining terms in the expectation are uniformly integrable and converge in law, then
                \begin{displaymath}
                \begin{array}{l}
                    \displaystyle \lim_{n\rightarrow\infty}\limsup_{K\rightarrow\infty}\E\left\{\sup_{0\leq s\leq T}\left\langle\nu^K_s,f_n\right\rangle\right\} 
                    \\
                    \displaystyle \hspace{1cm} \leq e^{2\|b\|_{\infty}T}\lim_{n\rightarrow\infty}\E\left\{ \left\langle\nu_0,f_n\right\rangle + C\int_0^T\int_0^{\infty}\frac{(1+x)}{n}\bigg(f'\big(\frac{x}{n}\big) + \frac{1}{n}\big|f''\big|\big(\frac{x}{n}\big) \bigg)\nu_u(dx) du \right\} = 0.
                \end{array}
                \end{displaymath}
            \end{proof}
            
        \subsection{Step 2 : Tightness}
            Given the above Lemma \ref{fn}, the tightness of the sequence of laws $\left\{\iQ^K = \iL(\nu^K,R^K),K>1 \right\}$ follows. It is adapted from \cite{sylVin}, Theorem 7.4.
            \begin{prop}\label{tight}
                The sequence of laws $\left\{\iQ^K = \iL(\nu^K,R^K),K>1 \right\}$ is tight in the space of probability measures $\iP(\D([0,T],\etr))$ for all $T>0$.
            \end{prop}
            \begin{proof}
                Let us first endow the measure space $\iM_F(\R_+)$ with its vague topology and prove the tightness property for the sequence $\left\{ \iQ^K=\iL\left(\nu^K, R^K\right),K>1 \right\}$. By noticing that $\iC^2_c(\R_+,\R)$ is a dense subspace of $\iC_0(\R_+,\R)$ for the topology of uniform convergence on compact sets, it suffices according to \cite{roel} to show that for each test function $f\in\iC_c^2(\R_+,\R)$ the sequence of processes $\left\{ \left(\left\langle\nu^K_t,f\right\rangle, R^K_t\right)_{t\geq 0},K>1 \right\}$ is tight in the Skorohod space $\D([0,T],\R^2)$. First, remember that we have the decomposition
                \[ \left\langle\nu^K_t,f\right\rangle = \left\langle\nu^K_0,f\right\rangle + V^{K,f}_t + M^{K,f}_t \]
                and thanks to (\ref{diff}),
                \[ \lim_{K\rightarrow\infty}\E\left\langle M^{K,f}\right\rangle_t = 0 \, , \, \forall t\geq 0. \]
                It then directly follows from Doob's inequality that the sequence of martingales $\left\{ M^{K,f},K>1 \right\}$ converges in $\eL^2$ toward $0$ locally uniformly in time. This sequence is then tight in the Skorohod space $\D([0,T],\R)$, and all we need to prove now is the tightness of the other sequences $\{V^{K,f},K>1\}$ and $\{R^K,K>1\}$. For that purpose, we use the Aldous criterion (see \cite{ald}), that consists in showing the uniform controls
                \begin{displaymath}\begin{array}{l}
                    \displaystyle \textrm{(a) } \sup_K\E\left[ \sup_{0\leq t\leq T}\left| V^{K,f}_t \right| \right] < \infty\, , ~~ \textrm{(b) } \sup_K\E\left[ \sup_{0\leq t\leq T}\left|R^K_t\right| \right] < \infty
                \end{array}
                \end{displaymath}
                and showing that the Aldous condition holds for $(R^K_t,V^{K,f}_t)_{t\geq 0}$. The assertions (a) and (b) are immediate. Indeed, $R^K_t \leq r\vee R_0 = \bar{R}$ for any $t\geq 0$ and thanks to the boundedness of the test function,
                \begin{displaymath}
                    \sup_{0\leq t\leq T} \left|V^{K,f}_t\right| \leq C_f\sup_{0\leq t\leq T}\left[\left\langle\nu^K_t,1\right\rangle + \left|M^{K,f}_t\right| \right].
                \end{displaymath}
                Let us now consider $\varepsilon>0$ and a stopping time $\tau$ that satisfies $\tau+\epsilon\leq T$, then 
                \begin{displaymath}\begin{array}{l}
                    \displaystyle \E\left(\left| R^K_{\tau+\varepsilon}-R^K_{\tau} \right|\right) \leq \E\left(\int_{\tau}^{\tau+\varepsilon}\left|r_{in}-R^K_s -  \int_0^{\infty}\chi(x,R^K_s)\nu^K_s(dx)\right|ds\right) \vspace{0.20cm}
                    \\
                    \displaystyle \qquad\qquad\qquad\quad~~ \leq \varepsilon\left[r_{in}+\bar{R} + \|\chi\|_{\infty}\E\left(\sup_{0\leq t\leq T}\left\langle\nu^K_t,1\right\rangle\right)\right].
                \end{array}
                \end{displaymath}
                By a similar computation, we obtain
                \begin{displaymath}
                \E\left(\left|V^{K,f}_{\tau+\varepsilon}-V^{K,f}_{\tau}\right|\right) \leq C(f)\varepsilon\E\left[\sup_{0\leq t\leq T}\left\langle\nu^K_t,1\right\rangle\right],
                \end{displaymath}
                that implies the Aldous condition thanks to Lemma \ref{tailK}. The sequence of law $(\iQ^K)_K$ is then tight and by the Prokhorov's theorem one can extract from each sub-sequence, a sub-sub-sequence that converges in $\iP(\D([0,T],\vag))$. Let us denote by $\iQ$ a limit value of this sequence, and for simplicity of the notation, $(\iQ^K)_K$ a sub-sequence that converges towards $\iQ$. For each $K>1$ we have
                \begin{equation}
                    \sup_{0\leq s\leq T}\sup_{f\in L^{\infty}(\R_+),\|f\|_{\infty}\leq 1}\left| \left\langle\nu^K_{t},f\right\rangle - \left\langle\nu^K_{t-},f\right\rangle \right| \leq 1/K.
                \end{equation}
                We then deduce from the continuity of the mapping $\mu\mapsto\sup_{0\leq t\leq T}\left| \left\langle\mu_{t},f\right\rangle - \left\langle\mu_{t-},f\right\rangle \right|$ for $f\in\iC_c(\iE\times\R_+,\R)$, that the limiting law $\iQ$ only charges the set $\iC([0,T],\vag)$. We are aiming to prove the same result for the weak topology on $\iM_F(\R_+)$ and this is where the sequence of functions introduced in Lemma \ref{fn} is usefull.
                
                Let us denote by $(\nu,R)$ a process of law $\iQ$, then according to the above argument its trajectories are in the space $\iC([0,T],\vag)$. As in \cite{bjour}, it follows from the Fatou lemma that for $n,l\in\N$,
                \begin{displaymath}
                    \E\left[ \sup_{0\leq t\leq T}\left\langle\nu_t,(1-f_l)f_n\right\rangle \right] \leq \liminf_{K\rightarrow\infty}\E\left[ \sup_{0\leq t\leq T}\left\langle\nu^K_t,(1-f_l)f_n\right\rangle \right] \leq \limsup_{K\rightarrow\infty}\E\left[ \sup_{0\leq t\leq T}\left\langle\nu_t^K,f_n\right\rangle \right].
                \end{displaymath}
                As $l\rightarrow\infty$ it follows from Beppo-Levi's theorem applied to the left term, that
                \begin{displaymath}
                    \E\left[ \sup_{0\leq t\leq T}\left\langle\nu_t,f_n\right\rangle \right] \leq \limsup_{K\rightarrow\infty}\E\left[ \sup_{0\leq t\leq T}\left\langle\nu_t^K,f_n\right\rangle \right]
                \end{displaymath}
                and then by the Lemma \ref{fn},
                \begin{displaymath}
                    \lim_{n\rightarrow\infty}\E\left[ \sup_{0\leq t\leq T}\left\langle\nu_t,f_n\right\rangle \right] \leq \lim_{n\rightarrow\infty}\limsup_{K\rightarrow\infty}\E\left[ \sup_{0\leq t\leq T}\left\langle\nu_t^K,f_n\right\rangle \right] = 0.
                \end{displaymath}
                One can therefore extract a sub-sequence $\left(\sup_{0\leq t\leq T}\langle\nu_t,f_{n_k}\rangle\right)_k$ that converges almost surely towards $0$, and then the limit law $\iQ$ only charges the set $\iC([0,T],\etr)$. We deduce that the convergence of $\{\iQ^K,K>1\}$ towards $\iQ$ for the weak topology on $\iM_F(\R_+)$ holds if the sequence of total masses $\left\{\left\langle\nu^K,1\right\rangle,K>1\right\}$ converges in law towards $\left\langle\nu,1\right\rangle$ in the Skorohod space $\D([0,T],\R)$ (see \cite{melroel}). So let us consider a Lipschitz continuous function $\Phi : \D([0,T],\R)\rightarrow\R$, and $n\in\N$.
                \begin{displaymath}\begin{array}{l}
                    \displaystyle \left|\E\left[ \Phi\left(\left\langle\nu^K,1\right\rangle\right) - \Phi\left(\left\langle\nu,1\right\rangle\right)\right] \right| \leq \E\left| \Phi\left(\left\langle\nu^K,1\right\rangle\right) - \Phi\left(\left\langle\nu^K,1-f_n\right\rangle\right)\right| \vspace{0.20cm}
                    \\
                    \displaystyle \hspace{5cm} +~\left|\E\left[ \Phi\left(\left\langle\nu^K,1-f_n\right\rangle\right) - \Phi\left(\left\langle\nu,1-f_n\right\rangle\right)\right] \right| \vspace{0.20cm} 
                    \\
                    \displaystyle \hspace{5cm} +\, \E\left| \Phi\left(\left\langle\nu,1-f_n\right\rangle\right) - \Phi\left(\left\langle\nu,1\right\rangle\right) \right| \vspace{0.20cm}
                    \\
                    \displaystyle \hspace{4.8cm} \leq C_{\Phi}\E\left\{ \sup_{0\leq t\leq T}\left\langle\nu^K_t,f_n\right\rangle + \sup_{0\leq t\leq T}\left\langle\nu_t,f_n\right\rangle \right\} \vspace{0.20cm}
                    \\
                    \displaystyle \hspace{5cm} +\, \left|\E\left[ \Phi\left(\left\langle\nu^K,1-f_n\right\rangle\right) - \Phi\left(\left\langle\nu,1-f_n\right\rangle\right)\right] \right| \vspace{0.20cm}
                \end{array}
                \end{displaymath}
                The function $1-f_n$ being continuous with a compact support and the sequence $\{\nu^K,K>1\}$ converging in law toward the limit process $\nu$ in $\D([0,T],\vag)$, we have
                \begin{displaymath}
                    \lim_{K\rightarrow\infty}\left|\E\left[ \Phi\left(\left\langle\nu^K,1-f_n\right\rangle\right) - \Phi\left(\left\langle\nu,1-f_n\right\rangle\right)\right] \right| = 0.
                \end{displaymath}
                It follows that for all $n\in\N$,
                \begin{displaymath}
                \begin{array}{l}
                    \displaystyle \limsup_{K\rightarrow\infty}\left|\E\left[ \Phi\left(\left\langle\nu^K,1\right\rangle\right) - \Phi\left(\left\langle\nu,1\right\rangle\right)\right] \right| 
                    \\
                    \displaystyle \hspace{3cm}\leq C_{\Phi}\left[\limsup_{K\rightarrow\infty}\E\left\{ \sup_{0\leq t\leq T}\left\langle\nu^K_t,f_n\right\rangle\right\} + \E\left\{\sup_{0\leq t\leq T}\left\langle\nu_t,f_n\right\rangle \right\}\right].
                \end{array}
                \end{displaymath}
                Thanks to Lemma \ref{fn}, that implies that as $n\rightarrow\infty$ in the right term,
                \begin{displaymath}
                    \limsup_{K\rightarrow\infty}\left|\E\left[ \Phi\left(\left\langle\nu^K,1\right\rangle\right) - \Phi\left(\left\langle\nu,1\right\rangle\right)\right] \right| = 0.
                \end{displaymath}
                The sub-sequence $\{(\langle\nu^K_t,1\rangle)_{t\in [0,T]},K>1 \}$ then converges in law towards the process $(\langle\nu_t,1\rangle)_{t\in [0,T]}$, and then the subsequence $(\iQ^K)_K$ converges towards $\iQ$ in the space of probability distributions $\iP(\D([0,T],\etr))$. That is enough to conclude that our initial sequence of laws is tight.
            \end{proof}
            
        \subsection{Step 3 : Identification of a limiting process}
            We are here interested in the characterization of a process $(\nu_t,R_t)_{t\in [0,T]}$ whose law is a limiting value of the sequence $\{\iQ^K,K>1\}$. Given the operator defined in (\ref{fragop}), we have the following result
            \begin{prop}
                Let us denote by $\iQ$ a limiting value of the sequence of distributions $\left\{\iQ^K,K>1\right\}$. Then a process $(\nu,R)$ of law $\iQ$ satisfies (\ref{determ}).
            \end{prop}
            \begin{proof}
                Let $f\in\iC_c^2(\R_+,\R)$ be a test function and $t\in [0,T]$ a fixed time. We introduce the following real-valued functions defined for $(\mu,r)$ in the subset of $\iC([0,T],\etr)$ constitued of processes that satisfy $|r| \leq \bar{R}$
                \begin{displaymath}\begin{array}{l}
                    \displaystyle \Phi^f_t(\mu,r) = \left\langle\mu_t,f\right\rangle - \left\langle\mu_0,f\right\rangle - \int_0^t\int_0^{\infty}\left\{ \zeta(x,r_s)f'(x) + D(x,r_s)f''(x) \right\}\mu_s(dx)ds \vspace{0.20cm}
                    \\
                    \displaystyle \qquad\qquad -~\int_0^t\int_0^{\infty}\big\{ b(x,r_s)\iG[f](x) - d(x)f(x) \big\}\mu_s(dx)ds \vspace{0.20cm}
                    \\
                    \displaystyle \Phi_{t}(\mu,r) = r_t - r_0 - \int_0^t\left\{r_{in}-r_s - \int_0^{\infty}\chi(x,r_s)\mu_s(dx)\right\}ds.
                \end{array}
                \end{displaymath}
                We are aiming to show that 
                \begin{equation}
                    \E\left[\left|\Phi^f_t(\nu,R)\right|^2 + \left|\Phi_{t}(\nu,R)\right| \right] = 0.
                \end{equation}
                For this purpose, we first note that the function $\Phi_{t}$ defined as above is continuous and 
                \begin{displaymath}
                    \Phi_{t}\left(\nu^K,R^K\right) = 0.
                \end{displaymath}
                It then follows by the Fatou lemma that 
                \begin{displaymath}
                    \E\left|\Phi_{t}\left(\nu,R\right)\right| \leq \liminf_{K\rightarrow\infty}\E\left|\Phi_{t}\left(\nu^K,R^K\right)\right| = 0.
                \end{displaymath}
                Futhermore, the terms in the integrals in the expression of the function $\Phi^f_t$ are continuous and bounded, then this function is continuous. In addition, it follows from the decomposition (\ref{doobM}) that
                \begin{displaymath}
                    \Phi^f_t\left(\nu^K,R^K\right) = M^{K,f}_t
                \end{displaymath}
                and then thanks to the Fatou lemma and Doob's inequality,
                \begin{displaymath}
                \begin{array}{l}
                    \displaystyle \E\left[\left(\Phi^f_t\left(\nu,R\right)\right)^2 \right] \leq \liminf_{K\rightarrow\infty}\E\left[\left(\Phi^f_t\left(\nu^K,R^K\right)\right)^2 \right] \vspace{0.20cm}
                    \\
                    \displaystyle \hspace{2.7cm} \leq \liminf_{K\rightarrow\infty}\E\left[ \sup_{0\leq s\leq t}\left|M^{K,f}_s\right|^2 \right] \leq C \liminf_{K\rightarrow\infty}\E\left[\left\langle M^{K,f}\right\rangle_t\right].
                \end{array}
                \end{displaymath}
                We now deduce from (\ref{diff}) that
                \begin{equation}\label{comp}
                    \E\left[\left(\Phi^f_t\left(\nu,R\right)\right)^2 \right] \leq \liminf_{K\rightarrow\infty} \frac{C}{K}\E\left[\sup_{0\leq t\leq T}\left\langle\nu^K_t,1\right\rangle\right] = 0.
                \end{equation}
            \end{proof}
    
    \section{Some properties of a specific SDE}\label{appB}
        Let us introduce two continuous functions $h(x,s),q(x,s)\in\iC(\R_+\times[0,T],\R)$ that satisfy the following assumption.
        \begin{assume}\label{hypapp}
            \item[\quad\textsc{(A.1)}] Boundary condition : for all $s\leq [0, T]$,
            \[ q(0,s) = 0 \textrm{ and }h(0,s) \geq 0. \]
                
            \item[\quad\textsc{(A.2)}] Lipschitz condition : for all $x,y\geq 0$ and $s\in [0,T]$,
            \[ |h(x,s) - h(y,s)| + |q(x,s) - q(y,s)| \leq C|x-y|. \]
            
            \item[\quad\textsc{(A.3)}] The function $q(x,s)$ is non negative.
        \end{assume}
        \noindent Then we have the following results
        \begin{lemma}\label{appenBf}
            There is weak existence and uniqueness, and strong existence for the stochastic differential equation
            \begin{equation}\label{edsapp}
                dZ_u = h(Z_u,u)du + \sqrt{q(Z_u,u)}dW_u, \forall u\leq T.
            \end{equation}
            The solution $(Z^x_{s,u})_{u\in[s,T]}$ that starts at time $s<T$ at the position $x\geq 0$, has a zero local time at $0$, is almost surely non negative at any time and satifies the moment estimate
            \begin{equation}\label{momest}
                \E\left\{ \sup_{s\leq u\leq T}\left(Z^x_{s,u}\right)^p \right\} \leq C_{p,T}(1+x^p),\, \forall p>0.
            \end{equation}
            In addition, for all $x_1,x_2\geq 0$ and $s_1,s_2\leq T$, we have for all $u\leq T$
            \begin{equation}
                \E\left( \left| Z^{x_1}_{s_1,u} - Z^{x_2}_{s_2,u} \right| \right) \leq C_{T}\left\{|x_1-x_2| + (1+x_1+x_2)|s_1-s_2| + |1+x_1+x_2|^{1/2}|s_1-s_2|^{1/2}\right\},
            \end{equation}
            where $(Z^{x_1}_{s_1,u})_{u\in[s_1,T]}, (Z^{x_2}_{s_2,u})_{u\in[s_2,T]}$ are assumed to be built with the same Brownian motion, and satisfies $Z^{x_j}_{s_j,u} = x_j$ for all $u\leq s_j$ with $j=1,2$.
        \end{lemma}
        
        \begin{lemma}\label{appenBs}
            Let $p(x)\in\iC^1_b(\R_+,\R)$ and $a(x,s)\in\iC_b(\R_+\times[0,T],\R)$ such that
            \[ |a(x,s)-a(y,s)|\leq \bar{a}|x-y|, \, \forall x,y\geq 0, \forall s\leq T, \]
            then for a given $t\leq T$, the function defined by
            \begin{equation}
                w_s(x) = \E\left[ p(Z^x_{s,t})e^{\int_s^ta(Z^x_{s,u},u)du} \right],\, \forall (s,x)\in [0,t]\times\R_+
            \end{equation}
            satisfy for all $s_1,s_2\leq t$ and $x_1,x_2\geq 0$ the following inequality
            \begin{equation}\label{contrg}
            \begin{array}{l}
                \displaystyle \left| w_{s_1}(x_1)-w_{s_2}(x_2) \right| \leq C_T \left(\|p\|_{\infty} + \|p'\|_{\infty}\right)\bigg\{ |x_1-x_2| + (1+x_1)^{1/2}|s_1-s_2|^{1/2} \vspace{0.20cm}
                \\
                \displaystyle \hspace{8.5cm} +\, (1+x_1)|s_1-s_2| \bigg\}.
            \end{array}
            \end{equation}
        \end{lemma}

        \subsection{Proof of Lemma \ref{appenBf}}
            We split this proof into the several steps.
            
            \paragraph{Existence and uniqueness :} According to \cite{ikWa}, Theorem 1.1 it suffices to show that we have weak existence and pathwise uniqueness to conclude that there is weak existence and uniqueness, and strong existence. In addition it follows from the previous paragraph that this solution will be well defined at any time. The stochastic differential equation (\ref{edsapp}) is non homogeneous with continuous coefficients, then the coupled process $(Z_u,u)_{u\geq 0}$ satisfies an homogeneous stochastic differential equation with continuous coefficients. We then have weak existence for each initial condition $x\geq 0$ and initial time $s\in[0,T]$ thanks to \cite{ikWa}, Theorem 2.3, p173. Furthermore, for $x,y\geq 0$ and $s\leq u\leq T$ we have by hypothesis
            \[ \big|h(x,u) - h(y,u)\big| + \big|\sqrt{q(x,u)} - \sqrt{q(y,u)}\big|^2 \leq C |x-y|, \]
            thanks to the inequality $|\sqrt{a}-\sqrt{b}|\leq \sqrt{|a-b|}\, ,\, \forall a,b\geq 0$. We deduce from (\cite{yaWaII}) that there is pathwise uniqueness.
            
            \paragraph{Moment estimate and positiveness :}
            Notice that the solution that we constructed in the previous paragraph is well defined until it becomes negative or explodes. In order to get well defined terms for all times, we introduce the new equation
            \begin{equation}\label{neweds}
                d\tilde{Z}_u = h(|\tilde{Z}_u|,u)du + \sqrt{q(|\tilde{Z}_u|,u)}\,dW_u , \forall u\leq T
            \end{equation}
            for which there is weak existence and pathwise uniqueness by an argument similar to the one in the previous paragraph, but only until explosion time. It follows from the continuity of the function $h(x,s)$ and hypothesis (A.2) of Assumption \ref{hypapp} that
            \[ |h(x,s)| \leq C(1+x), \forall x\geq 0,\forall s\leq T. \]
            For given $x\geq 0$ and $s\leq t$, we consider the sequence of stopping times defined by
            \[ \lambda_N = \inf\{ v\geq s : |\tilde{Z}^x_{s,v}| > N \}\wedge T,\, N>1 \]
            that converges increasingly. Then for $p\geq 2$ and $v\in[s,T]$ we have
            \[ |\tilde{Z}^x_{s,v\wedge\lambda_N}|^p \leq C_p\left\{x^p + \int_s^{v\wedge\lambda_N}|h(|\tilde{Z}^x_{s,u}|,u)|^pdu + \left|\int_s^{v\wedge\lambda_N}\sqrt{q(|\tilde{Z}^x_{s,u}|,u)}\,dW_u\right|^p \right\} \]
            that implies that
            \begin{displaymath}
                \begin{array}{l}
                    \displaystyle \sup_{s\leq u\leq v\wedge\lambda_N}|\tilde{Z}^x_{s,u}|^p \leq C_p\bigg\{x^p + \int_s^{v\wedge\lambda_N}|h(|\tilde{Z}^x_{s,u}|,u)|du + \sup_{s\leq u\leq v}\bigg(\bigg|\int_s^{u\wedge\lambda_N}\sqrt{q(|\tilde{Z}^x_{s,u}|,u)}\,dW_u\bigg|^p\bigg) \bigg\} \vspace{0.30cm}
                    \\
                    \displaystyle \qquad\qquad\qquad~~~~ \leq C_p\bigg\{x^p + \int_s^{v\wedge\lambda_N}\left[1+|\tilde{Z}^x_{s,u}|^p\right]du + \sup_{s\leq u\leq v}\bigg(\bigg|\int_s^{u\wedge\lambda_N}\sqrt{q(|\tilde{Z}^x_{s,u}|,u)}\,dW_u\bigg|^p\bigg) \bigg\}.
                \end{array}
            \end{displaymath}
            It follows from the Doob inequality that
            \begin{displaymath}\begin{array}{l}
                \displaystyle \E\left\{ \sup_{s\leq u\leq v}\left(\left|\int_s^{u\wedge\lambda_N}\sqrt{q(|\tilde{Z}^x_{s,u}|,u)}\,dW_u\right|^p\right) \right\} \leq C_p\E\left\{ \left( \int_s^{v\wedge\lambda_N}q(|\tilde{Z}^x_{s,u}|,u)du \right)^{p/2}  \right\} \vspace{0.20cm}
                \\
                \displaystyle \hspace{7.3cm} \leq C_p\E\left\{ \left( \int_s^{v\wedge\lambda_N}|\tilde{Z}^x_{s,u}|du \right)^{p/2}  \right\} \vspace{0.20cm}
                \\
                \displaystyle \hspace{7.3cm} \leq C_p\E\left\{ \left( \int_s^{v}|\tilde{Z}^x_{s,u\wedge\lambda_N}|du \right)^{p/2}  \right\} \vspace{0.20cm}
                \\
                \displaystyle \hspace{7.3cm} \leq C_p\E\left\{ \int_s^{v}|\tilde{Z}^x_{s,u\wedge\lambda_N}|^{p/2}du  \right\} \vspace{0.20cm}
                \\
                \displaystyle \hspace{7.3cm} \leq C_p\E\left\{ \int_s^{v}\left[1+|\tilde{Z}^x_{s,u\wedge\lambda_N}|^{p}\right]du  \right\}
            \end{array}
            \end{displaymath}
            and then
            \begin{displaymath}
                \begin{array}{l}
                    \displaystyle \E\left[\sup_{s\leq u\leq v\wedge\lambda_N}|\tilde{Z}^x_{s,u}|^p\right] \leq C_p\left\{x^p + \E\left(\int_s^{v}\left[1+|\tilde{Z}^x_{s,u\wedge\lambda_N}|^p\right]du\right) \right\} \vspace{0.30cm}
                    \\
                    \displaystyle \qquad\qquad\qquad\qquad~~~~ \leq  C_p\left\{x^p + \int_s^{v}\left[1+\E\left\{\sup_{s\leq z\leq u\wedge\lambda_N}|\tilde{Z}^x_{s,z}|^p\right\}\right]du \right\}.
                \end{array}
            \end{displaymath}
            Hence, we obtain for all $v\in [s,T]$
            \[ 1+\E\left[\sup_{s\leq u\leq v\wedge\lambda_N}|\tilde{Z}^x_{s,u}|^p\right] \leq C_p\left\{1+x^p + \int_s^{v}\left[1+\E\left\{\sup_{s\leq z\leq u\wedge\lambda_N}|\tilde{Z}^x_{s,z\wedge\lambda_N}|^p\right\}\right]du \right\} , \]
            and thanks to the Gronwall lemma,
            \[ \E\left[\sup_{s\leq u\leq v\wedge\lambda_N}|\tilde{Z}^x_{s,u}|^p\right] \leq C_p\left(1+x^p\right)e^{C_p(v-s)}. \]
            If we assume that there exists $T_{0}<T$ such that $\eP\left\{ \lim_N\lambda_N< T_0 \right\} > 0$, then for all $v\in [T_0,T]$, we will have
            \[ N^p\,\eP\left\{ \lim_N\lambda_N< T_0 \right\} \leq \E\left[\sup_{s\leq u\leq v\wedge\lambda_N}|\tilde{Z}^x_{s,u}|^p\right] \leq C_p\left(1+x^p\right)e^{C_p(v-s)}. \]
            That is impossible because the left side term tends towards infinity as $N\rightarrow\infty$. It follows that $\lambda_N\xrightarrow[]{a.s} T$ and by the monotone convergence theorem,
            \[ \E\left[\sup_{s\leq u\leq v}|\tilde{Z}^x_{s,u}|^p\right] \leq C_p\left(1+x^p\right)e^{C_p(v-s)}\, ,\, \forall v\in [s,T]. \]
            It then suffices to take $v = T$. The case $0< p< 2$ is obviously solved thanks to the H\"older inequality
            \[ \E\left[\sup_{s\leq u\leq T}|\tilde{Z}^x_{s,u}|^p\right] \leq \left(\E\left[\sup_{s\leq u\leq T}|\tilde{Z}^x_{s,u}|^2\right]\right)^{p/2} \leq C_{p,T}(1+x^p). \]
            The process $(\tilde{Z}^x_{s,u})_{u\in[s,T]}$ is then well defined at any time and satisfies a moment estimate similar to (\ref{momest}). We will now show that this process is almost surely non negative, and then by uniqueness this will imply that it is a modification of the process $(Z^x_{s,u})_{u\in[s,T]}$. Indeed, the following decomposition holds
            \begin{equation}
                \tilde{Z}_{s,t}^x = (\tilde{Z}^x_{s,t})_+ - (\tilde{Z}^x_{s,t})_- , \forall t\in [s,T]
            \end{equation}
            and it follows from the Tanaka formula that
            \begin{equation}\label{tanaka}
                (\tilde{Z}^x_{s,t})_- = -\int_s^t1_{\{\tilde{Z}^x_{s,t}\leq 0\}}d\tilde{Z}^x_{s,t} + \frac{1}{2}L^0_t(\tilde{Z}^x_{s,\cdot}).
            \end{equation}
            For $\varepsilon>0$,
            \[ \int_s^t 1_{\{0< \tilde{Z}^x_{s,u}\leq \varepsilon\}}\frac{d\langle \tilde{Z}^x_{s,\cdot}\rangle_u}{\tilde{Z}^x_{s,u}} = \int_s^t1_{\{0< \tilde{Z}^x_{s,u}\leq \varepsilon\}}\frac{q(\tilde{Z}^x_{s,u},u)}{\tilde{Z}^x_{s,u}}du < \infty \textrm{ p.s,} \]
            since $(x,u)\mapsto q(x,u)/x$ is bounded. Then the local time $L^0_{t}(\tilde{Z}^x_{s,\cdot})$ is zero almost surely (see \cite{revYor}, Ch. IX, Lemma 3.3). Equation (\ref{tanaka}) becomes
            \[ (\tilde{Z}^x_{s,t})_- = -\int_s^t1_{\{\tilde{Z}^x_{s,u}\leq 0\}}h(-\tilde{Z}^x_{s,u},u)du - \int_s^t1_{\{\tilde{Z}^x_{s,u}\leq 0\}}\sqrt{q(-\tilde{Z}^x_{s,u},u)}\,dW_u  \]
            where the second term in the right side of the equality is a true martingale thanks to the moment estimate that we previously stated. It then follows from the hypothesis (A.1) and (A.2) in Assumption \ref{hypapp} that
            \begin{displaymath}
                \begin{array}{l}
                    \displaystyle \E\left[ (\tilde{Z}^x_{s,t})_- \right] = -\E\left(\int_s^t1_{\{\tilde{Z}^x_{s,u}\leq 0\}}h(-\tilde{Z}^x_{s,u},u)du\right) \vspace{0.20cm}
                    \\
                    \displaystyle \qquad\qquad~~\, \leq -\E\left(\int_s^t1_{\{\tilde{Z}^x_{s,u}\leq 0\}}\left[ h(-\tilde{Z}^x_{s,u},u) - h(0,u)\right]du\right) \vspace{0.20cm}
                    \\
                    \displaystyle \qquad\qquad~~\, \leq C\E\left(\int_s^t1_{\{\tilde{Z}^x_{s,u}\leq 0\}}\left| \tilde{Z}^x_{s,u}\right|du\right) \vspace{0.20cm}
                    \\
                    \displaystyle \qquad\qquad~~\, \leq C \int_s^t\E\left[( \tilde{Z}^x_{s,u})_-\right]du
                \end{array}
            \end{displaymath}
            and hence by the Gronwall lemma, $\E\left[ (\tilde{Z}^x_{s,t})_- \right] = 0$ for any $t\in [s,T]$, thus $(\tilde{Z}^x_{s,t})_- = 0$ p.s. We deduce that the stochastic process $(\tilde{Z}^x_{s,t})_{t\in[s,T]}$ is almost surely non negative at any time, and it is possible to remove the absolute values in (\ref{neweds}) to conclude by a uniqueness argument that
            \[ \forall t\in [s,T],\,  \tilde{Z}^x_{s,t} = Z^x_{s,t} \textrm{ p.s.} \]
            The stochastic process $(Z^x_{s,t})_{t\in[s,T]}$ is then almost surely non negative at any time, satisfies the moment estimate (\ref{momest}) and has a zero local time at $0$.
            
            \paragraph{Dependence on initial conditions :} Thanks to the strong existence, it is possible to construct two solutions $(Z^{x_1}_{s_1,v})_{v\in [s_1,T]}$ and $(Z^{x_2}_{s_2,v})_{v\in[s_2,T]}$ on the same filtered probability space and with the same Brownian motion, for all $x_1,x_2\geq 0$ and $0\leq s_1\leq s_2\leq T$. Those solutions are well extended on $[0,T]$ if we set 
            \[ Z^{x_j}_{s_j,v} = x_j, \forall v\in [0,s_j], j=1,2. \]
            Hence, it follows from the Tanaka formula that for all $s_2\leq v\leq T$,
            \begin{equation}\label{eqT}
                \left|Z^{x_1}_{s_1,v}-Z^{x_2}_{s_2,v}\right| = \left| Z^{x_1}_{s_1,s_2} - x_2 \right| + \int_{s_2}^v\textrm{sign}(Z^{x_1}_{s_1,u}-Z^{x_2}_{s_2,u})d(Z^{x_1}_{s_1,u}-Z^{x_2}_{s_2,u}) + l^0_v(Z^{x_1}_{s_1,\cdot}-Z^{x_2}_{s_2,\cdot})
            \end{equation}
            where $l^0_v(Z^{x_1}_{s_1,\cdot}-Z^{x_2}_{s_2,\cdot})$ is the local time of the process $Z^{x_1}_{s_1,\cdot}-Z^{x_2}_{s_2,\cdot}$ at $0$ up to time $v$. Furthermore, thanks to the inequality $|\sqrt{a}-\sqrt{b}|\leq \sqrt{|a-b|}\, ,\forall a,b\geq 0$ again, we have for $\varepsilon>0$, 
            \begin{displaymath}\begin{array}{l}
                \displaystyle \int_{s_2}^v 1_{\{0< Z^{x_1}_{s_1,u}-Z^{x_2}_{s_2,u}\leq \varepsilon\}}\frac{d\left\langle Z^{x_1}_{s_1,\cdot}-Z^{x_2}_{s_2,\cdot}\right\rangle_u}{Z^{x_1}_{s_1,u}-Z^{x_2}_{s_2,u}} = \int_{s_2}^v 1_{\{0< Z^{x_1}_{s_1,u}-Z^{x_2}_{s_2,u}\leq \varepsilon\}}\frac{\left|\sqrt{q(Z^{x_1}_{s_1,u},u)} - \sqrt{q(Z^{x_2}_{s_2,u},u)}\right|^2}{Z^{x_1}_{s_1,u}-Z^{x_2}_{s_2,u}}du \vspace{0.20cm}
                \\
                \displaystyle \hspace{6.5cm} \leq C\int_{s_2}^v 1_{\{0< Z^{x_1}_{s_1,u}-Z^{x_2}_{s_2,u}\leq \varepsilon\}}du < \infty .
            \end{array}
            \end{displaymath}
            We deduce thanks to \cite{revYor}{, Ch. IX, lemma 3.3} that the local time $l^0_{v}(Z^{x_1}_{s_1,\cdot}-Z^{x_2}_{s_2,\cdot})$ is zero almost surely for all $v\in [s_2,T]$, and it follows from (\ref{eqT}) that
            \begin{displaymath}\begin{array}{l}
                \displaystyle \E\left|Z^{x_1}_{s_1,v}-Z^{x_2}_{s_2,v}\right| \leq \E\left(\left|Z^{x_1}_{s_1,s_2}-x_2\right|\right) + \E\int_{s_2}^v\textrm{sign}(Z^{x_1}_{s_1,u}-Z^{x_2}_{s_2,u})\left[h(Z^{x_1}_{s_1,u},u)-h(Z^{x_2}_{s_2,u},u) \right]du \vspace{0.20cm} 
                \\
                \displaystyle \qquad\quad\quad\,~~~\quad\quad \leq \E\left(\left|Z^{x_1}_{s_1,s_2}-x_2\right|\right) + A\int_{s_2}^v\E\left(\left|Z^{x_1}_{s_1,v}-Z^{x_2}_{s_2,v}\right|\right)du
            \end{array}
            \end{displaymath}
            which implies by the Gronwall lemma that
            \begin{equation}\label{ics}
                \E\left|Z^{x_1}_{s_1,v}-Z^{x_2}_{s_2,v}\right| \leq \E\left(\left|Z^{x_1}_{s_1,s_2}-x_2\right|\right)e^{A(v-s_2)}\, , \, \forall v\in [s_2,T].
            \end{equation}
            We are now interested in what happens for $s_1\leq v\leq s_2$. Indeed,
            \begin{displaymath}\begin{array}{l}
                \displaystyle \E\left(\left|Z^{x_1}_{s_1,v}-Z^{x_2}_{s_2,v}\right|\right) = \E\left\{ \left| x_1 + \int_{s_1}^vh(Z^{x_1}_{s_1,u},u)du + \int_{s_1}^v\sqrt{q(Z^{x_1}_{s_1,u},u)}\, dW_u - x_2 \right| \right\} \vspace{0.20cm}
                \\
                \displaystyle \hspace{1cm} \leq \E\left\{ |x_1-x_2| + \int_{s_1}^v|h(Z^{x_1}_{s_1,u},u)|du + \left|\int_{s_1}^v\sqrt{q(Z^{x_1}_{s_1,u},u)}\, dW_u \right| \right\} \vspace{0.20cm}
                \\
                \displaystyle \hspace{1cm} \leq |x_1-x_2| + C\int_{s_1}^{s_2}\left[1+\E(Z^{x_1}_{s_1,u})\right]du + \E\left\{\sup_{s_1\leq v\leq s_2}\left|\int_{s_1}^v\sqrt{q(Z^{x_1}_{s_1,u},u)}\, dW_u \right| \right\}
            \end{array}
            \end{displaymath}
            and thanks to the Doob inequality,
            \begin{displaymath}\begin{array}{l}
                \displaystyle \E\left\{\sup_{s_1\leq v\leq s_2}\left|\int_{s_1}^v\sqrt{q(Z^{x_1}_{s_1,u},u)}\, dW_u \right| \right\} \leq \E\left\{\left|\int_{s_1}^{s_2}q(Z^{x_1}_{s_1,u},u)\, du \right|^{1/2} \right\} \vspace{0.20cm}
                \\
                \displaystyle \hspace{6cm} \leq C\sqrt{\E\left(\int_{s_1}^{s_2}Z^{x_1}_{s_1,u}du\right)} \vspace{0.20cm}
                \\
                \displaystyle \hspace{6cm} \leq C\sqrt{|s_1-s_2|\E\left(\sup_{s_1\leq u\leq T}Z^{x_1}_{s_1,u}\right)}.
            \end{array}
            \end{displaymath}
            Hence, we have for all $v\in [s_1,s_2]$,
            \[ \E\left(\left|Z^{x_1}_{s_1,v}-Z^{x_2}_{s_2,v}\right|\right) \leq |x_1-x_2| + C\int_{s_1}^{s_2}\left[1+\E(Z^{x_1}_{s_1,u})\right]du + C\sqrt{|s_1-s_2|\E\left(\sup_{s_1\leq u\leq T}Z^{x_1}_{s_1,u}\right)} \]
            which implies that
            \begin{equation}\label{ici}
                \E\left(\left|Z^{x_1}_{s_1,v}-Z^{x_2}_{s_2,v}\right|\right) \leq |x_1-x_2| + C_1(1+x_1)|s_1-s_2| + C_2|1+x_1|^{1/2}|s_1-s_2|^{1/2}, \forall v\in [s_1,s_2].
            \end{equation}
            We finally obtain without any order on $s_1,s_2\in [0,T], x_1,x_2\geq 0$ that for all $v\in[s_1\wedge s_2, T]$
            \[ \E\left(\left|Z^{x_1}_{s_1,v}-Z^{x_2}_{s_2,v}\right|\right) \leq C_{T}\left\{|x_1-x_2| + (1+x_1+x_2)|s_1-s_2| + |1+x_1+x_2|^{1/2}|s_1-s_2|^{1/2}\right\} \]
            that is easily extended for $v\in [0,T]$.
        
        \subsection{Proof of Lemma \ref{appenBs}}
            Let us first introduce the decomposition
            \begin{equation}\label{decreg}
                \left| w_{s_1}(x_1)-w_{s_2}(x_2) \right| \leq |w_{s_1}(x_1)-w_{s_2}(x_1)| + |w_{s_2}(x_1)-w_{s_2}(x_2)|
            \end{equation}
            thanks to the triangular inequality. Furthermore, by considering that $s_1\leq s_2$, we have
            \begin{displaymath}
                \begin{array}{l}
                    \displaystyle \left| w_{s_1}(x_1) - w_{s_2}(x_1) \right| \leq \E\left(\left| p(Z^{x_1}_{s_1,t})e^{\int_{s_1}^ta(Z^{x_1}_{s_1,u},u)du} - p(Z^{x_1}_{s_2,t})e^{\int_{s_2}^ta(Z^{x_1}_{s_2,u},u)du} \right|\right) \vspace{0.20cm}
                    \\
                    \displaystyle \qquad\qquad \leq \E\left(\left| p(Z^{x_1}_{s_1,t})-p(Z^{x_1}_{s_2,t})\right|e^{\int_{s_1}^ta(Z^{x_1}_{s_1,u},u)du}\right) \vspace{0.20cm}
                    \\
                    \displaystyle \hspace{4cm}+\, \E\left(|p(Z^{x_1}_{s_1,t})|\left|e^{\int_{s_1}^ta(Z^{x_1}_{s_1,u},u)du} - e^{\int_{s_2}^ta(Z^{x_1}_{s_2,u},u)du}\right|\right) \vspace{0.20cm}
                    \\
                    \displaystyle \qquad\qquad \leq e^{\|a\|_{\infty}(t-s_1)}\|p'\|_{\infty}\E\left(\left| Z^{x_1}_{s_1,t}-Z^{x_1}_{s_2,t}\right|\right) \vspace{0.20cm}
                    \\
                    \displaystyle \hspace{3cm}+\, \|p\|_{\infty}\left\{\E\left(\left|e^{\int_{s_1}^ta(Z^{x_1}_{s_1,u},u)du} - e^{\int_{s_1}^ta(Z^{x_1}_{s_2,u},u)du}\right|\right) \right. \vspace{0.20cm}
                    \\
                    \displaystyle \hspace{6cm} + \left.\E\left(\left|e^{\int_{s_1}^ta(Z^{x_1}_{s_2,u},u)du} - e^{\int_{s_2}^ta(Z^{x_1}_{s_2,u},u)du}\right|\right) \right\} \vspace{0.20cm}
                    \\
                    \displaystyle \qquad\qquad \leq e^{\|a\|_{\infty}(t-s_1)}\bigg[\|p'\|_{\infty}\E\left(\left| Z^{x_1}_{s_1,t}-Z^{x_1}_{s_2,t}\right|\right) \vspace{0.20cm}
                    \\
                    \displaystyle \hspace{1.5cm}+\, \|p\|_{\infty}\bigg\{\E\left(\left|e^{\int_{s_1}^t\left[a(Z^{x_1}_{s_1,u},u)-a(Z^{x_1}_{s_2,u},u) \right]du} - 1\right|\right) + \E\left(\left|e^{-\int_{s_1}^{s_2}a(Z^{x_1}_{s_2,u},u)du} - 1\right|\right) \bigg\}\bigg].
                \end{array}
            \end{displaymath}
            Let us set for all $s_1\leq v\leq t$
            \[ V^{x_1}_{s_1,s_2}(v) = e^{\int_{s_1}^v\left[a(Z^{x_1}_{s_1,u},u) - a(Z^{x_1}_{s_2,u},u)\right]du} - 1, \]
            then
            \[ V^{x_1}_{s_1,s_2}(v) = \int_{s_1}^v\left[a(Z^{x_1}_{s_1,u},u) - a(Z^{x_1}_{s_2,u},u)\right]V^{x}_{s_1,s_2}(u)du + \int_{s_1}^v\left[a(Z^{x_1}_{s_1,u},u) - a(Z^{x_1}_{s_2,u},u)\right]du \]
            that implies that,
            \[ \left|V^{x_1}_{s_1,s_2}(v)\right| \leq 2\|a\|_{\infty}\int_{s_1}^v\left|V^{x}_{s_1,s_2}(u)\right|du + \bar{a}\int_{s_1}^t\left| Z^{x_1}_{s_1,u} - Z^{x_1}_{s_2,u} \right|du \]
            and thanks to the Gronwall lemma,
            \[ \left|V^{x_1}_{s_1,s_2}(v)\right| \leq \bar{a} \, e^{2\|a\|_{\infty}(v-s_1)} \int_{s_1}^t\left| Z^{x_1}_{s_1,u} - Z^{x_1}_{s_2,u} \right|du \, ,\, \forall v\in [s_1,t]. \]
            We show by a similar argument that
            \[ \left|e^{-\int_{s_1}^{s_2}a(Y^x_{s_2,u},u)du} - 1\right| \leq e^{\|a\|_{\infty}|s_1-s_2|}\|a\|_{\infty}|s_1-s_2|. \]
            Hence, one can write for all $x_1\geq 0, s_1\leq s_2 \leq t$
            \begin{displaymath}\begin{array}{l}
                \displaystyle \left| w_{s_1}(x_1) - w_{s_2}(x_1) \right| \leq e^{\|a\|_{\infty}(t-s_1)}\bigg\{ \|p'\|_{\infty}\E\left(\left| Z^{x_1}_{s_1,t}-Z^{x_1}_{s_2,t}\right|\right) \vspace{0.20cm}
                \\
                \displaystyle \hspace{1cm} +\,\|p\|_{\infty}\left[ \bar{a}\, e^{2\|a\|_{\infty}(t-s_1)} \int_{s_1}^t\E\left(\left| Z^{x_1}_{s_1,u} - Z^{x_1}_{s_2,u} \right|\right)du + e^{\|a\|_{\infty}|s_1-s_2|}\|a\|_{\infty}|s_1-s_2|\right] \bigg\}
            \end{array}
            \end{displaymath}
            and thanks to Lemma \ref{appenBf}, we have for all $s_1,s_2\leq t$
            \begin{equation}\label{contrgs}
                \left| w_{s_1}(x_1) - w_{s_2}(x_1) \right| \leq C_T \left(\|p\|_{\infty} + \|p'\|_{\infty}\right)\left\{ (1+x_1)^{1/2}|s_1-s_2|^{1/2} + (1+x_1)|s_1-s_2| \right\}.
            \end{equation}
            In addition, for $x_1,x_2\geq 0$ and $s_2\leq t$,
            \begin{displaymath}
                \begin{array}{l}
                    \displaystyle \left| w_{s_2}(x_1) - w_{s_2}(x_2) \right| \leq \E\left(\left| p(Z^{x_1}_{s_2,t})e^{\int_{s_2}^ta(Z^{x_1}_{s_2,u},u)du} - p(Z^{x_2}_{s_2,t})e^{\int_s^ta(Z^{x_2}_{s_2,u},u)du} \right|\right) \vspace{0.20cm}
                    \\
                    \displaystyle \qquad\qquad\qquad\qquad~ \leq \E\left(\left| p(Z^{x_1}_{s_2,t})-p(Z^{x_2}_{s_2,t})\right|e^{\int_{s_2}^ta(Z^{x_1}_{s_2,u},u)du}\right) \vspace{0.20cm}
                    \\
                    \displaystyle \qquad\qquad\qquad\qquad\quad~ +~ \E\left(|p(Z^{x_2}_{s_2,t})|e^{\int_s^ta(Z^{x_2}_{s_2,u},u)du}\left|e^{\int_{s_2}^t\left[a(Z^{x_1}_{s_2,u},u) - a(Z^{x_2}_{s_2,u},u)\right]du} - 1\right|\right) \vspace{0.20cm}
                    \\
                    \displaystyle \qquad\qquad\qquad\qquad~ \leq e^{\|a\|_{\infty}(t-s_2)}\bigg\{\|p'\|_{\infty}\E\left(\left| Z^{x_1}_{s_2,t}-Z^{x_2}_{s_2,t}\right|\right) \vspace{0.20cm}
                    \\
                    \displaystyle \hspace{6cm} +\, \|p\|_{\infty}\E\left(\left|e^{\int_s^t\left[a(Z^{x_1}_{s_2,u},u) - a(Z^{x_2}_{s_2,u},u)\right]du} - 1\right|\right) \bigg\}.
                \end{array}
            \end{displaymath}
            Let us now set for all $s_2\leq v\leq t$
            \[ U^{x_1,x_2}_{s_2,v} = e^{\int_{s_2}^v\left[a(Y^{x_1}_{s_2,u},u) - a(Z^{x_2}_{s_2,u},u)\right]du} - 1, \]
            then
            \[ U^{x_1,x_2}_{s_2,v} = \int_{s_2}^v\left[ a(Z^{x_1}_{s_2,u},u) - a(Z^{x_2}_{s_2,u},u)  \right]U^{x_1,x_2}_{s_2,u}du + \int_{s_2}^v\left[ a(Z^{x_1}_{s_2,u},u) - a(Z^{x_2}_{s_2,u},u) \right]du \]
            that implies that
            \[ \left|U^{x_1,x_2}_{s_2,v}\right| \leq 2\|a\|_{\infty}\int_{s_2}^v\left|U^{x_1,x_2}_{s_2,u}\right|du + \bar{a}\int_{s_2}^t\left| Z^{x_1}_{s_2,u} - Z^{x_2}_{s_2,u} \right|du \]
            and thanks to the Gronwall lemma,
            \[ \left|U^{x_1,x_2}_{s_2,v}\right| \leq \bar{a}\, e^{2\|a\|_{\infty}(v-s_2)} \int_{s_2}^t\left| Z^{x_1}_{s_2,u} - Z^{x_2}_{s_2,u} \right|du \, ,\, \forall v\in [s_2,t]. \]
            Hence for all $x_1,x_2\geq 0, s_2\leq t$ we have
            \begin{displaymath}
            \begin{array}{l}
                \left| w_{s_2}(x_1) - w_{s_2}(x_2) \right| \leq e^{\|a\|_{\infty}(t-s_2)}\bigg\{\|p'\|_{\infty}\E\left(\left| Z^{x_1}_{s_2,t}-Z^{x_2}_{s_2,t}\right|\right) \vspace{0.10cm}
                \\
                \displaystyle \hspace{5cm} +\, \bar{a}\, e^{2\|a\|_{\infty}(t-s)} \|p\|_{\infty}\int_{s_2}^t\E\left(\left| Z^{x_1}_{s_2,u} - Z^{x_2}_{s_2,u} \right|\right)du \bigg\}
            \end{array}
            \end{displaymath}
            and thanks to Lemma \ref{appenBf},
            \begin{equation}\label{contrgx}
                \left| w_{s_2}(x_1) - w_{s_2}(x_2) \right| \leq C_T\left(\|p\|_{\infty} + \|p'\|_{\infty}\right)|x_1-x_2|, \forall x_1,x_2\geq 0, \forall s_2\leq t.
            \end{equation}
            Finally, the inequality (\ref{contrg}) follows from (\ref{decreg}), (\ref{contrgs}) and (\ref{contrgx}).

    \section{Difference operator and Besov spaces $B^s_{1,\infty}(\R)$}\label{appC}
            For $\alpha\in\, ]0,1[$, we denote by $\eC^{\alpha}_b(\R)$ the H\"older-Zygmund space that is the set of real valued functions on $\R$ such that
            \[ \|f\|_{\eC^{\alpha}_b(\R)} = \|f\|_{\infty} + \sup_{x,y\in\R,x\neq y}\frac{|f(x)-f(y)|}{|x-y|^{\alpha}} < \infty. \]
            We also introduce the difference operator defined for all $f:\R\rightarrow\R$, for given $m\geq 1$ integer and $h\in\R$ as
            \[ \forall x\in\R, \, \Delta^1_hf(x) = f(x+h) - f(x)\, ;\, \Delta^m_hf(x) = \Delta^1_h(\Delta^{m-1}_hf)(x) , \textrm{ pour }m>1. \]
            that is, by a recurrence argument,
            \begin{equation}\label{deltamh}
                \Delta^m_hf(x) = \sum_{j=0}^m(-1)^{m-j} \begin{pmatrix}
                                                            m \\
                                                            j
                                                        \end{pmatrix}f(x+jh)
            \end{equation}
            with the following properties.
            \begin{lemma}\label{propD}
                Let $m\geq 1$ be an integer, $\alpha\in\, ]0,1[$, $h\in\R$ and $f : \R\rightarrow\R$, then :
                \begin{enumerate}
                    \item If $f\in\eC^{\alpha}_b(\R)$, 
                    \begin{equation}
                        \|\Delta^m_hf\|_{\infty} \leq C_m |h|^{\alpha}\|f\|_{\eC^{\alpha}_b(\R)}.
                    \end{equation}
                    
                    \item If $f\in\eC^{\alpha}_b(\R)$, 
                    \begin{equation}
                        |\Delta^m_hf(x) - \Delta^m_hf(y)| \leq C_m \|f\|_{\eC^{\alpha}_b(\R)} |x-y|^{\alpha} , \forall x,y\in\R.
                    \end{equation}
                    
                    \item If $f\in\iC^m(\R,\R)$,
                    \begin{equation}
                        \|\Delta^m_hf\|_{\eL^1} \leq C_m |h|^m\|\partial_x^mf\|_{\eL^1}
                    \end{equation}
                    
                    \item If $f\in\iC_b(\R,\R)$ and $g : \R\rightarrow\R$ is an integrable function, then for all $a\in\R$,
                    \begin{equation}
                        \int_{\R}\Delta^m_hf(x + a) g(x)dx = \int_{\R}f(x + a)\Delta^m_{-h}g(x)dx
                    \end{equation}
                \end{enumerate}
            \end{lemma}
            \begin{proof}
                For $0<\alpha<1$, $h\in\R$ and an integer $m\geq 1$,
                \begin{enumerate}
                    \item If $f\in\eC^{\alpha}_b(\R)$, we proceed by a recurrence argument on $m\geq 1$. Indeed,
                    \[ \|\Delta^1_hf\|_{\infty} = \sup_{x\in\R}\left| f(x+h)-f(x) \right| \leq |h|^{\alpha}\|f\|_{\eC^{\alpha}_b(\R)} \]
                    and if we assume that the property holds at range $m$, we get
                    \[ \|\Delta^{m+1}_hf\|_{\infty} = \sup_{x\in\R}\left| \Delta^m_hf(x+h)-\Delta^m_hf(x) \right| \leq 2C_m|h|^{\alpha}\|f\|_{\eC^{\alpha}_b(\R)} \]
                    that corresponds to the property with $C_{m+1} = 2C_m$. We deduce that it holds for each integer $m\geq 1$.
                    
                    \item If $f\in\eC^{\alpha}_b(\R)$, then thanks to (\ref{deltamh}) we get for all $x,y\in\R$,
                    \begin{displaymath}
                        |\Delta^m_hf(x) - \Delta^m_hf(y)| \leq \sum_{j=0}^m \begin{pmatrix}
                                                                                m \\
                                                                                j
                                                                            \end{pmatrix}|f(x+jh) - f(y+jh)| \leq 2^m\|f\|_{\eC^{\alpha}_b(\R)}|x-y|^{\alpha}.
                    \end{displaymath}
                    
                    \item We proceed by a recurrence argument to show firstly that 
                    \[ \mathbf{(P) : }~~\forall m\geq 1, ~\Delta^m_hf(x) = h^m\int_{0}^mH_m(t)\partial_x^mf(x+th)dt, \textrm{ if }f\in\iC^m(\R,\R) \]
                    where $H_m$ is bounded and does not depend on $f$. Indeed for $f\in\iC^1(\R,\R)$, then
                    \[ \Delta^{m+1}_hf(x) = f(x+h)-f(x) = \int_x^{x+h}f'(t)dt = h\int_0^1f'(x+th)dt  \]
                    and hence $H_1(t) = 1$.
                    We now assume that the property holds at range $m$, then if $f\in\iC^{m+1}(\R,\R)$,
                    \begin{displaymath}
                        \begin{array}{l}
                            \displaystyle \Delta^{m+1}_hf(x) = \Delta_h^{m}f(x+h) - \Delta^m_hf(x) \vspace{0.20cm}
                            \\
                            \displaystyle \qquad\qquad~~\, = h^m\int_0^mH_m(t)\left[ \partial_x^mf(x+(1+t)h) - \partial_x^mf(x+th) \right]dt \vspace{0.20cm}
                            \\
                            \displaystyle \qquad\qquad~~\, = h^m\int_0^mH_m(t)\int_{x+th}^{x+(1+t)h} \partial_x^{m+1}f(u)dudt \vspace{0.20cm}
                            \\
                            \displaystyle \qquad\qquad~~\, = h^{m+1}\int_0^mH_m(t)\int_{t}^{1+t} \partial_x^{m+1}f(x+uh)dudt \vspace{0.20cm}
                            \\
                            \displaystyle \qquad\qquad~~\, = h^{m+1}\int_0^mH_m(t)\int_{0}^{m+1} \partial_x^{m+1}f(x+uh)1_{\{u\in[t,1+t]\}}dudt \vspace{0.20cm}
                            \\
                            \displaystyle \qquad\qquad~~\, = h^{m+1}\int_0^{m+1}\underbrace{\left(\int_0^mH_m(t)1_{\{u\in[t,1+t]\}}dt\right)}_{=\, H_{m+1}(u)} \partial_x^{m+1}f(x+uh)du \vspace{0.20cm}
                        \end{array}
                    \end{displaymath}
                    where the function $H_{m+1}$ is bounded and does not depend on $f$. We deduce that the property $\mathbf{(P)}$ holds, and then for all $f\in\iC^m(\R,\R)$,
                    \[ \|\Delta^m_hf\|_{\eL^1} \leq |h|^m\|H_m\|_{\infty}\int_{\R}\int_0^m|\partial_x^mf(x+th)|dtdx = m\|H_m\|_{\infty}|h|^m\|\partial_x^mf\|_{\eL^1}. \]
                    
                    \item Let $f\in\iC_b(\R,\R)$, and $g:\R\rightarrow\R$ an integrable function and $a\in\R$, then by changing variables in the integral
                    \begin{displaymath}\begin{array}{l}
                        \displaystyle \int_{\R}\Delta^m_hf(x + a) g(x)dx = \sum_{j=0}^m (-1)^{m-j} \begin{pmatrix}
                                                                                m \\
                                                                                j
                                                                            \end{pmatrix}\int_{\R}f(x+a+jh)g(x)dx \vspace{0.20cm}
                        \\
                        \displaystyle \qquad\qquad\qquad\qquad\quad~~\, = \sum_{j=0}^m (-1)^{m-j} \begin{pmatrix}
                                                                                m \\
                                                                                j
                                                                            \end{pmatrix}\int_{\R}f(x+a)g(x - jh)dx \vspace{0.20cm}
                        \\
                        \displaystyle \qquad\qquad\qquad\qquad\quad~~\, = \int_{\R}f(x+a)\Delta^m_{-h}g(x)dx
                    \end{array}
                    \end{displaymath}
                \end{enumerate}
            \end{proof}
            \noindent The difference operator allows to define the Besov spaces $\iB^s_{1,\infty}(\R)$ as the set of function $f:\R\rightarrow\R$ that satisfy for $m>s$
            \[ \|f\|_{\iB^s_{1,\infty}} := \|f\|_{\eL^1(\R)} + \sup_{|h|\leq 1}|h|^{-s}\|\Delta^m_hf\|_{\eL^1(\R)} < \infty \]
            (see Triebel \cite{trieb83} Theorem 2.5.12, or \cite{trieb92} Theorem 2.6.1). The following result follows
            \begin{lemma}\label{appendBes}
                Let $(E,\iF_E,\mu)$ be a measured space where $\mu$ is a positive measure. Let $(x,y)\mapsto f(x,y)$ an integrable function on $\R\times E$, then if we denote by 
                \[ \Phi_f(x) = \int_Ef(x,y)\mu(dy), \forall x\geq 0 \]
                we have for all $s>0$,
                \begin{equation}
                    \left\| \Phi_f \right\|_{\iB^s_{1,\infty}(\R)} \leq \int_E\|f(\cdot,y)\|_{\iB^s_{1,\infty}(\R)}\mu(dy).
                \end{equation}
            \end{lemma}
            \begin{proof}
                Let $s>0$, then it follows from Fubini theorem that
                \begin{displaymath}
                    \|\Phi\|_{\eL^1(\R)} = \int_{\R}\left| \int_E f(x,y) \mu(dy) \right|dx \leq \int_E\|f(\cdot,y)\|_{\eL^1(\R)}\mu(dy).
                \end{displaymath}
                Similarly, for any $m>s$ and $h\in[-1,1]$,
                \begin{displaymath}
                    \|\Delta^m_h\Phi\|_{\eL^1(\R)} = \int_{\R}\left| \int_E \big(\Delta^m_h f(\cdot,y)\big)(x) \mu(dy) \right|dx \leq \int_E\|\Delta^m_hf(\cdot,y)\|_{\eL^1(\R)}\mu(dy).
                \end{displaymath}
                The result directly follows by taking the supremum in $h$ and summing the above inequalities.
            \end{proof}
            Furthermore, the above definition of Besov spaces allows a sufficient condition for the existence of a density for random variables given by the following result
            \begin{lemma}[\cite{romito}, Lemma A.1]\label{condRom}
                Let $X$ be a real valued random variable. If there are an integer $m\geq 1$, a real number $\theta>0$, a real $\alpha>0$, with $\alpha<\theta<m$, and a constant $K>0$ such that for every $\phi\in\eC^{\alpha}_b(\R)$ and $h\in\R$ with $|h|\leq 1$,
                \[ \E\left[ \Delta^m_h\phi(X) \right] \leq K|h|^{\theta}\|\phi\|_{\eC^{\alpha}_b(\R)}, \]
                then $X$ has a density $f_X$ with respect to the Lebesgue measure on $\R$. Moreover $f_X\in\iB^{\theta-\alpha}_{1,\infty}(\R)$ and
                \[ \|f\|_{\iB^{\theta-\alpha}_{1,\infty}} \approxleq 1+K. \]
            \end{lemma}
            \noindent The consequence that follows gives a sufficient condition that will is used in the paper
            \begin{lemma}\label{lemcond}
                Let $X$ be a real valued random variable. For a given non negative real valued function $\sigma(x)$ such that $\sigma(X)\in \eL^1$ with $\E\big[ \sigma(X) \big] \neq 0$, if there exists an integer $m\geq 1$, a real number $\theta>0$, a real $\alpha>0$ with $\alpha<\theta<m$, and a constant $K>0$ such that for every $\phi\in\eC^{\alpha}_b(\R)$ and $h\in\R$ with $|h|\leq 1$,
                \begin{equation}\label{hypmes}
                    \E\left[ \sigma(X)\Delta^m_h\phi(X) \right] \leq K |h|^{\theta}\|\phi\|_{\eC^{\alpha}_b(\R)},
                \end{equation}
                then $X$ admits a density $f_X$ on $\{ x : \sigma(x)\neq 0 \}$. In addition, if we denote by $f_X$ its density on this set, then the function $x\mapsto\sigma(x)f_X(x)$ is in the Besov space $\iB^{\theta-\alpha}_{1,\infty}(\R)$ and satisfies the bound
                \begin{equation}\label{boundbesov}
                    \|\sigma f_X\|_{\iB^{\theta-\alpha}_{1,\infty}(\R)} \approxleq K + \E\left[ \sigma(X) \right] 
                \end{equation}
            \end{lemma}
            \begin{proof}
                Let us denote by $\mu(dx)$ the law of $X$ and set
                \[ \nu(dx) = \sigma(x)\mu(dx) \]
                that is a non negative finite measure by assumption. Then it follows from (\ref{hypmes}) that
                \[ \E\left[ \sigma(X)\Delta^m_h\phi(X) \right] = \|\nu\|_{\eL^1}\int_0^{\infty}\Delta^m_h\phi(x)\frac{\nu(dx)}{\|\nu\|_{\eL^1}} = \|\nu\|_{\eL^1} \E\left[ \Delta^m_h\phi(Y) \right] \]
                where $Y$ is a real valued random variable of law $\nu(dy)/\|\nu\|_{\eL^1}$. It follows from (\ref{hypmes}) that 
                \[ \E\left[ \Delta^m_h\phi(Y) \right] \leq \frac{K }{\|\nu\|_{\eL^1}}|h|^{\theta}\|\phi\|_{\eC_b^{\alpha}(\R)} \]
                that implies thanks to Lemma \ref{condRom} that the probability measure $\nu(dy)/\|\nu\|_{\eL^1}$ admits a density $g_X\in\iB^{\theta-\alpha}_{1,\infty}(\R)$ with respect to the Lebesgue measure on $\R$ that satisies the bound 
                \[ \|g_X\|_{\iB^{\theta-\alpha}_{1,\infty}(\R)} \approxleq 1+\frac{K}{\|\nu\|_{\eL^1}}.  \]
                Then $\mu(dx)$ admits a density $f_X$ on the set $\{x : \sigma(x)\neq 0\}$ that is defined by
                \[ f_X(x) = \frac{\|\nu\|_{\eL^1}}{\sigma(x)}g_X(x) \]
                and the bound (\ref{boundbesov}) follows.
            \end{proof}
            
    \section{$\eL^1$ estimates of derivatives of the Gaussian density}\label{appD}
            We are interested in the density of a centered Gaussian random variable with variance $\sigma^2$, that is
            \[ g_{\sigma}(x) = \frac{1}{\sigma\sqrt{2\pi}}\exp\left( -\frac{x^2}{2\sigma^2} \right) , \forall x\in\R. \]
            We have the following result
            \begin{lemma}\label{dernorm}
                For each integer $m\geq 1$, 
                \[ \|\partial^m_xg_{\sigma}\|_{\eL^1} \leq \frac{C_m}{\sigma^m} \]
                where the constant $C_m$ does not depend on $\sigma$.
            \end{lemma}
            \begin{proof}
                Because of its exponential factor, let us set
                \[ \partial_x^mg_{\sigma}(x) = \frac{P^{\sigma}_m(x)}{\sigma^m}\frac{1}{\sigma\sqrt{2\pi}}\exp\left( -\frac{x^2}{2\sigma^2} \right), \forall x\in\R, \forall m\geq 1. \]
                Then $P^{\sigma}_{1}(x) = - x/\sigma$ and
                \[ P^{\sigma}_{m+1} = \sigma\partial_xP_m^{\sigma} - \frac{x}{\sigma}P_m^{\sigma} \]
                that implies that each $P^{\sigma}_m$ is a polynomial, thus
                \[ P_m^{\sigma}(x) = \sum_{k=0}^mA^m_k(\sigma)x^k \]
                with
                \begin{displaymath}\begin{array}{l}
                    \displaystyle A^{m+1}_0(\sigma) = \sigma A^{m}_1(\sigma), \vspace{0.20cm} 
                    \\
                    \displaystyle A^{m+1}_k(\sigma) = \sigma(k+1)A^m_{k+1}(\sigma) - \frac{1}{\sigma}A^m_{k-1}(\sigma) ,\, k=1,...,m-1 \vspace{0.20cm}
                    \\
                    \displaystyle A^{m+1}_m(\sigma) = -\frac{1}{\sigma}A^{m}_{m-1}(\sigma) \vspace{0.20cm}
                    \\
                    \displaystyle A^{m+1}_{m+1}(\sigma) = -\frac{1}{\sigma}A^{m}_m(\sigma)
                \end{array}
                \end{displaymath}
                We easily verify by a recurrence argument on $m\geq 1$ that
                \[ \forall m\geq 1 , \, |A^m_k| \leq \frac{C_m}{\sigma^k} \textrm{ pour }k=1,...,m \]
                and then
                \begin{displaymath}
                    \|\partial^m_xg_{\sigma}\|_{\eL^1} = \frac{1}{\sigma^m}\int_{\R}\left|P_m^{\sigma}(x)\right|g_{\sigma}(x)dx \leq \frac{1}{\sigma^m}\sum_{k=0}^m\left| A^m_k(\sigma) \right|\int_{\R}\left|x\right|^kg_{\sigma}(x)dx \leq \frac{C_m}{\sigma^m}
                \end{displaymath}
            \end{proof}
        
\end{document}